\documentclass[12pt]{article}

\usepackage{comment}

\usepackage[T1]{fontenc}
\usepackage[utf8]{inputenc}
\usepackage{textcomp}
\usepackage{lmodern}
\usepackage{amsmath,amssymb,amsfonts,amsthm}
\usepackage{bm}

\usepackage{enumerate}
\usepackage{enumitem}

\usepackage[dvipsnames,svgnames,x11names]{xcolor}
\usepackage{hyperref}

\IfFileExists{microtype.sty}{\usepackage{microtype}}{}
\makeatletter
\@ifundefined{KOMAClassName}{%
  \IfFileExists{parskip.sty}{%
    \usepackage{parskip}
  }{%
    \setlength{\parindent}{0pt}
    \setlength{\parskip}{6pt plus 2pt minus 1pt}%
  }
}{%
  \KOMAoptions{parskip=half}%
}
\makeatother
\makeatletter
\ifx\paragraph\undefined\else
  \let\oldparagraph\paragraph
  \renewcommand{\paragraph}{
    \@ifstar
      \xxxParagraphStar
      \xxxParagraphNoStar
  }
  \newcommand{\xxxParagraphStar}[1]{\oldparagraph*{#1}\mbox{}}
  \newcommand{\xxxParagraphNoStar}[1]{\oldparagraph{#1}\mbox{}}
\fi
\ifx\subparagraph\undefined\else
  \let\oldsubparagraph\subparagraph
  \renewcommand{\subparagraph}{
    \@ifstar
      \xxxSubParagraphStar
      \xxxSubParagraphNoStar
  }
  \newcommand{\xxxSubParagraphStar}[1]{\oldsubparagraph*{#1}\mbox{}}
  \newcommand{\xxxSubParagraphNoStar}[1]{\oldsubparagraph{#1}\mbox{}}
\fi
\makeatother

\usepackage{longtable,booktabs,array}
\usepackage{calc} 
\usepackage{etoolbox}
\makeatletter
\patchcmd\longtable{\par}{\if@noskipsec\mbox{}\fi\par}{}{}
\makeatother
\IfFileExists{footnotehyper.sty}{\usepackage{footnotehyper}}{\usepackage{footnote}}
\makesavenoteenv{longtable}
\usepackage{graphicx}
\makeatletter
\def\maxwidth{\ifdim\Gin@nat@width>\linewidth\linewidth\else\Gin@nat@width\fi}
\def\maxheight{\ifdim\Gin@nat@height>\textheight\textheight\else\Gin@nat@height\fi}
\makeatother
\setkeys{Gin}{width=\maxwidth,height=\maxheight,keepaspectratio}
\makeatletter
\def\fps@figure{htbp}
\makeatother

\makeatletter
\@ifpackageloaded{caption}{}{\usepackage{caption}}
\AtBeginDocument{%
\ifdefined\contentsname
  \renewcommand*\contentsname{Table of contents}
\else
  \newcommand\contentsname{Table of contents}
\fi
\ifdefined\listfigurename
  \renewcommand*\listfigurename{List of Figures}
\else
  \newcommand\listfigurename{List of Figures}
\fi
\ifdefined\listtablename
  \renewcommand*\listtablename{List of Tables}
\else
  \newcommand\listtablename{List of Tables}
\fi
\ifdefined\figurename
  \renewcommand*\figurename{Figure}
\else
  \newcommand\figurename{Figure}
\fi
\ifdefined\tablename
  \renewcommand*\tablename{Table}
\else
  \newcommand\tablename{Table}
\fi
}
\@ifpackageloaded{float}{}{\usepackage{float}}
\floatstyle{ruled}
\@ifundefined{c@chapter}{\newfloat{codelisting}{h}{lop}}{\newfloat{codelisting}{h}{lop}[chapter]}
\floatname{codelisting}{Listing}

\makeatother
\makeatletter
\@ifpackageloaded{caption}{}{\usepackage{caption}}
\@ifpackageloaded{subcaption}{}{\usepackage{subcaption}}
\makeatother

\usepackage[]{natbib}
\usepackage{bookmark}

\IfFileExists{xurl.sty}{\usepackage{xurl}}{} 
\hypersetup{
  pdftitle={Title},
  pdfauthor={Author 1; Author 2},
  pdfkeywords={3 to 6 keywords, that do not appear in the title},
  colorlinks=true,
  linkcolor={blue},
  filecolor={Maroon},
  citecolor={Blue},
  urlcolor={Blue},
  pdfcreator={LaTeX via pandoc}}

\newcommand{\anon}{1}

\theoremstyle{plain}
\newtheorem{theorem}{Theorem}
\newtheorem{lemma}{Lemma}

\newtheorem{corollary}{Corollary}

\theoremstyle{definition}
\newtheorem{definition}{Definition}

\newtheorem{example}{Example}
\newtheorem{assumption}{Assumption}

\newtheorem{remark}{Remark}
\numberwithin{equation}{section}

\definecolor{Red}{RGB}{225,0,0}

\definecolor{Blue}{RGB}{0,0,255}

\definecolor{Cyan}{RGB}{0,180,255}

\definecolor{Green}{RGB}{0,160,0}

\definecolor{Alert}{RGB}{255,122,0}

\definecolor{NavyBlue}{RGB}{0,100,175}

\definecolor{NavyRed}{RGB}{125,0,0}

\newcommand{\eps}{\varepsilon}

\newcommand{\wt}{\widetilde}
\newcommand{\wh}{\widehat}

\newcommand{\argmin}{\mathop{\rm arg\min}}
\newcommand{\argmax}{\mathop{\rm arg\max}}

\newcommand{\Tr}{\operatorname{Tr}}

\newcommand{\rank}{\operatorname{rk}}
\newcommand{\Sp}{\operatorname{Sp}}

\newcommand{\alg}{\mA}

\newcommand{\diag}{\operatorname{diag}}
\newcommand{\dof}{\operatorname{dof}}

\newcommand{\fss}{\operatorname{FSS}}

\newcommand{\ls}{\operatorname{LS}}

\newcommand{\pcr}{\operatorname{PCR}}
\newcommand{\pls}{\operatorname{PLS}}
\newcommand{\spa}{\operatorname{span}}
\newcommand{\spr}{\operatorname{SPR}}

\newcommand{\B}{\mathbb{B}}
\newcommand{\E}{\mathbb{E}}

\renewcommand{\P}{\mathbb{P}}
\newcommand{\R}{\mathbb{R}}

\newcommand{\mA}{\mathcal{A}}
\newcommand{\mB}{\mathcal{B}}
\newcommand{\mC}{\mathcal{C}}

\newcommand{\mE}{\mathcal{E}}

\newcommand{\mK}{\mathcal{K}}

\newcommand{\mN}{\mathcal{N}}

\newcommand{\mR}{\mathcal{R}}

\newcommand{\mV}{\mathcal{V}}

\newcommand{\bzero}{\mathbf{0}}

\newcommand{\bA}{\mathbf{A}}
\newcommand{\bb}{\mathbf{b}}
\newcommand{\bB}{\mathbf{B}}

\newcommand{\be}{\mathbf{e}}
\newcommand{\bE}{\mathbf{E}}

\newcommand{\bG}{\mathbf{G}}
\newcommand{\bH}{\mathbf{H}}
\newcommand{\bI}{\mathbf{I}}

\newcommand{\bK}{\mathbf{K}}

\newcommand{\bP}{\mathbf{P}}
\newcommand{\bq}{\mathbf{q}}
\newcommand{\bQ}{\mathbf{Q}}

\newcommand{\bu}{\mathbf{u}}
\newcommand{\bU}{\mathbf{U}}
\newcommand{\bv}{\mathbf{v}}
\newcommand{\bV}{\mathbf{V}}
\newcommand{\bw}{\mathbf{w}}

\newcommand{\bx}{\mathbf{x}}
\newcommand{\bX}{\mathbf{X}}
\newcommand{\by}{\mathbf{y}}

\newcommand{\bZ}{\mathbf{Z}}

\newcommand{\balpha}{\bm{\alpha}}
\newcommand{\bbeta}{\bm{\beta}}

\newcommand{\bDelta}{\bm{\Delta}}

\newcommand{\bxi}{\bm{\xi}}
\newcommand{\bzeta}{\bm{\zeta}}

\newcommand{\bsigma}{\boldsymbol{\sigma}}
\newcommand{\bSigma}{\boldsymbol{\Sigma}}
\newcommand{\btheta}{\bm{\theta}}

\usepackage{chngcntr}
\counterwithout{equation}{section} 

\begin{document}

\def\spacingset#1{\renewcommand{\baselinestretch}%
{#1}\small\normalsize} \spacingset{1}


\if1\anon
{
  \title{\bf On identification \\ in ill-posed linear regression}
  \author{Gianluca Finocchio\hspace{.2cm}\\
    Department of Statistics and Operations Research, Universit\"at Wien\\
    and \\
    Tatyana Krivobokova \\
    Department of Statistics and Operations Research, Universit\"at Wien}
  \maketitle
} \fi

\if0\anon
{
  \bigskip
  \bigskip
  \bigskip
  \begin{center}
    {\LARGE\bf Title}
\end{center}
  \medskip
} \fi

\bigskip
\begin{abstract}
\noindent A novel framework is introduced to formalize identifiability in well-specified but ill-posed linear regression models. The framework is distribution-free and accommodates highly correlated features that may or may not relate to the response, reflecting typical real-data structures. First, the identifiable parameter is defined as the least-squares solution obtained by regressing the response on the largest subset of relevant features whose condition number does not exceed a specified threshold, and the relative risk incurred by using this predictor instead of the optimal one is quantified. Second, simple, verifiable conditions are provided under which a broad class of linear dimensionality-reduction algorithms can estimate identifiable parameters; algorithms satisfying these conditions are termed statistically interpretable. Third, sharp high-probability error bounds are derived for these algorithms, with rates explicitly reflecting the degree of ill-posedness. With heavy-tailed features and sufficiently low effective rank, these algorithms achieve convergence rates that improve upon both the minimax least-squares rate and lower bounds for sparse estimation under sub-Gaussian features. Results are illustrated via simulations and a real-data application, in which effective rank grows logarithmically with dimension. The framework may extend to algorithms modeling nonlinear response-feature dependence.
\end{abstract}

\noindent%
\vfill

\newpage
\spacingset{1.8} 


\section{Introduction}\label{sec:intro}
A linear regression model $y=\bx^\top\bbeta+\eps$ involving a response variable $y\in\R$ and a vector of features $\bx=(x_{1},\ldots,x_{p})^\top\in\R^p$, where the conditional expectation of the response is $\E(y|\bx)=\bx^\top\bbeta$ with some unknown vector of coefficients $\bbeta=(\beta_{1},\ldots,\beta_{p})^\top\in\R^p$, is the most classical and well-studied statistical setting. If this model is well-specified and $\bx$ has no highly correlated entries, so that $\bbeta$ is identifiable, then regression coefficients ${\beta}_i$ are interpretable in the sense that they describe the contribution of the corresponding $x_i$ to the response, while keeping other entries of $\bx$ fixed, that is, $\beta_i=\partial\mathbb{E}(y|\bx)/\partial x_i$.   \\
It turns out that in many modern applications, practitioners face settings in which $\bx$ contains both highly correlated features and features irrelevant to the response. A particularly important example of such an application is genome-wide association studies (GWAS), as discussed in \cite{Uffelmann2021gen}. Here, the features $\bx$ represent the entire genome, and the goal is to identify genotypes associated with the phenotypes of specific diseases. It is natural to expect that not all the genes in the pool are relevant to the specific response under investigation, while those that are relevant might be highly correlated. Another example is the data on protein dynamics reported in \citet{krivobokova2012}, which studied atomic configurations that drive a specific protein's biological function. One instance involved the gated water channel of {\it Pichia pastoris yeast}, in which the channel's opening diameter and the Euclidean coordinates of its atoms were analyzed. It is reasonable to assume that only atoms in the vicinity of the channel contribute to its opening and that, due to spatial proximity, the trajectories of these atoms are highly correlated. \\
In both examples, the primary objective is to gain insights into the underlying system and to identify (linear combinations of) features that are essential for explaining the response, despite the presence of both highly correlated and irrelevant features. To pinpoint the problem of parameter identification and interpretation in linear regression in the presence of highly correlated and irrelevant features, let us consider a toy example.

\begin{example}[Toy Model]\label{ex:toy}
Let us consider a nontrivial toy problem in which a linear model is well-specified but ill-posed. With $p=3$, we let the zero-mean features be $\bx=(x_{1},x_{2},x_{3})^\top$. For some $\beta_{1},\beta_{2}\in\R$ and $\bbeta=(\beta_{1},\beta_{2},0)^\top$, we let the response be $y=\bx^\top\bbeta + \epsilon$, with $\E(\epsilon|\bx)=0$ and $\E(\epsilon^2|\bx)<\infty$. For some $0\leq \rho\leq 1$, we assume that
\begin{align*} 
	\bSigma(\rho) = \E(\bx\bx^\top) = \left(
	\begin{matrix} 
		1 & \rho & 0 \\
		\rho & 1 & 0 \\
		0 & 0 & 2
	\end{matrix}\right), \quad
	\bsigma(\rho) = \E(\bx y) =\left(
	\begin{matrix} 
		\beta_{1} + \rho\beta_{2} \\ 
		\rho\beta_{1} + \beta_{2} \\
		0
	\end{matrix}\right),
\end{align*}
so that the feature $x_{3}$ is uncorrelated with the response and the other features, whereas the features $x_{1}$ and $x_{2}$ are correlated with the response and among themselves. Moreover, the covariance of $x_3$ dominates that of $x_1$ and $x_2$. For any $\rho\in[0,1]$, the unique minimum-$L^2$-norm solution of the least-squares problem is given by $\bbeta_{\ls}(\rho) = \bSigma^\dagger(\rho) \bsigma(\rho)$, where $\bSigma^\dagger(\rho)$ is a Moore-Penrose inverse of $\bSigma(\rho)$, or just an inverse, if the matrix is invertible.\\ 
Let us consider the case $\rho=1$. Obviously, with $x_1$ and $x_2$ perfectly correlated, the coefficients $\bbeta$ lose their intended meaning, since $\bx^\top\bbeta$ can be generated by infinitely many coefficient vectors. That is, $\bbeta$ is not identifiable and hence is no longer interpretable. The unique minimum $L^2$-norm least-squares solution is given by $\bbeta_{\ls}(1)=\{(\beta_1+\beta_2)/2,(\beta_1+\beta_2)/2,0\}^\top$. Hence, each of perfectly correlated $x_1$ and $x_2$ gets the same weight $(\beta_1+\beta_2)/2$, which is meaningful. 
At the same time, if $\rho=1$, then $\bx$ can be projected onto a one-dimensional space. If we choose the projection along the second eigenvector of $\bSigma(1)$ given by $\bold{u}_2=(1/\sqrt{2},1/\sqrt{2},0)^\top$, we get that $\bx^\top\bbeta=\bx^\top\bold{u}_2\bold{u}_2^\top\bbeta=\bx^\top \bbeta_{\ls}(1)$, where $\bold{u}_2\bold{u}_2^\top\bbeta=\bbeta_{\ls}(1)$, a coefficient vector constrained to a specific one-dimension space, is perfectly identifiable. \\
Note that choosing the first eigenvector of $\bSigma(1)$, given by $\bold{u}_1=(0,0,1)^\top$, would result in $\bold{u}_1\bold{u}_1^\top\bbeta=(0,0,0)^\top$, which is meaningless. Hence, the projection should be carefully chosen and not along directions irrelevant to the response. Of course, in practice, it is unknown which features are irrelevant, and it is rather of interest to identify them. \\
What can be said about the identifiability of $\bbeta$ if $\bSigma(\rho)$ becomes ill-posed, that is, if $\rho\rightarrow1$? Formally, $\bbeta_{\ls}(\rho)=(\beta_1,\beta_2,0)^\top$ for any $\rho\in[0,1)$ and is identifiable. However, if $\rho$ is very close to $1$, reflecting an extremely high correlation between $x_1$ and $x_2$, it seems reasonable to use identifiable and hence interpretable $\bbeta_{\ls}(1)=\bold{u}_2\bold{u}_2^\top\bbeta=\{(\beta_1+\beta_2)/2,(\beta_1+\beta_2)/2,0\}^\top$ instead of $\bbeta_{\ls}(\rho)$, if the corresponding error is negligible. At the first glance, the error made by replacing $\bbeta_{\ls}(\rho)$ by $\bbeta_{\ls}(1)$ is large, since $\bbeta_{\ls}(\rho)\nrightarrow\bbeta_{\ls}(1)$ as $\rho\rightarrow1$.
However, $\E\{\bx^\top\bbeta_{\ls}(\rho)-\bx^\top\bbeta_{\ls}(1)\}^2=(1-\rho)(\beta_1-\beta_2)^2/2\rightarrow0$ as $\rho\rightarrow1$. That is, replacing  $\bbeta_{\ls}(\rho)$ by $\bbeta_{\ls}(1)$ for $\rho\rightarrow 1$, we get identifiable parameters that give the same weight to highly correlated features $x_1$ and $x_2$, thereby making a negligible error in approximating $\bx^\top\bbeta$.
\end{example}
Altogether, obtaining identifiable parameters (up to an error) in an ill-posed linear regression reduces to finding an appropriate projection onto a lower-dimensional subspace of $\bx$, such that the resulting error in approximating $\bx^\top\bbeta$ is acceptably small. Importantly, the direction of the projection should not be along the features irrelevant for the response, which are, in general, unknown in practice.\\
Surprisingly, the problem of identifiability and hence interpretability in ill-posed linear regression has, to the best of our knowledge, not been systematically addressed within a sufficiently general framework. Since in practice regression coefficients $\bbeta$ are estimated using various dimensionality reduction algorithms, the existing literature merely provides a set of conditions under which a given algorithm produces a consistent estimator. We briefly overview a few of the most well-known algorithms here.\\
The first class of algorithms projects the features along directions chosen independently of the response, thereby achieving unsupervised reduction. This includes Principal Component Regression (PCR) devised by \cite{hotelling1933anal} and all similar projection algorithms discussed by \cite{bing2021pre}, which only use the features' covariance matrix. Under the regularity conditions summarized by \cite{fan2023}, these methods are shown to achieve both consistent model selection \citep[see][]{stock2002} and estimation \citep[see][]{bing2021pre}. Despite the ability of these algorithms to identify well-posed clusters of features, there is no logical reason for the principal components to contain any information about the response; see Section~3 in \cite{cox1968}.  \\
The second class of algorithms involves sparse projections of the features. This includes all algorithms reviewed by \cite{FreijeiroGonzlez2021cri}, such as Best Subset Selection by \cite{beale1967}, the LASSO by \cite{tibshirani1996reg} and all its variations, such as the Elastic Net by \cite{Zou2005reg}, the Adaptive LASSO by \cite{Zou2006ada} and many more. The LASSO estimator achieves consistent model selection under the regularity conditions discussed, e.g., by \cite{vandegeer2009cond}, as well as oracle estimation rates, as shown by \cite{Bellec2018slo}. However, it has become apparent that these conditions are not satisfied in ill-posed genomic datasets, as demonstrated by \cite{Wang2018}. \\
The third class of algorithms projects onto subspaces that preserve the conditional distributions of the response given the features, thereby yielding sufficient reductions in the sense of \cite{Adragni2009}. This includes Sliced Inverse Regression devised by \cite{Li1991sli} and Partial Least Squares (PLS) proposed by \cite{Wold66non}, among others. In well-specified models, the performance of Partial Least Squares as a dimensionality reduction tool has been investigated by \cite{cook2013env} and \cite{Cook2021pls}. Under the same assumptions, its finite-sample estimation rates have been obtained by \cite{singer2016partial}, whereas its asymptotic prediction risk has been studied by \cite{Cook2019par}. The empirical evidence \citep[see][]{krivobokova2012} suggests that this class of algorithms outperforms competitors from the previous two classes in ill-posed problems.  \\
It is possible to mix and match principles from the previous classes to create hybrid algorithms that still fall within our framework. These methods combine multiple regularization methods and are designed to outperform competitors in well-specified models. A few variants of Principal Components incorporating sparse representations for clustering are AdaCLV by \cite{Marion2020ada} and VC PCR by \cite{Marion2024vc}. A few variants of Partial Least Squares incorporating different penalizations are Sparse PLS by \cite{Chun2010} and Regularized PLS by \cite{Allen2012reg}, just to name a few. Our findings suggest that additional layers of regularization are often redundant or counterproductive. \\
In this work, we offer a different perspective: under minimal assumptions about the data-generating process that reflect typical properties of real-world datasets, we formulate conditions that an algorithm should satisfy to produce an identifiable parameter estimator, thereby quantifying the resulting error.  In particular, we exploit the fact that ill-posedness implies that the response depends, up to an error, on a lower-dimensional projection of the relevant features. First, we define a parameter as identifiable if the population prediction error incurred by replacing the true parameter with its low-dimensional projection is negligible. 
Next, we develop a rigorous framework to characterize when a dimensionality reduction algorithm can produce an interpretable solution. Within our framework, an algorithm is considered interpretable if it delivers an identifiable solution with sufficiently small error. To this end, we introduce a set of conditions that an algorithm must satisfy to be interpretable and discuss them for several classical algorithms. Finally, theoretical bounds on the population and sample errors of algorithms for ill-posed linear regression are derived. It is shown that only statistically interpretable algorithms can achieve negligible population error, whereas any stable algorithm can deliver consistent estimators for the population counterparts. Thereby, the sample convergence rates are driven by the effective dimension of the feature covariance matrix, which is closely related to the problem's degree of ill-posedness.

\section{Main Results} \label{sec:lm} 

We begin by stating our main assumptions, establishing the necessary notation, and giving the main definitions.

\begin{assumption}[Linear, 2 moments]\label{ass:x.y.lm.2nd}
    The features $\bx\in\R^p$ are a random vector and the response $y\in\R$ is a random variable, they are both centered and have finite second moments $\bSigma=\E(\bx\bx^\top)\in\R_{\succeq0}^{p\times p}$, $\bsigma=\E(\bx y)\in\R^p\setminus\{\bzero_{p}\}$ and $\E(y^2)>0$. The features are possibly degenerate with $1\leq r_{\bx} = \rank(\bSigma) \leq p$. For some vector $\bbeta\in\R^p$, the conditional expectation of the response given the features is linear with $\E(y|\bx)=\bx^\top\bbeta$.
\end{assumption}
Under Assumption \ref{ass:x.y.lm.2nd}, the unique minimum-$L^2$-norm solution of the population least-squares problem is given by
\begin{align*}
	\bbeta_{\ls}(\bSigma,\bsigma) = \argmin_{\bbeta\in\R^p} \E(\bx^\top\bbeta - y)^2 = \bSigma^\dagger\bsigma\in\R^p,
\end{align*}
see Lemma \ref{lem:x.y.ls}, where $\bSigma^\dagger$ denotes a unique Moore-Penrose inverse of  $\bSigma$ (or just an inverse, if the matrix is invertible). In particular, it also holds that $\E(y|\bx)=\bx^\top\bbeta_{\ls}(\bSigma,\bsigma)$ and without loss of generality $\bbeta_{\ls}(\bSigma,\bsigma)=\bbeta$.

\subsection{On Identifiable Parameters}\label{sec:iden}
Without additional assumptions, our first goal is to characterize the largest span of linear combinations of features that is irrelevant for the response. As verified in Lemma~\ref{lem:x.y.cov}, the features $\bx$ belong to $\mR(\bSigma)$, the range of $\bSigma$, almost surely. 
Let ${\mB},{\mB}^\bot\subseteq\mR(\bSigma)$ be complementary linear subspaces that partition the range  $\mR(\bSigma)={\mB}\oplus{\mB}^\bot$. Denoting $\bU_{\mB}\in\R^{p\times p}$ the orthogonal projection of $\R^p$ onto ${\mB}$, we define the projected features $\bx_{\mB}=\bU_{\mB}\bx$ and $\bx_{\mB^\bot}=\bU_{\mB^\bot}\bx$. With this, we define the relevant and irrelevant subspaces for the response.
\begin{definition}\label{def:rel.sub}
The subspace 
\begin{align*}
	\mB_{y}^\bot = \argmax \Big\{\dim({\mB}^\bot) : \mR(\bSigma)={\mB}\oplus{\mB}^\bot,\ \E(\bx_{\mB^\bot} y) = \bzero_{p},\ \E(\bx_{\mB^\bot} \bx_{\mB}^\top) = \bzero_{p\times p}\Big\}
\end{align*}
is called {\it irrelevant subspace} for $y$, while its orthogonal complement is called {\it relevant subspace} and is denoted by $\mB_{y}$.
Moreover, $\bx_{y}=\bU_{\mB_y}\bx$ are referred to as {\it relevant features}, while $\bx_{y^\bot}=\bU_{\mB_{y}^\bot}\bx$ are {\it irrelevant features}.
\end{definition}
That is, $ \mB_{y}^\bot$ is the largest subspace of $\mR(\bSigma)$ for which the projected features on it are uncorrelated with both the response $y$ and the projection of the features along its complement. 
 This induces the orthogonal decompositions
$
    \mR(\bSigma) = \mB_{y} \oplus \mB_{y}^\bot$ and $\bx = \bx_{y} \oplus \bx_{y^\bot} \in \mR(\bSigma).
$\\
Note that the definition of $\mB_{y}^\bot $ does not require the linearity assumption $\E(y|\bx)=\bx^\top\bbeta$. Also, the relevant subspace $\mB_{y}$ is allowed to be the whole range $\mR(\bSigma)$ if the irrelevant subspace has dimension zero. 
We prove the following in Appendix~\ref{app:proof:lm}.

\begin{lemma}\label{lem:ls.rel.pop}
    Let $(\bx,y)\in\R^p\times\R$ satisfy Assumption~\ref{ass:x.y.lm.2nd}. Then, the relevant subspace $\mB_{y}$ is unique and $\bbeta_{\ls}(\bSigma,\bsigma) = \bbeta_{\ls}(\bSigma_y,\bsigma_y)$, where $\bSigma_{y}=\E(\bx_y\bx_y^\top)$ and $\bsigma_{y}=\E(\bx_y y)$.\end{lemma}

That is, the vector $\bbeta$ depends only on the moments $\bSigma_{y}$ and $\bsigma_{y}$ of the relevant pair $(\bx_{y},y)$, and $\E(y|\bx) = \bx^\top\bbeta = \bx_{y}^\top\bbeta$ is the best linear predictor for the response in the least-squares sense. \\
Although the relevant subspace is sufficient in solving the population least-squares problem, the covariance matrix $\bSigma_{y}$ of the relevant features might still be ill-posed with a very large condition number $\kappa_{2}(\bSigma_{y})=\|\bSigma_{y}\|_{op}\|\bSigma_{y}^\dagger\|_{op}$. From a practical point of view, this is an issue for consistent estimation via sample least-squares, since the corresponding convergence rate is proportional to such a condition number. From a theoretical point of view, this is an issue because it is hard to identify entries of $\bbeta$ that correspond to highly correlated relevant features. To address this problem, it is reasonable to replace the ill-posed $\bx_{y}$ with its closest, in $L_2$ sense, well-posed low-dimensional projection. It is well known that this approximation is obtained by orthogonal projection onto the subspace spanned by the leading eigenvectors of $\bSigma_{y}$. \\
In what follows, we denote $\mB_{s}\subseteq\mB_{y}$ the span of the first $s$ eigenspaces of $\bSigma_{y}$, for any $1\leq s\leq d_{y}$ where $d_{y}$ is the number of its unique non-zero eigenvalues. Moreover, $\mB_{s}^\bot$ denotes the orthogonal complement of $\mB_{s}$ in $\mB_{y}$. Projecting $\bx_{y}$ onto $\mB_{s}$, we get $\bx_{s}=\bU_{\mB_{s}}\bx_{y}$ so that $\bbeta_{s} = \bbeta_{LS}(\bSigma_s,\bsigma_s)=\bSigma_{s}^\dagger\bsigma_{s}$ is the minimum-$L^2$-norm solution of the population least-squares problem on $(\bx_{s},y)$ with $\bSigma_s=\E(\bx_s\bx_s^\top)$ and $\bsigma_s=\E(\bx_sy)$. It is assumed that $\bsigma_{s}\neq\bzero_{p}$, so that $\bbeta_s$ is nontrivial. \\
Following the idea presented in Example \ref{ex:toy}, we would like to choose $\mB_s$ such that the error of replacing the best linear predictor $\bx^\top\bbeta=\bx_{y}^\top\bbeta$ with the projected predictor $\bx^\top\bbeta_{s}=\bx_{y}^\top\bbeta_{s}$ is acceptably small. The following lemma, which we prove in Appendix~\ref{app:proof:lm}, provides an upper bound on the corresponding risk. 

\begin{lemma}\label{lem:phi.B}
	Let $(\bx,y)\in\R^p\times\R$ satisfy Assumption~\ref{ass:x.y.lm.2nd} and $\mB_{s}\subseteq\mB_{y}$ be any span of first $s$ eigenspaces of $\bSigma_{y}$ for $1\leq s\leq d_{y}$. The relative risk of the features $\bx_{s}$ for $y$ is
	\begin{align*}
		\eps_{s} = \frac{\E(\bx_{y}^\top\bbeta - \bx_{y}^\top\bbeta_{s})^2}{\|\bSigma_{y}\|_{op}\  \|\bbeta\|_{2}^2} \leq \frac{1}{\kappa_{2}(\bSigma_{s+1})} \cdot \frac{\|\bbeta_{s^\bot}\|_{2}^2}{\|\bbeta\|_{2}^2},
	\end{align*}
	with the convention that $\kappa_{2}(\bSigma_{d_{y}+1}) = +\infty$.
\end{lemma}
Since the upper bound on $\eps_s$ depends on the condition number $\kappa_{2}(\bSigma_{s+1})$, larger $s$ lead to a larger condition number and to a smaller risk. Bearing this in mind, we give the following definition of an identifiable parameter in an ill-posed linear regression. 
\begin{definition}\label{def:x.B.y.param}
	Let $\mB_{s}\subseteq\mB_{y}$ be any span of first $s$ eigenspaces of $\bSigma_{y}$ and $\tau>1$. We say that $\bbeta_{s}$ is $\tau$-\emph{identifiable} if $\kappa_{2}(\bSigma_{s}^{1/2}) < \tau$ and $\kappa_{2}(\bSigma_{s+1}^{1/2}) \geq  \tau$. The corresponding linear subspace $\mB_{s}$ is also referred to as $\tau$-\emph{identifiable} and its dimension $\dim(\mB_s)$ gives the \emph{degrees-of-freedom} of the problem. 
\end{definition}
Note that from Lemma \ref{lem:phi.B} follows that if $\bbeta_{s}$ is $\tau$-identifiable, then the relative risk of replacing $\bx_y^\top\bbeta$ by $\bx_y^\top\bbeta_s$ is bounded by $\eps_{s}<\tau^{-2}$. 
Therefore, it seems reasonable to take $s$ such that $\tau$ is at most close to $10$, since otherwise the problem is considered ill-posed in such classical references as \cite{Belsley1980} or \cite{Salmeron2019}. 

\begin{example}[Toy Model, cont.] \label{ex:toy:cont}
	In the toy model from Example~\ref{ex:toy} the relevant features are $\bx_{y}=(x_{1},x_{2},0)^\top$ and the moments of the relevant pair $(\bx_{y},y)$ become $\bSigma_{y}(\rho) = \E(\bx_{y}\bx_{y}^\top)$ and $\bsigma_{y}(\rho) = \E(\bx_{y}y)$. The condition number of the relevant covariance matrix is $\kappa_{2}\{\bSigma_{y}(\rho)\}=(1+\rho)/(1-\rho)$ and goes to infinity when $\rho\to1$. 
	The relevant subspace $\mB_{y}=\R^2\times\{0\}$ is spanned by the eigenvectors $\bv_1=(1/\sqrt{2},1/\sqrt{2},0)^\top$ and $\bold{v}_2=(1/\sqrt{2},-1/\sqrt{2},0)^\top$ of $\bSigma_{y}$. The span of eigenspaces of the relevant covariance matrix are $\mB_{1}=\spa\{\bv_1\}$ and $\mB_{2}=\spa\{\bv_1,\bv_2\}$. 
	The condition numbers and relative risks are
	\begin{align*} 
		\kappa_{2}(\bSigma_{1}^{1/2}(\rho)) &= 1,\quad \eps_{1}(\rho) \leq \frac{1-\rho}{1+\rho} \cdot \frac{(\beta_{1}-\beta_{2})^2}{2(\beta_{1}^2+\beta_{2}^2)}, \quad
		\kappa_{2}(\bSigma_{2}^{1/2}(\rho)) = \sqrt{\frac{1+\rho}{1-\rho}},\quad \eps_{2}(\rho) = 0.
	\end{align*}
	In view of Definition~\ref{def:x.B.y.param}, in the ill-posed setting with $\rho=0.98$, the subspace $\mB_{2}$ yields condition number $ \kappa_{2}\{\bSigma_{2}^{1/2}(\rho)\}=\sqrt{99}\simeq9.95$. Hence, the $\sqrt{99}$-identifiable parameter is $\bbeta_{1}$ with the corresponding relative risk $\eps_{1}(\rho) \leq 0.01$. Obviously, $\bbeta_{1}$ coincides with the interpretable solution $\bbeta_{\ls}(1)$ of the degenerate model discussed in Example~\ref{ex:toy}.
\end{example}

\subsection{Reduction Algorithms}\label{sec:lr}
We aim to study a broad class of reduction algorithms that depend only on the (empirical) moments of the population pair $(\bx,y)$. More specifically, we define an algorithm to be a measurable function $\alg: \R_{\succeq0}^{p\times p}\times \R^p\rightarrow \R^{p\times p}$ that produces regression coefficient estimators as a minimum-$L_2$-norm solution to the least-squares problem of regressing response on the features projected onto a series of $p$ embedded linear subspaces. \\
Algorithms treated in this work take a deterministic matrix-vector pair $(\bA,\bb)\in\R_{\succeq0}^{p\times p}\times \R^p$ and produce $p$ embedded linear subspaces $\{\bzero_{p}\} = \mB_{\alg}^{(0)}(\bA,\bb) \subseteq \mB_{\alg}^{(1)}(\bA,\bb) \subseteq \cdots \subseteq \mB_{\alg}^{(p)}(\bA,\bb) \subseteq \R^p$, with dimensions $\dim\{\mB_{\alg}^{(t)}(\bA,\bb)\} \leq t$. Denote also $\bU_{\alg}^{(t)}(\bA,\bb)$ the corresponding orthogonal projections of $\R^p$ onto $\mB_{\alg}^{(t)}(\bA,\bb)$. Subsequently, corresponding regression coefficients $\bbeta_{\alg}^{(t)}(\bA,\bb)$ can be calculated. We refer to Appendix~\ref{sec:rls:alg} for the formal numerical definition, and to Section~\ref{sec:appl} for examples on how classical algorithms fit into our framework. \\
The algorithms that we consider need to be continuous in some sense. In this work, the notion of algorithm stability will be crucial.

\begin{definition}\label{def:reg.alg.stab}
	Let $(\bA,\bb)\in\R_{\succeq0}^{p\times p}\times \R^p$ be fixed and consider any subspace $\mB_\alg(\bA,\bb)$ together with the corresponding orthogonal projection $\bU_\alg(\bA,\bb)$. Let also $(\wt{\bA},\wt{\bb})\in\R_{\succeq0}^{p\times p}\times \R^p$ be any other matrix-vector pair and set
$$
		{\eps}(\wt\bA,\wt\bb,\bA,\bb) = \frac{\|\wt{\bA}-\bA\|_{op}}{\|\bA\|_{op}} \vee \frac{\|\wt{\bb}-\bb\|_{2}}{\|\bb\|_{2}}.
$$
The algorithm $\alg(\bA,\bb)$ is said to be {\it stable} if there exist constant $C_\alg(\bA,\bb)\geq1$ such that $\|{\bU}_{\alg}(\wt\bA,\wt\bb)-\bU_{\alg}(\bA,\bb)\|_{op} \leq C_{\alg}(\bA,\bb)\ {\eps}(\wt\bA,\wt\bb,\bA,\bb)$, where ${\bU}_{\alg}(\wt\bA,\wt\bb)$ is the orthogonal projection onto  any $\mB_\alg(\wt{\bA},\wt{\bb})$ such that $\dim\{{\mB}_{\alg}(\wt\bA,\wt\bb)\}=\dim\{\mB_{\alg}(\bA,\bb)\}$. \end{definition}
In connection to Definition \ref{def:reg.alg.stab}, it is also useful to define a constant 
$$
		M_{\alg}(\bA,\bb)= 2 \ \kappa_{2}(\bA_U)  \{4\ C_{\alg}(\bA,\bb) + 1\}  \left\{\frac{\|\bA\|_{op}}{\|\bA_U\|_{op}} \vee \frac{\|\bb\|_{2}}{\|\bb_U\|_{2}}\right\},
$$
where $\bA_U=\bU_{\alg}(\bA,\bb)\bA\bU_{\alg}(\bA,\bb)$ and $\bb_U=\bU_\alg(\bA,\bb)\bb$.

This notion of stability under perturbation is formalized in full generality in Section~\ref{sec:rls:alg.pert}, together with examples showing that classical algorithms are stable.

\subsubsection{Population Reduction Algorithms and Interpretability}\label{sec:lr:pop}
A population reduction algorithm $\alg(\bSigma,\bsigma)$ takes $\bSigma$ and $\bsigma$ as input and produces coefficients
$
	\bbeta_{\alg}^{(t)}=\bbeta_{\alg}^{(t)}(\bSigma,\bsigma)$, $t=1,\ldots,p$ by constructing linear subspaces $\mB_{\alg}^{(t)}=\mB_\alg^{(t)}(\bSigma,\bsigma)$ such that $\{\bzero_{p}\} = \mB_{\alg}^{(0)} \subseteq \mB_{\alg}^{(1)} \subseteq \cdots \subseteq \mB_{\alg}^{(p)} \subseteq \R^p$ with dimensions $\dim(\mB_{\alg}^{(t)}) \leq t$. More specifically, $\bbeta_\alg^{(t)}=(\bU_{\alg}^{(t)}\bSigma \bU_{\alg}^{(t)})^\dagger\bU_{\alg}^{(t)}\bsigma$, where $\bU_{\alg}^{(t)}=\bU_{\alg}^{(t)}(\bSigma,\bb)$ denotes the orthogonal projections onto corresponding $\mB_{\alg}^{(t)}$. In particular, it is assumed that $\bbeta_{\alg}^{(p)}=\bbeta_{\ls}(\bSigma,\bsigma)=\bbeta$ recovers the population least-squares solution. \\
In this section, the performance of such population reduction algorithms $\alg(\bSigma,\bsigma)$ for estimation of the $\tau$-identifiable parameter $\bbeta_{s}$ is studied. First, we define properties of algorithms that are necessary to upper-bound the error made by applying algorithms to $(\bx,y)$ instead of $(\bx_s,y)$.\\
Lemma~\ref{lem:ls.rel.pop} shows that the population least-squares solution $\bbeta_{\ls}(\bSigma,\bsigma)=\bbeta_{\ls}(\bSigma_{y},\bsigma_y)$ only depends on the moments of the relevant pair $(\bx_{y},y)$ and the irrelevant information is implicitly discarded. Since an algorithm only has access to the moments of the whole population pair $(\bx,y)$, it is paramount for the identification of $\bbeta_{s}$ to devise methods that can implicitly adapt to the unknown relevant directions.
\begin{definition}\label{def:reg.alg.pop.adap}
	A population algorithm $\alg(\bSigma,\bsigma)$ is said to be {\it adaptive} if 
	$\alg(\bSigma,\bsigma) = \alg(\bSigma_{y},\bsigma_{y})$.
\end{definition}
As a rule of thumb, population algorithms $\alg(\bSigma,\bsigma)$ that make decisions based on the components of $\bbeta_{\ls}(\bSigma,\bsigma)$ are expected to be adaptive. We refer to Section~\ref{sec:appl} below for more detailed examples.\\
Adaptivity alone guarantees only that such algorithms are not misled by irrelevant information. However, it does not, by itself, ensure that the relevant information is used efficiently. 
Let $\mB_\alg^*=\mB_\alg^{(p)}(\bSigma_s,\bsigma_s)$, that is $\mB_\alg^*$ is a linear subspace produced by an algorithm that has an oracle knowledge of the relevant subspace $\mB_s$.
\begin{definition}\label{def:reg.alg.pop.parsim}
	A population algorithm $\alg(\bSigma,\bsigma)$ is said to be {\it parsimonious} if $\mB_{\alg}^*\subseteq\mB_{s}$.
\end{definition}

As a rule of thumb, population algorithms $\alg(\bSigma_{s},\bsigma_{s})$ that make decisions based on the spectrum of $\bSigma_{s}$ are expected to be parsimonious. We refer to Section~\ref{sec:appl} below for more detailed examples.
Finally, we can define statistically interpretable algorithms.\begin{definition}\label{def:reg.alg.pop.interp}
	A population algorithm $\alg(\bSigma,\bsigma)$ is said to be {\it statistically interpretable} for the $\tau$-identifiable parameter $\bbeta_{s}$ if it is adaptive, parsimonious, and $\alg(\bSigma_s,\bsigma_s)$ is stable.
\end{definition}
Our next goal is to evaluate the error a statistically interpretable algorithm $\alg(\bSigma,\bsigma)$ makes in estimating the $\tau$-interpretable $\bbeta_s$. First note that $\bbeta^*_\alg=\bbeta_\alg^{(p)}(\bSigma_s,\bsigma_s)=\bSigma_s^\dagger\bsigma_s=\bbeta_s$. That is, if a statistically interpretable algorithm is applied with an oracle knowledge of $\bSigma_s$ and $\bsigma_s$, it produces $\bbeta_s$, while  $\mB_{\alg}^*\subseteq\mB_{s}$. Let now $\bbeta_\alg=\bbeta_\alg^{(t)}(\bSigma,\bsigma)$, where $t$ is such that $\mbox{dim}(\mB_\alg^*)=\mbox{dim}(\mB_\alg^{(t)})$. Hence, $\bbeta_\alg$ and $\bbeta_\alg^*=\bbeta_s$ are compatible.\\
The following theorem shows that statistical interpretability is a sufficient condition for a population algorithm to estimate the $\tau$-identifiable parameter and provides an upper bound on the estimation error of such algorithms. By means of examples, we show in Section~\ref{sec:appl} that this condition is nearly necessary since algorithms that are not statistically interpretable can have arbitrarily large population error. 

\begin{theorem}\label{thm:x.y.reg.alg.pop}
	Let $(\bx,y)\in\R^p\times\R$ satisfy Assumption~\ref{ass:x.y.lm.2nd} and let $\alg(\bSigma,\bsigma)$ be a population statistically interpretable algorithm. Let $\bbeta_{\alg}^*\in\mB_{\alg}^*$ and $\bbeta_{\alg}\in\mB_{\alg}^{(t)}$ be compatible parameters with $\dim(\mB_{\alg}^*)=\dim(\mB_{\alg}^{(t)})$ and recall that $\bbeta_\alg^*=\bbeta_s$. If the size of the perturbation is $\eps^*=\eps(\bSigma_{y},\bsigma_{y},\bSigma_s,\bsigma_s) < 1/M_{\alg}(\bSigma_s,\bsigma_s)$, then the population error is
	\begin{align*}
		\frac{\|\bbeta_{\alg}-\bbeta_{s}\|_{2}}{\|\bbeta_{s}\|_{2}} &\leq \frac{5}{2}\ M_{\alg}(\bSigma_s,\bsigma_s)\ \eps^*.
	\end{align*}
\end{theorem}
The proof, as well as the extension of this result that includes early-stopping, is given in Appendix~\ref{app:proof:lm}. \\
From Lemma~\ref{lem:x.rel.y.eps} follows that the size of the perturbation $\eps^*$ is at most proportional to $\kappa_{2}(\bSigma_{s+1})^{-1} < \tau^{-2}$. 
By the definition, $M_{\alg}(\bSigma_s,\bsigma_s)$ is proportional to $\kappa_{2}(\bSigma_{s}) \leq \tau^2$. If in an ill-posed setting the ratio between the $s$-th and the $(s+1)$-th eigenvalue of $\bSigma_y$ goes to infinity, then $M_\alg (\bSigma_s,\bsigma_s)\eps^*=o(1)$. \\
	We address here all our assumptions leading to the population error bounds in Theorem~\ref{thm:x.y.reg.alg.pop}. This result is a special case of a general perturbation bound which we establish with Theorem~\ref{thm:alg.pert} in Section~\ref{sec:rls:alg.pert}. Assumption~\ref{ass:x.y.lm.2nd} is the weakest condition under which the set of non-trivial solutions to the population least-squares problem exists. 
	An assumption of adaptivity means that the population algorithm implicitly discards irrelevant information about the response. This is a sufficient condition that we exploit to remove the contribution of the irrelevant features from the perturbation size and use $\eps^*$ instead of $\eps(\bSigma,\bsigma,\bSigma_s,\bsigma_s)$. The threshold $1/M_{\alg}(\bSigma_s,\bsigma_s)$ quantifies the level of ill-posedness of the relevant features.
	An assumption of parsimony implies that it is harmless to replace the $\tau$-interpretable parameter $\bbeta_{s}$ with the parameter $\bbeta_{\alg}^*$ computed by the population algorithm with oracle knowledge. 
	Finally, the required stability of an algorithm $\mA(\bSigma_s,\bsigma_s)$ is natural, and the stability of classical algorithms has been investigated in the literature. The stability of the population PCR algorithm is shown by \cite{davis1970rot} and \cite{godunov1993gua}. The stability of the population PLS algorithm is discussed by \cite{carpraux1994SotK} and \cite{kuznetsov1997Pbot}. For the stability of penalized sparse methods, one can consult \cite{cerone2019lin} and \cite{fosson2020spa}. See Section~\ref{sec:rls:alg.pert} for more details.
	
\subsubsection{Sample Reduction Algorithms and Consistency}\label{sec:lm:sam}
Consider a dataset $(\bX,\by)\in\R^{n\times p}\times\R^n$ consisting of $n\geq1$ i.i.d. realizations $(\bx_{i},y_{i})$ of the same population pair $(\bx,y)\in\R^p\times\R$ under Assumption~\ref{ass:x.y.lm.2nd}. In this section, we investigate the performance of sample reduction algorithms $\alg(\wh{\bSigma},\wh{\bsigma})$ that depend only on the sample moments $\wh{\bSigma}=n^{-1}\bX^\top\bX$ and $\wh{\bsigma}=n^{-1}\bX^\top\by$ estimated from the dataset $(\bX,\by)$. We denote  $\wh{\bbeta}_{\ls}(\bX,\by) = \wh{\bSigma}^\dagger\wh{\bsigma}$  the corresponding sample least-squares solution. Similar to the population case, $\alg(\wh{\bSigma},\wh{\bsigma})$ computes coefficients $\wh{\bbeta}_{\alg}^{(t)}={\bbeta}_{\alg}^{(t)}(\wh\bSigma,\wh\bsigma)$, $t=1,\ldots,p$ and $
	\wh{\bbeta}_{\alg} $ is the estimator of $\bbeta_\alg$, both having the same degrees-of-freedom.
\begin{assumption}[Linear, $q>4$ moments]\label{ass:x.y.lm.4th}
	Let $(\bx,y)\in\R^p\times\R$ satisfy Assumption~\ref{ass:x.y.lm.2nd} and let for some $q>4$
	\begin{align*}
		L_{y} =\frac{\E(y^q)^{\frac{1}{q}}}{\E(y^2)^{\frac{1}{2}}}<\infty,\quad
		L_{\bx} = \sup_{\bv\in\partial\B_{2}^p(1)} \frac{\E(|\bx^\top\bv|^q)^{1/q}}{\E(|\bx^\top\bv|^2)^{1/2}}<\infty.
	\end{align*}
\end{assumption}

Let $(\bx,y)\in\R^p\times\R$ satisfy Assumption~\ref{ass:x.y.lm.4th} and denote
\begin{align}\label{eq:x.proj.geom}
	r_{\bx} = \rank(\bSigma),\quad 
	\rho_{\bx} = \frac{\E(\|\bx\|_{2}^2)}{\|\bSigma\|_{op}},\quad 
	\rho_{\bx,n} = \frac{\E(\max_{1\leq i\leq n} \|\bx_{i}\|_{2}^q)^{2/q}}{\|\bSigma\|_{op}}.
\end{align}
The rank $r_{\bx}$ is the dimension of the span of the support of $\bx$. The effective rank $\rho_{\bx} \leq r_{\bx}$ can be rewritten as the weighted average $\Tr(\bSigma)/\|\bSigma\|_{op}$ and measures the interplay between dimension and variation. The uniform effective rank $\rho_{\bx,n}$ accounts for the variability of a sample of i.i.d. realizations of $\bx$. We define the sequence
\begin{align} \label{eq:x.y.delta.star.n}
	\delta_{n} = \sqrt{\frac{\rho_{\bx}}{n}} + \frac{\rho_{\bx,n}}{n},
\end{align}
summarizing the intrinsic geometrical complexity.

\begin{theorem} \label{thm:x.y.reg.alg.sam}
	Let $(\bx,y)\in\R^p\times\R$ satisfy Assumption~\ref{ass:x.y.lm.4th}. Let $(\bX,\by)\in\R^{n\times p}\times\R^n$ be a dataset of i.i.d. realizations of $(\bx,y)$. Let $\alg(\wh\bSigma,\wh\bsigma)$ be a sample reduction algorithm such that $\alg(\bSigma,\bsigma)$ is stable in the sense of Definition~\ref{def:reg.alg.stab}. With $\delta_{n}$ the complexity in Equation~\eqref{eq:x.y.delta.star.n}, some absolute constant $C\geq1$, $K = \max\{C,\ L_{y} L_{\bx}\ \sigma_{y} \|\bSigma\|_{op}^{{1}/{2}}/\|\bsigma\|_{2}\}$, assume $\lim_{n\to\infty}\delta_{n}=0$ and $K M_{\alg}(\bSigma,\bsigma) \delta_{n} < 1/2$. Then, for any $K M_{\alg}(\bSigma,\bsigma) \delta_{n} < \nu_{n} < 1/2$, the size of the sample perturbation $\wh\eps=\eps(\wh\bSigma,\wh\bsigma,\bSigma,\bsigma)$ satisfies 
	$\P\left(\wh{\eps} \leq K {\delta_n}/{\nu_n}\right) \geq 1-2\nu_{n}.$ On this event, the sample error is bounded by
	\begin{align*}
		\frac{\|\wh{\bbeta}_{\alg} - \bbeta_{\alg}\|_{2}}{\|\bbeta_{\alg}\|_{2}} &\leq \frac{5}{2}\  M_{\alg}(\bSigma,\bsigma)\  \wh\eps .
	\end{align*}    
\end{theorem}
We refer to Appendix~\ref{app:proof:lm} for the proof and for an extension (see Corollary~\ref{cor:x.y.reg.alg.sam}) of this result to include early stopping. Different from Theorem~\ref{thm:x.y.reg.alg.pop} for the population algorithms error, neither adaptivity nor parsimony is required to bound the sample errors in Theorem~\ref{thm:x.y.reg.alg.sam}. Under the assumptions of the latter, it is always possible to find a sequence $\nu_{n}\to0$ arbitrarily slow and $\delta_{n}/\nu_n \to 0$ when $n\to\infty$. This means that stable algorithms $\alg(\bSigma,\bsigma)$ for dimensionality reduction yield the convergence in probability $\wh{\bbeta}_{\alg} \xrightarrow{\P} \bbeta_{\alg}$, $n\rightarrow\infty$ or, in other words, all sample reduction algorithms are consistent for their population counterpart. \\
We address here assumptions leading to the sample error bounds in Theorem~\ref{thm:x.y.reg.alg.sam}. This theorem is a special case of a general perturbation bound which we establish with Theorem~\ref{thm:alg.pert} in Section~\ref{sec:rls:alg.pert}. Assumption~\ref{ass:x.y.lm.4th} is a stronger version of Assumption~\ref{ass:x.y.lm.2nd} in the sense that it requires finiteness of more than four moments instead of two. Our theoretical results hold under the regime $\rho_{\bx,n}/n \to 0$. Note that this implies $\rho_{\bx}/n \to 0$ due to $\rho_{\bx}\leq\rho_{\bx,n}$ as shown in Lemma \ref{lem:x.complexity}, whereas nothing is assumed on the ratio $p/n$ which can be arbitrarily large. \\
Combining Theorem \ref{thm:x.y.reg.alg.pop} and Theorem~\ref{thm:x.y.reg.alg.sam}, gives the upper bound on the error that a statistically interpretable algorithm $\alg(\wh\bSigma,\wh\bsigma)$ makes when estimating $\tau$-identifiable $\bbeta_s$.
\begin{theorem} \label{thm:x.y.reg.alg.pop.sam}
	Under the assumptions of Theorem~\ref{thm:x.y.reg.alg.pop} and Theorem~\ref{thm:x.y.reg.alg.sam}, on the same event of probability at least $1-2\nu_{n}$, the estimation error is
	\begin{align*}
		\frac{\|\wh{\bbeta}_{\alg} - \bbeta_{s}\|_{2}}{\|\bbeta_{s}\|_{2}} &\leq \frac{5}{2}\ M_{\alg}(\bSigma_s,\bsigma_s)\ \eps^* + \frac{35}{4}\  M_{\alg}(\bSigma,\bsigma)\ \wh\eps.
	\end{align*} 
\end{theorem}
Theorem~\ref{thm:x.y.reg.alg.pop.sam} shows that the overall estimation error of a statistically interpretable sample algorithm is determined by the sum of the population and sample errors. Obviously, only statistically interpretable algorithms can achieve optimal estimation errors, while all others would incur arbitrarily large population errors.\\
Now, let us turn to the discussion on the convergence rate of the sample errors. The sample errors in Theorem~\ref{thm:x.y.reg.alg.sam} are obtained for a sufficiently large sample of i.i.d. realizations $(\bx_{i},y_{i})$. The sample error (up to a constant) is bounded by
$$
\wh\eps\lesssim \sqrt{\frac{\rho_{\bx}}{n}} + \frac{\rho_{\bx,n}}{n},
$$
giving the corresponding convergence rate, that depends on the effective rank $\rho_\bx$ and uniform effective rank $\rho_{\bx,n}$. This result is due to Theorem 6 in \cite{Jirak2025}. The effective rank $\rho_\bx$ is related to the degree of the ill-posedness of the problem (a smaller effective rank implies a higher degree of ill-posedness), while the uniform effective rank $\rho_{\bx,n}$ depends on the distribution of the features. Roughly speaking, lighter tails of the feature distribution lead to a smaller uniform effective rank. For example, under assumption of sub-Gaussianity, $\wh\eps$ is driven solely by $\sqrt{\rho_{\bx}/n}$. Avoiding any distributional assumptions and relying only on Assumption \ref{ass:x.y.lm.4th}, one can bound the rate as follows, see Lemma~\ref{lem:x.complexity} and Corollary~\ref{cor:mom.mat.heavy.new},
$$
\wh\eps\lesssim \sqrt{\frac{\rho_{\bx}}{n}}  \left(1 + L_{\bx} \sqrt{\frac{\rho_{\bx}}{n^{\frac{q-4}{q}}}} \right).
$$
This shows that if the effective rank is sufficiently small, that is, $\rho_\bx\lesssim n^{(q-4)/q}$, where $q$ is the number of feature moments, then the convergence rate is the same as for sub-Gaussian features. Note that one can construct algorithms (that cannot be implemented) which achieve the rate $\sqrt{\rho_\bx/n}$ under Assumption \ref{ass:x.y.lm.4th}, see  \cite{Bartl2025uni}.\\
The convergence rate of Theorem \ref{thm:x.y.reg.alg.sam} contrasts known results on the minimax rate $\sqrt{\rank(\bSigma)/n}$ of sample least-squares estimators of $\bbeta$, as shown by \cite{Mourtada2022} and on the minimax rate $\sqrt{{\log(p/\|\bbeta\|_0)\,\|\bbeta\|_0}/{n}} $ of sparse methods derived by \cite{Bellec2018slo}. In many applications $\bSigma$ has a full rank $p$ and sparsity is proportional to $p$ as well, that is, $\|\bbeta\|_0 = c p$ for some small constant $c \in (0,1)$. Hence, the convergence rate of both methods would be $\sqrt{p/n}$. If the problem is ill-posed, in many applications one encounters the effective rank to be $\rho_\bx=\log(p)$, leading to a much faster convergence rate $\sqrt{\log(p)/n}$. In Section \ref{sec:num} we demonstrate this effect on simulated and real data.

\section{Applications to Algorithms}\label{sec:appl}
In this section, we address the implications of our theory for existing linear dimensionality reduction algorithms. The most prominent strategies within our framework are based on unsupervised projection, sparsity, and sufficient reduction. As mentioned above, stability of classical algorithms is well-studied, see Section~\ref{sec:rls:alg.pert} for more details. In what follows, we provide examples of population algorithms that fit our framework and discuss whether they are adaptive or parsimonious. 

\subsection{Unsupervised Projections}
The class of unsupervised projection methods consists of all algorithms that project the features along directions chosen independently of the response. One prominent example is the population PCR algorithm $\pcr(\bSigma,\bsigma)$, inspired by \cite{hotelling1933anal}, which makes the following choices. For all $t=1,\ldots,p$, the algorithm selects subspaces $\mB_{\pcr}^{(t)} = \spa\{\bu_{1},\ldots,\bu_{t}\}$ spanned by the first $t$ eigenvectors of $\bSigma$ sorted according to the decreasing order of the corresponding eigenvalues, including repeated ones. The sequence of subspaces becomes constant as soon as $t\geq\rank(\bSigma)$. Therefore, $\bbeta_{\pcr}^{(p)} = \bbeta_{\ls}$ recovers the population least-squares. It is immediate to check that such an algorithm is parsimonious in the sense of Definition~\ref{def:reg.alg.pop.parsim} since the largest selected subspace is $\mB_{\pcr}^{(p)}(\bSigma_{s},\bsigma_{s})=\mR(\bSigma_{s})=\mB_{s}$. We now show that this population algorithm is not adaptive in the sense of Definition~\ref{def:reg.alg.pop.adap}. Assume that the $\tau$-identifiable parameter is one-dimensional with $\bbeta_{1}\in\mB_{1}$ and $\dim(\mB_{1})=1$. Furthermore, assume the largest eigenvalue of $\bSigma$ is strictly larger than the largest eigenvalue of $\bSigma_{y}$. Then, the population algorithm $\pcr(\bSigma_{y},\bsigma_{y})$ selects $\mB_{\pcr}^{(1)}(\bSigma_{y},\bsigma_{y}) = \spa\{\bu_{1}(\bSigma_{y})\}$ which is orthogonal to $\mB_{\pcr}^{(1)}(\bSigma,\bsigma) = \spa\{\bu_{1}(\bSigma)\}$ since they are spanned by eigenvectors of $\bSigma$ from different eigenspaces. Hence, $\pcr(\bSigma,\bsigma)\neq \pcr(\bSigma_y,\bsigma_y)$ and the PCR algorithm is not adaptive. 
In particular, this implies the relative error $\|\bbeta_{\pcr}-\bbeta_{1}\|_{2}/\|\bbeta_{1}\|_{2} = 1$ regardless of $\eps^*$, since $\bbeta_{\pcr} = \bzero_{p}$ with one degree of freedom. Obviously, any algorithm based on unsupervised projections cannot be adaptive since it is ignorant of $\bSigma_y$. 

\subsection{Sparse Projections}
The class of sparse projection methods comprises all algorithms that select active feature sets to explain the response. The method of Forward Subset Selection (FSS) inspired by \cite{beale1967} yields the population algorithm $\fss(\bSigma,\bsigma)$ that makes the following choices. Let $J_{\ls} = \{1\leq j\leq p: \be_{j}^\top\bbeta_{\ls}\neq0\}$ denote the active set of $\bbeta_{\ls}=\bbeta_{\ls}(\bSigma,\bsigma)$, which has cardinality $|J_{\ls}| = \|\bbeta_{\ls}\|_{0}$. For all $t=1,\ldots,p$, the algorithm outputs coefficients $\bbeta_{\fss}^{(t)}$ that are obtained from $\bbeta_{\ls}$ by setting to zero all the entries that are not in the selected active subset $J_{\fss}^{(t)}\subseteq J_{\ls}$. Each active subset is obtained from the previous one as $J_{\fss}^{(t)} = j_{\fss}^{(t)} \cup J_{\fss}^{(t-1)}$ where $j_{\fss}^{(t)}$ is the index that minimizes the residual $j\mapsto\E(y-\bx^\top\bbeta_{\fss}^{(t-1)} -\bx^\top\be_{j}\be_{j}^\top\bbeta_{\ls})^2$. The sequence of subspaces $\mB_{\fss}^{(t)} = \spa\{\be_{j} : j\in J_{\fss}^{(t)}\}$, which are spanned by the selected active directions, becomes constant as soon as $t\geq\|\bbeta_{\ls}\|_{0}$. Therefore, $\bbeta_{\fss}^{(p)} = \bbeta_{\ls}$ recovers the population least-squares. Since all choices made by $\fss(\bSigma,\bsigma)$ only involve the components of $\bbeta_{\ls}(\bSigma,\bsigma)$, and since Lemma~\ref{lem:ls.rel.pop} shows that $\bbeta_{\ls}(\bSigma_{y},\bsigma_{y})=\bbeta_{\ls}(\bSigma,\bsigma)$, this population algorithm is adaptive in the sense of Definition~\ref{def:reg.alg.pop.adap}. On the other hand, we now show that this population algorithm is not parsimonious in the sense of Definition~\ref{def:reg.alg.pop.parsim}. Assume the $\tau$-identifiable subspace is one-dimensional with $\mB_{1}=\spa\{{\bf1}_{p}\}$ spanned by the vector ${\bf1}_{p}=\be_{1}+\cdots+\be_{p}$, so that the $\tau$-identifiable parameter is $\bbeta_{1}=c {\bf1}_{p}$ for some $c\neq0$. In order for $\fss(\bSigma,\bsigma)$ to be parsimonious, we need that $\mB_{\fss}^{(p)}(\bSigma_{1},\bsigma_{1}) \subseteq \mB_{1}$, which is never possible, as long as $p\geq2$, since $\mB_{\fss}^{(p)}(\bSigma_{1},\bsigma_{1}) = \R^p$ has dimension $p > 1 = \dim(\mB_{1})$. In such a case, the performance of the population algorithm in Theorem~\ref{thm:x.y.reg.alg.pop} must be measured in terms of the parameter $\bbeta_{\fss} = \beta_{j}\be_{j}$ for some $j$, using $1$ degree-of-freedom. In particular, this implies $\|\bbeta_{\fss}-\bbeta_{1}\|_{2}/\|\bbeta_{1}\|_{2} \geq \sqrt{p-1}/\sqrt{p}$, regardless of the error $\eps^*$. From these considerations, it is clear that any algorithm, including Lasso and its variations, that performs sparse model selection can not be parsimonious when the model is only sparse up to an unknown rotation.

\subsection{Sufficient Projections}
The class of sufficient projection methods consists of all those algorithms that project the features along directions that preserve the conditional distribution of the response or its first two moments. One popular example is the population PLS algorithm $\pls(\bSigma,\bsigma)$, inspired by \cite{Wold66non}, which makes the following choices. For all $t=1,\ldots,p$, the algorithm selects subspaces $\mB_{\pls}^{(t)} = \spa\{\bsigma,\ldots,\bSigma^{t-1}\bsigma\}$ spanned the first $t$ Krylov vectors generated by $\bSigma,\bsigma$. We show in Lemma~\ref{lem:ls.krylov} that $\bbeta_{\ls}\in\mB_{\pls}^{(p)}$ and $\bbeta_{\pls}^{(p)} = \bbeta_{\ls}$ recovers the population least-squares. This policy is parsimonious in the sense of Definition~\ref{def:reg.alg.pop.parsim} since the largest selected subspace satisfies $\mB_{\pls}^{(p)}(\bSigma_{s},\bsigma_{s})\subseteq\mR(\bSigma_{s})=\mB_{s}$. We now show that the population algorithm is adaptive in the sense of Definition~\ref{def:reg.alg.pop.adap}. To see this, consider $\pls(\bSigma_{y},\bsigma_{y})$ and consider $\bSigma_{y}^{t-1}\bsigma_{y}$ the corresponding Krylov vectors. It is sufficient to check that $\bSigma^{t-1}\bsigma=\bSigma_{y}^{t-1}\bsigma_{y}$ for all $1\leq t\leq p$. This is done by induction. For $t=1$, the base case $\bsigma = \bsigma_{y}$ holds by definition of relevant subspace. Let now $\bSigma^{t-2}\bsigma=\bSigma_{y}^{t-2}\bsigma_{y}$ hold. Then, $\bSigma^{t-1}\bsigma = \bSigma \bSigma_y^{t-2}\bsigma_y=\bSigma_{y}\bSigma_{y}^{t-2}\bsigma_{y} + \bSigma_{y^\bot}\bSigma_{y}^{t-2}\bsigma_{y} = \bSigma_{y}^{t-1}\bsigma_{y} $. Theorem~\ref{thm:x.y.reg.alg.pop} always applies for the population PLS algorithm and the error $\|\bbeta_{\pls}-\bbeta_{s}\|_{2}/\|\bbeta_{s}\|_{2}$ becomes zero when $\eps^*=0$ as well. Naturally, sufficient projection algorithms are constructed to be adaptive, but parsimony can only be ensured by early stopping.

\subsection{Toy Model}
We work out the application of classical algorithms to the toy model from Example~\ref{ex:toy}. We fix a special case where $\beta_{1}\neq0$ and $\beta_{2}=0$. One finds
\begin{align*} 
	\bSigma(\rho) = \left(
	\begin{matrix} 
		1 & \rho & 0 \\
		\rho & 1 & 0 \\
		0 & 0 & 2
	\end{matrix}\right), \quad
	\bsigma(\rho) = \left(
	\begin{matrix} 
		\beta_{1} \\ 
		\rho\beta_{1} \\
		0
	\end{matrix}\right), \quad
	\bSigma_{y}(\rho) = \left(
	\begin{matrix} 
		1 & \rho & 0 \\
		\rho & 1 & 0 \\
		0 & 0 & 0
	\end{matrix}\right), \quad
	\bsigma_{y}(\rho) = \bsigma(\rho),
\end{align*}
and $\bbeta_{\ls}(\bSigma,\bsigma) = (\beta_{1},0,0)^\top$. The eigenvectors of the covariance matrix $\bSigma(\rho) $ are $\bu_{1}=(0,0,1)^\top$, $\bu_{2}=(1/\sqrt{2},1/\sqrt{2},0)^\top$, $\bu_{3}=(1/\sqrt{2},-1/\sqrt{2},0)^\top$ with corresponding eigenvalues $2 > 1+\rho > 1-\rho$. We have shown in Example~\ref{ex:toy:cont} that the $\sqrt{99}$-identifiable parameter is $\bbeta_{1} = (\beta_{1}/2,\beta_{1}/2,0)^\top$ and that the $\sqrt{99}$-identifiable subspace is one-dimensional with $\mB_{1}=\spa\{\bu_{2}\}$ both independent on $\rho$. We now compare how different dimensionality reduction strategies deal with the estimation of $\bbeta_{1}$ while using subspaces of dimensions exactly $\dim(\mB_{1})=1$. For unsupervised reduction, we consider $\beta_{\pcr} = \bu_{1}\bu_{1}^\top\bbeta_{\ls} = (0,0,0)^\top$ the projection of $\bbeta_{\ls}$ along the direction $\bu_{1}$ of largest variance of $\bx$. For sparse reduction, we consider $\beta_{\spr} = \be_{1}\be_{1}^\top\bbeta_{\ls} = (\beta_{1},0,0)^\top$ the projection of $\bbeta_{\ls}$ along the canonical vector $\be_{1}=(1,0,0)^\top$. For sufficient reduction, we consider $\bbeta_{\pls} = \bu_{\bsigma}(\rho)\bu_{\bsigma}(\rho)^\top\bbeta_{\ls} = \{\beta_{1}/(1+\rho^2),\rho\beta_{1}/(1+\rho^2),0\}^\top$ the projection of $\bbeta_{\ls}$ along the direction $\bu_{\bsigma}(\rho) = \bsigma(\rho)/\|\bsigma(\rho)\|_{2}$ of largest covariance between $\bx$ and $y$. We compute the errors in our Theorem~\ref{thm:x.y.reg.alg.pop} directly, without additional assumptions, and find the population error $\delta_{\alg}=\|\bbeta_{\alg}-\bbeta_{1}\|_{2}/{\|\bbeta_{1}\|_{2}}$ with $\rho=0.98$ for all three algorithms to be
\begin{align*}
	\delta_{\pcr} = 1,\ 
	\delta_{\spr} = 1,\ 
	\delta_{\pls} = \frac{(1-\rho)\sqrt{(1+\rho)^2+(1-\rho)^2}}{\sqrt{2}(1+\rho^2)} \approx 0.014,
\end{align*}
showing that the direction of maximal covariance $\bu_{\bsigma}(\rho)$ is the only one that yields a small estimation error for the $\sqrt{99}$-identifiable parameter $\bbeta_{1}$. This is true despite the direction $\be_{1}$ being the one that better explains the conditional expectation $\E(y|\bx)$. In fact, if we define the relative prediction risk for algorithms as
$$
\eps_{\alg} = \frac{\E(\{\bx^\top\bbeta_{\ls} - \bx^\top\bbeta_{\alg}\}^2)}{\|\bSigma\|_{op} \|\bbeta_{\ls}\|_{2}^2} 
$$
then for $\rho=0.98$ we find
\begin{align*}
 \eps_{\pcr} = \frac{1}{2},\quad \eps_{\spr} = 0,\quad \eps_{\pls} = \frac{\rho^2(1-\rho^2)}{2(1+\rho^2)^2} \approx 0.005,
\end{align*}
showing that the sparse vector $\beta_{\spr}$ provides the best predictor of the response, whereas $\beta_{\pls}$ has positive but small risk.

\section{Numerical Studies} \label{sec:num}
We confirm our findings with empirical studies on both simulated and real datasets.

\subsection{Simulated Data}\label{sec:num-simul}
We consider a setting that is compatible with genomics applications where $p\gg n$ and the population covariance matrix has full-rank $r_{\bx}=p$ but a small effective rank $\rho_{\bx}\ll n$.  We simulate only one setting where the underlying model is non-spare and there are irrelevant features for the response with large variation.

We simulate our dataset $(\bX,\by)$ according to the following scheme:
\begin{enumerate}[label=(\roman*)]
	\item we choose $n=200$ and $p=1000$; the features have rank $r_{\bx}=p$, the number of relevant features is $r_{y}=100$ and the true degrees-of-freedom are $r=5$;
	\item we draw a latent dataset $\bQ = (\bq_{1},\ldots,\bq_{n})^\top \in\R^{n\times r}$ as
	\begin{align*}
		\bq_{i} \overset{ind}{\sim} \mN\Big(\bzero_{r}, \diag(\bsigma_{\bq}^2)\Big) \in \R^{r},\quad 5 = (\bsigma_{\bq})_{1} > \ldots > (\bsigma_{\bq})_{r} = 1,
	\end{align*}
	which implies well-posed latent features $\bq$ with $\kappa_{2}(\bSigma^{1/2}_{\bq})=5$;
	\item we draw the latent relevant dataset $\bQ_{y} = (\bq_{y,1},\ldots,\bq_{y,n})^\top \in\R^{n\times r_{y}}$ as
	\begin{align*}
		\bq_{y,i} | \bq_{i} \overset{ind}{\sim} \mN\left( 
		\bigg(\begin{matrix}
			\bq_{i} \\
			\bzero
		\end{matrix}\bigg)
		, \diag(\bsigma_{0}^2) \right) \in \R^{r_{y}},\quad 10^{-1} = (\bsigma_{0})_{1} > \ldots > (\bsigma_{0})_{r_{y}} = 10^{-6},
	\end{align*}
	which implies ill-posed relevant latent features $\bq_{y}$ with $\kappa_{2}(\bSigma^{1/2}_{\bq_{y}}) \simeq 5 \cdot 10^{6}$;
	\item we draw the irrelevant latent dataset $\bQ_{y^\bot} = (\bq_{y^\bot,1},\ldots,\bq_{y^\bot,n})^\top \in\R^{n\times (p-r_{y})}$ as
	\begin{align*}
		\bq_{y^\bot,i} \overset{ind}{\sim} \mN\Big(\bzero, \diag(\bsigma_{y^\bot}^2) \Big) \in \R^{p-r_{y}},\quad 10 = (\bsigma_{y^\bot})_{1} > \ldots > (\bsigma_{y^\bot})_{r_{y}} = 10^{-6},
	\end{align*}
	which implies ill-posed irrelevant latent features $\bq_{y^\bot}$ with $\kappa_{2}(\bSigma^{1/2}_{\bq_{y^\bot}}) \simeq 10^{7}$;
	\item with deterministic orthonormal matrix $\bU\in\R^{p\times p}$, we assemble the observed dataset $\bX = (\bx_{1},\ldots,\bx_{n})^\top \in\R^{n\times p}$
	\begin{align*}
		\bX &= \left(\bQ_{y}\ |\ \bQ_{y^\bot}\right) \bU^\top \in \R^{n\times p},
	\end{align*}
	the eigenvalues decay exponentially and result in a small effective rank $\rho_{\bx} = \Tr(\bSigma)/\|\bSigma\|_{op} \simeq 2$;
	\item with deterministic $\balpha = (1,2,\ldots,r)^\top\in\R^{r}$ we draw the observed response vector $\by = (y_{1},\ldots,y_{n})^\top\in\R^n$ as
	\begin{align*}
		\quad y_{i} | \bq_{i} \overset{ind}{\sim} \mN\left(\bq_{i}^\top\balpha, 1\right) \in \R ;
	\end{align*}
	\item with $\bP = \bU \bI_{r_{y},p} \bI_{r,r_{y}} \in\R^{p\times r}$ we obtain the oracle coefficients $\bbeta = \bP\balpha\in\R^p$.
\end{enumerate}
In this setup, the following population model holds for the pair $(\bx,y)$. For some well-posed latent vector $\bq\in\R^{r}$, one finds both $\E(y|\bq) = \bq^\top\balpha$ with $\balpha = \bbeta_{\ls}(\bq,y)$, and $\bx = \bx_{y} \oplus \bx_{y}$ with $\E(\bx_{y}|\bq) = \bP\bq$. The vector $\bbeta = \bP\balpha$ is the oracle vector of coefficients. With $\bP_{y} = \bU \bI_{r_{y},p} \in\R^{p\times r_{y}}$, the covariance matrix of the features is $\bSigma = \bP_{y}\{\bI_{r,r_{y}}\bSigma_{\bq}\bI_{r,r_{y}}^\top + \sigma_{0}^2\bI_{r_{y}}\}\bP_{y}^\top \oplus \bSigma_{\bx_{y^\bot}}$. Similarly, the covariance vector with the response is $\bsigma = \bP\bsigma_{\bq_{y},y}$. With $\tau=5$, the best $5$-identifiable subspace is $\mB_{r} = \mR(\bP\bSigma_{\bq}\bP^\top) = \mR(\bP\bP^\top)$ since $\kappa_{2}(\bSigma^{1/2}_{r}) \simeq 4.98 < 5$. The true degrees-of-freedom are $r = \Tr(\bP\bP^\top) = 5$ and best $5$-identifiable parameter is then $\bbeta_{r} = \bbeta_{\ls}(\bP\bP^\top\bx,y)$. Both the oracle parameter $\bbeta$ and the identifiable parameter $\bbeta_{r}$ belong to the same space $\mB_{r}$. From Theorem~\ref{thm:ls.pert}, the error between the two is no larger than $\|\bbeta_{r}-\bbeta\|_{2}/\|\bbeta\|_{2} \leq 5\cdot\sigma_{0}^2 = 5\cdot10^{-2}$. In what follows, we take $\bbeta$ as the parameter of interest.

\begin{figure}[ht]
	\centering
	\includegraphics[width=1\textwidth]{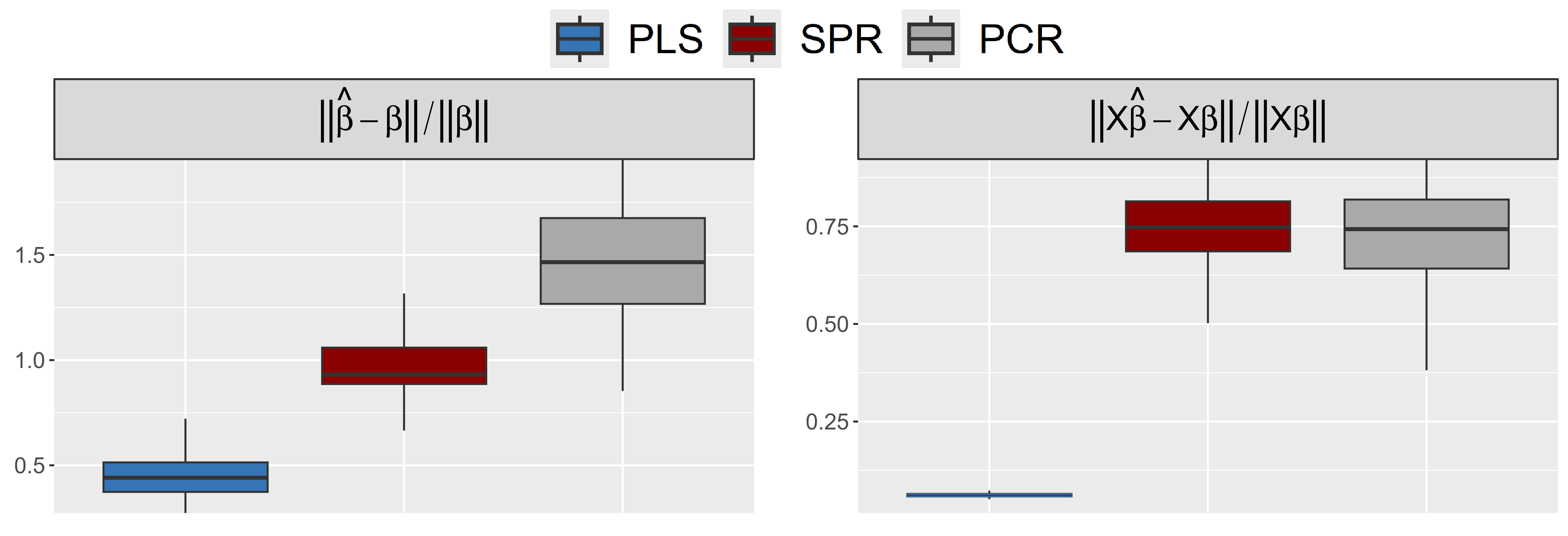}
	\vspace{-0.5cm}
	\caption{Performance of PLS (blue), SPR (red) and PCR (gray). The relative estimation error $\|\wh{\bbeta}_{s^*}-\bbeta\|_{2}/\|\bbeta\|_{2}$ and the relative approximation error $\|\bX\wh{\bbeta}_{s^*}-\bX\bbeta\|_{2}/\|\bX\bbeta\|_{2}$ of the sample estimators using oracle degrees-of-freedom $s^*$ of the oracle parameter $\bbeta$.}
	\label{fig:n>p}
\end{figure}

Over $K=500$ repetitions, we compute estimators $\wh{\bbeta}_{s^*}$ using different policies for dimensionality reduction with the oracle choice of degrees-of-freedom $s^*=r_{s}=5$. For each method, we compute the relative estimation error $\|\wh{\bbeta}_{s^*}-\bbeta\|_{2}/\|\bbeta\|_{2}$ and the relative approximation error $\|\bX\wh{\bbeta}_{s^*} - \bX\bbeta\|_{2}/\|\bX\bbeta\|_{2}$. We compare in Figure~\ref{fig:n>p} the performance of the following algorithms. For sufficient reduction, we use Partial Least Squares (PLS) implemented in \texttt{R} with \texttt{kernelpls.fit()} by \cite{Dayal1997par}. For unsupervised reduction, we use Principal Components Regression (PCR) implemented in \texttt{R} with \texttt{prcomp()} in the \texttt{stats} package. For sparse regression (SPR), we use Elastic Net with data-driven tuning-parameter optimization, implemented in \texttt{R} with \texttt{glmnet()} by \cite{Zou2005reg}. The latter solves a least-squares problem with penalization $\gamma\|\bbeta\|_{1}+(1-\gamma)\|\bbeta\|_{2}^2$ for some $\gamma\in[0,1]$ to be optimized. In this setting, we assume that the degrees of freedom $s^*$ of the problem are known. Despite this oracle knowledge, the bias of sparse and unsupervised methods is very large, since the problem is not sparse (though low-dimensional) and contains irrelevant features with large variation. The figure confirms our theoretical findings that PLS is superior when the goal is to estimate an interpretable vector of coefficients.

\subsection{Real Data}\label{sec:num-real}
We revisit the findings of \cite{krivobokova2012}, who analyzed data from MD simulations of the yeast aquaporin (Aqy1), the gated water channel of the yeast \textit{Pichia pastoris}. The data are given as Euclidean coordinates of $N=783$ atoms, thus $p=783\times 3 = 2349$ features, of Aqy1 observed in a 100-nanosecond time frame, split into $n=20000$ equidistant observations. Additionally, the diameter of the channel $y_{i}$ is measured by the distance between the centers of mass of certain residues of the protein $\bx_{i}$. We take the first half of the data as the training set $(\bX_{train},\by_{train})$ and the remaining half as the test set $(\bX_{test},\by_{test})$, each consisting of $n/2=10000$ observations. Since the data has been produced via molecular dynamics simulations, the observations $(\bx_{i},y_{i})$, $i=1,\ldots,n$, are not independent nor identically distributed, and \cite{singer2016partial} shows that PLS estimates might be inconsistent if one does not account for this dependence. We thus normalize the training data using an estimated temporal covariance matrix $\wh{\bSigma}\in\R^{n\times n}$ computed as in \cite{klockmann2023}. That is, we use $(\wt{\bX}_{train},\wt{\by}_{train})$ with $\wt{\bX}_{train}=\wh{\bSigma}^{-1/2}\bX_{train}$ and $\wt{\by}_{train}=\wh{\bSigma}^{-1/2}\by_{train}$. The rescaled dataset $(\wt{\bX}_{train},\wt{\by}_{train})$ is now treated as i.i.d., and we use it to estimate the following aspects of the problem. The condition number of the sample covariance matrix is $\kappa_{2}(\wt{\bSigma}_{\bx}) \simeq 1.3 \cdot 10^9$ whereas the effective rank is $\Tr(\wt{\bSigma}_{\bx})/\|\wt{\bSigma}_{\bx}\|_{op} \simeq 1$. Recall that the features are full-rank so $\rank(\wt{\bSigma}_{\bx}) = p = 2349$. This is a stark discrepancy suggesting a low-dimensional latent factor structure.   

\begin{figure}[ht]
	\centering
	\includegraphics[width=\textwidth]{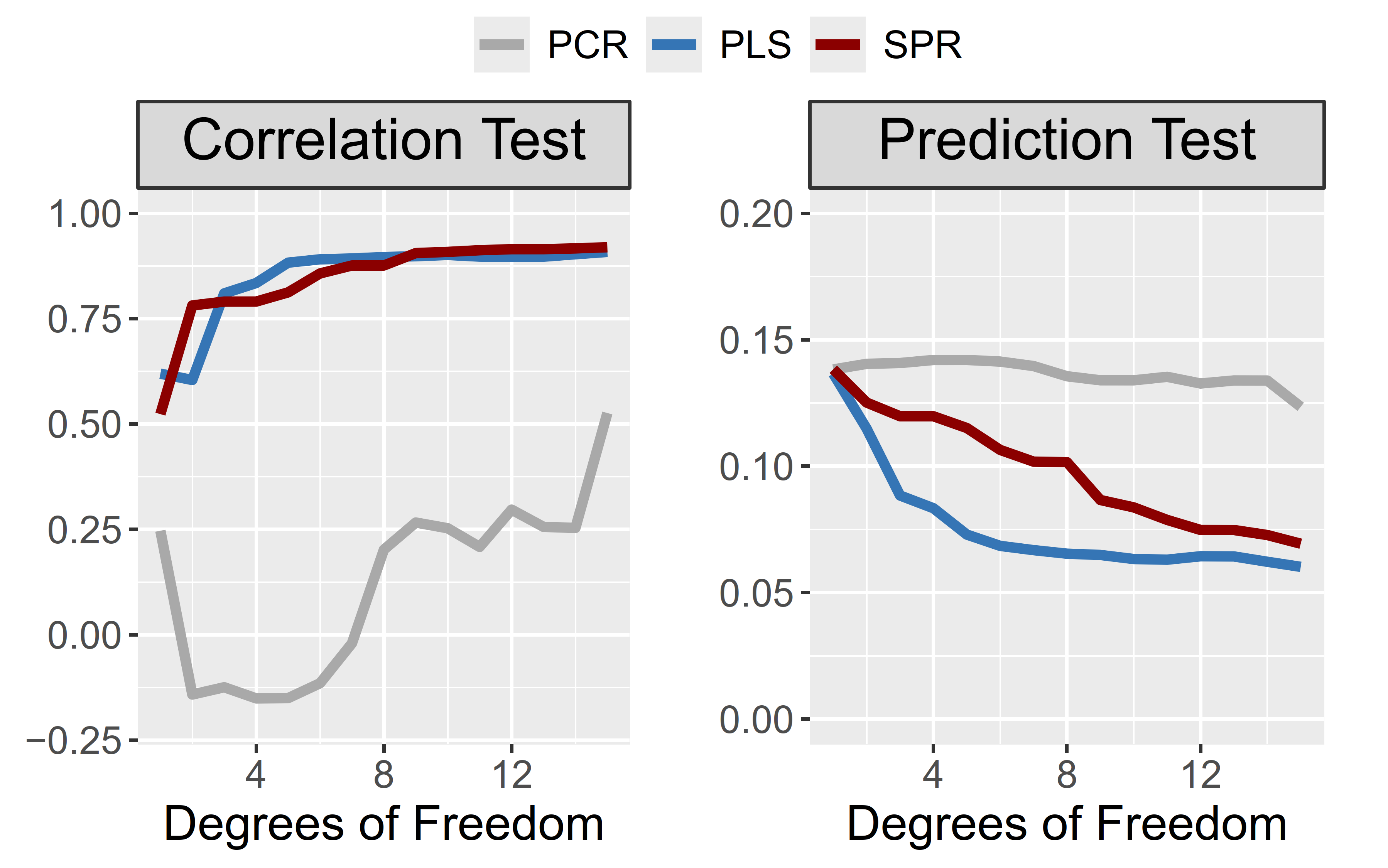}
	\caption{Comparison between PCR, PLS and SPR on the Aqy1 dataset studied by \cite{krivobokova2012} rescaled according to \cite{klockmann2023}. Correlation $\operatorname{cor}(\wh{\by}_{s},\by_{test})$ between estimated and true response on test data and relative prediction error $\|\wh{\by}_{s}-\by_{test}\|_{2}/\|\by_{test}\|_{2}$ between estimated and true response on test data.}
	\label{fig:aqy1-scale}
\end{figure}

From the training set, we compute PLS, PCR, and SPR estimators $\wh{\bbeta}_{s}$ for all degrees-of-freedom $s=1,\ldots,15$. We also compute the estimated condition number $\wh{\kappa}_{s}$ of the reduced sample covariance matrix on the training set. To evaluate the models, we compute the correlation between the estimated responses $\wh{\by}_{s} = \bX_{test} \wh{\bbeta}_{s}$ and the observed response $\by_{test}$ on the test set, together with the relative prediction error $\|\wh{\by}_{s}-\by_{test}\|_{2}/\|\by_{test}\|_{2}$. Figure~\ref{fig:aqy1-scale} shows that PLS outperforms PCR in terms of correlation and prediction on the test data, even with $s=15$, the correlation induced by PCR barely reaches $50\%$, whereas that of PLS is essentially $90\%$. Figure~\ref{fig:aqy1-scale} also shows that PLS outperforms SPR in terms of prediction, but not correlation. Not shown in the picture, the condition number for SPR at $s=15$ is $\wh{\kappa}_{\spr,15} \simeq 10^{2.25}$, which is larger than $\wh{\kappa}_{\pls,15} = \wh{\kappa}_{\pcr,15} \simeq 10^{2}$ obtained by PLS/PCR. Furthermore, we find that $s=4$ gives the largest number of degrees of freedom, where all methods have comparable condition numbers. We then use $\wh{\bbeta}_{\pls,4}$ as the benchmark parameter and find both $\|\wh{\bbeta}_{\spr,4}-\wh{\bbeta}_{\pls,4}\|_{2}/\|\wh{\bbeta}_{\pls,4}\|_{2} \simeq 1 \simeq \|\wh{\bbeta}_{\pcr,4}-\wh{\bbeta}_{\pls,4}\|_{2}/\|\wh{\bbeta}_{\pls,4}\|_{2}$, suggesting that the other methods are affected by a large bias.

\section{Discussion}\label{sec:disc}
In this work, we presented a general framework to address ill-posedness in linear regression problems. In particular, we defined identifiable parameters and statistically interpretable algorithms in this setting. We showed that only statistically interpretable algorithms have a negligible population error, while the sample error is driven by the degree of ill-posedness of the problem. \\
In the current era of artificial intelligence (AI) and machine learning (ML), sophisticated black-box algorithms have been developed to achieve the best possible prediction on unseen samples. Such flexible AI/ML algorithms tend to use all relevant information in the features and are rarely affected by ill-posedness in the features. In fact, such algorithms even leverage correlated features to improve prediction. 
At the same time, a vast body of research attempts to address the interpretability of AI/ML algorithms. Among the most prominent general-purpose methods are LIME (Local Interpretable Model-Agnostic Explanations) of \citet{Ribeiro2016}, which explains individual predictions by locally approximating the model with an interpretable surrogate, SHAP (SHapley Additive exPlanations) of \citet{Lundberg2017}, which approximates the contribution of each feature to a prediction, either locally or globally, and AGOP (Average Gradient Outer Product) by \citet{Radhakrishnan2024}, which assigns higher importance to features with larger gradients.  While SHAP is often regarded as one of the most general and mathematically principled interpretability methods grounded in cooperative game theory, its practical implementations are well known to fail when features are (highly) correlated. For example, \citet{Salih2025} demonstrated that both SHAP and LIME produce misleading attributions, particularly in the presence of collinear features, and are also dependent on the algorithm used. Also, AGOP is designed to completely ignore possible correlations among the features. Hence, the correlated features, which actually strengthen the prediction, hinder the interpretability of general algorithms. We believe that our framework can offer a way to treat statistical interpretability of broader classes of machine learning methods.

\appendix
\section{Reduction Algorithms for Least-Squares Problems}\label{sec:rls}
In this section we formalize our broad class of dimensionality reduction algorithms in terms of arbitrary matrices $\bA\in\R_{\succeq0}^{p\times p}$ and vectors $\bb\in\mR(\bA)$. One recovers the population reduction algorithms studied in Section~\ref{sec:lr:pop} when the matrix-vector pair consists of population moments $(\bSigma,\bsigma)$,  or the sample reduction algorithms studied in Section~\ref{sec:lm:sam} when the matrix-vector pair consists of sample moments $(\wh{\bSigma},\wh{\bsigma})$. Either way, the algorithms are deterministic. We also formalize a notion of stability with respect to small perturbations that is required to obtain general perturbation bounds. Such notion of of stability is measured with respect to the principal-angle between linear subspaces and the operator norm for orthogonal projection matrices. Before stating our main results, we recall the relevant facts from the numerical literature.

\subsection{Principal-Angle for Compatible Linear Subspaces} \label{sec:rls:dist} 
Two linear subspaces $\mE\subseteq\R^p$ and $\wt{\mE}\subseteq\R^p$ are orthogonal, denoted by $\mE\bot\wt{\mE}$, if $\be^\top\wt{\be}=0$ for all $\be\in\mE, \wt{\be}\in\wt{\mE}$. We say that two linear subspaces $\mE$ and $\wt{\mE}$ of $\R^p$ are compatible if they have the same dimension $1\leq d\leq p$. If $\bE=[\be_{1}|\cdots|\be_{d}] \in\R^{p\times d}$ is any orthonormal basis of $\mE$ and $\wt{\bE}=[\wt{\be}_{1}|\cdots|\wt{\be}_{d}] \in\R^{p\times d}$ is any orthonormal basis of $\wt{\mE}$, the singular values of $\bE^\top\wt{\bE} \in\R^{d\times d}$ do not depend on the choice of orthonormal bases and can be written as $\bsigma(\wt{\mE},\mE) = (\sigma_{1}(\wt{\mE},\mE),\ldots,\sigma_{d}(\wt{\mE},\mE))^\top\in\R^d$. Since the vectors are orthonormal, one finds $0\leq \sigma_{i}(\wt{\mE},\mE)\leq 1$ for all $1\leq i\leq d$ and can uniquely define the angle $0\leq \phi_{i}(\wt{\mE},\mE)\leq \pi/2$ between $\be_{i}$ and $\wt{\be}_{i}$ by $\phi_{i}(\wt{\mE},\mE) = \arccos(\sigma_{i}(\wt{\mE},\mE))$. We denote the collection of angles between compatible linear subspaces as
\begin{align}\label{eq:subsp.angles}
	\frac{\pi}{2} \geq \phi_{1}(\wt{\mE},\mE) \geq \ldots \geq \phi_{d}(\wt{\mE},\mE) \geq 0.
\end{align}
A detailed reference on angles is given by \cite{davis1970rot}. With $\bU_{\mE}$ and $\bU_{\wt{\mE}}$ respectively the orthogonal projections of $\R^p$ onto $\mE$ and $\wt{\mE}$, it is a classical result, for a proof see Theorem~5.1 by \cite{godunov1993gua}, that 
\begin{align}\label{eq:subsp.dist.angle}
	\phi_{1}(\wt{\mE},\mE) = \arcsin\left(\|\bU_{\wt{\mE}}-\bU_{\mE}\|_{op}\right),
\end{align}
which we refer to as the {\it principal-angle}. Following Section~8.6 by \cite{berger1987geo} one can prove the following properties of the principal-angle.

\begin{lemma}\label{lem:subsp.dist}
	The principal-angle $\phi_{1}(\cdot,\cdot)$ in Equation~\eqref{eq:subsp.dist.angle} satisfies the following:
	\begin{enumerate}[label=(\roman*),itemsep=0.25em,topsep=0.25em]
		\item $0 \leq \phi_{1}(\wt{\mE},\mE) = \phi_{1}(\mE,\wt{\mE}) \leq \frac{\pi}{2}$;
		\item $\phi_{1}(\wt{\mE},\mE) = 0 \iff \wt{\mE}=\mE$, whereas $\wt{\mE}\bot\mE \implies \phi_{1}(\wt{\mE},\mE) = \frac{\pi}{2}$;
		\item $\phi_{1}(\wh{\mE},\wt{\mE}) + \phi_{1}(\wt{\mE},\mE) \leq \frac{\pi}{2} \implies \phi_{1}(\wh{\mE},\mE) \leq \phi_{1}(\wh{\mE},\wt{\mE}) + \phi_{1}(\wt{\mE},\mE)$.
	\end{enumerate}
\end{lemma}

\subsection{Reduced Least-Squares Problems} \label{sec:rls:ls}
For any integers $d,p\geq1$, any matrix $\bA\in\R^{d\times p}$, any vector $\bb\in\R^d$ and any linear subspace $\mC\subseteq\R^p$, we denote $\ls(\bA,\bb,\mC) = \argmin_{\bzeta\in\mC}\|\bA\bzeta-\bb\|_{2}^2$ the set of solutions to the reduced least-squares problem. It is a classical result, see Theorem~4 by \cite{price1964mat}, that $\ls(\bA,\bb,\R^p) = \{\bA^\dagger\bb + (\bI_{p}-\bA^\dagger\bA)\bzeta : \bzeta\in\R^p\}$ and the minimum-$L^2$-norm solution $\bzeta_{\ls} = \bA^\dagger\bb\in\R^p$ belongs to the range $\mR(\bA^\dagger)\subseteq\R^p$. This means that restricting the least-squares problem to the range of its inverse operator always admits unique solution, that is, $\ls(\bA,\bb) = \ls(\bA,\bb,\mR(\bA^\dagger)) = \{\bA^\dagger\bb\}$. Furthermore, one can always replace the vector $\bb$ with its projection $\bA\bA^\dagger\bb$ onto the range $\mR(\bA^\dagger)$, since $\ls(\bA,\bA\bA^\dagger\bb) = \{\bA^\dagger\bb\}$ and $\bA\bA^\dagger\bb=\bb$ if and only if $\bb\in\mR(\bA^\dagger)$. Lemma~1 by \cite{penrose1955gen} provides the equivalent formulation $\bA^\dagger = (\bA^\top\bA)^\dagger\bA^\top$, so that $\bzeta_{\ls}=(\bA^\top\bA)^\dagger\bA^\top\bb$ is the unique solution of the equivalent least-squares problem $\ls(\bA^\top\bA,\bA^\top\bb)$ restricted to $\mR((\bA^\top\bA)^\dagger)=\mR(\bA^\top\bA)$ since $\bA^\top\bA\in\R^{p\times p}$ is symmetric and positive-semidefinite, see Theorem~20.5.1 by \cite{harville1997mat}.   

\begin{lemma}\label{lem:ls.C}
	For any integers $d,p\geq1$, any matrix $\bA\in\R^{d\times p}$, any vector $\bb\in\R^d$ and any linear subspace $\mC\subseteq\mR(\bA^\dagger)$, the reduced least-squares problem $\ls(\bA,\bb,\mC)$ admits unique solution $\bzeta_{\mC}=\bU_{\mC}\bzeta_{\ls}$ where $\bU_{\mC}\in\R^{p\times p}$ is the orthogonal projection of $\R^p$ onto $\mC$ and $\bzeta_{\ls}=\bA^\dagger\bb$ the minimum-$L^2$-norm solution of the unreduced least-squares problem. The reduced least-squares problem is equivalent to $\ls(\bU_{\mC}\bA^\top\bA\bU_{\mC},\bU_{\mC}\bA^\top\bb)$ and $\bzeta_{\mC}=(\bU_{\mC}\bA^\top\bA\bU_{\mC})^\dagger\bU_{\mC}\bA^\top\bb$.
\end{lemma}
\begin{proof}[Proof of Lemma~\ref{lem:ls.C}]
	Recall that $\bA^\dagger\bA$ is the orthogonal projection of $\R^p$ onto $\mR(\bA^\dagger)$, therefore $\bU_{\mC}\bA^\dagger\bA=\bA^\dagger\bA\bU_{\mC}=\bU_{\mC}$, since $\mR(\bU_{\mC})\subseteq\mR(\bA^\dagger)$. This implies
	\begin{align*}
		\bA\bU_{\mC} (\bU_{\mC}\bA^\dagger) \bA\bU_{\mC} = \bA\bU_{\mC} \bU_{\mC}\bA^\dagger \bA\bU_{\mC} = \bA\bU_{\mC},
	\end{align*}
	so that $(\bA\bU_{\mC})^\dagger=\bU_{\mC}\bA^\dagger$. By definition, $\bU_{\mC}\bzeta=\bzeta$ if and only if $\bzeta\in\mR(\bU_{\mC})$ and $\bU_{\mC}\bzeta\in\mR(\bU_{\mC})$ for all $\bzeta\in\R^p$. Thus, we can write
	\begin{align*}
		\mR(\bU_{\mC}) = \{\bzeta\in\R^p : \bU_{\mC}\bzeta=\bzeta\} = \{\bU_{\mC}\bzeta : \bzeta\in\R^p\}.
	\end{align*}
	With the above, we can now infer
	\begin{align*}
		\ls(\bA,\bb,\mC)
		&= \argmin_{\bzeta\in\mR(\bU_{\mC})} \|\bA\bzeta-\bb\|_{2}^2 \\
		&= \argmin_{\bzeta\in\mR(\bU_{\mC})} \|\bA\bU_{\mC}\bzeta-\bb\|_{2}^2 \\ 
		&= \bU_{\mC} \cdot \argmin_{\bzeta\in\R^p} \|\bA\bU_{\mC}\bzeta-\bb\|_{2}^2 \\
		&= \bU_{\mC} \cdot \{(\bA\bU_{\mC})^\dagger\bb + (\bI_{p}-(\bA\bU_{\mC})^\dagger(\bA\bU_{\mC}))\bzeta : \bzeta\in\R^p\} \\
		&= \bU_{\mC} \cdot \{\bU_{\mC}\bA^\dagger\bb + (\bI_{p}-\bU_{\mC}\bA^\dagger\bA\bU_{\mC})\bzeta : \bzeta\in\R^p\} \\
		&= \bU_{\mC} \cdot \{\bU_{\mC}\bA^\dagger\bb + (\bI_{p}-\bU_{\mC})\bzeta : \bzeta\in\R^p\} \\
		&= \{\bU_{\mC}\bA^\dagger\bb\}.
	\end{align*}
	This shows that $\bzeta_{\mC}=\bU_{\mC}\bzeta_{\ls}$ with $\bzeta_{\ls}=\bA^\dagger\bb$. To conclude the proof, we notice that the equivalent formulation
	\begin{align*}
		\bzeta_{\mC} = (\bA\bU_{\mC})^\dagger\bb = (\bU_{\mC}\bA^\top\bA\bU_{\mC})^\dagger\bU_{\mC}\bA^\top\bb
	\end{align*}
	is the unique solution of the equivalent least-squares problem $\ls(\bU_{\mC}\bA^\top\bA\bU_{\mC},\bU_{\mC}\bA^\top\bb)$.
\end{proof}

\subsection{Reduction Algorithms} \label{sec:rls:alg}
\begin{definition}\label{def:reg.alg}
	Given $\bA\in\R_{\succeq0}^{p\times p}$ and $\bb\in\mR(\bA)$, let $\alg(\cdot)$ be any function that selects:
	\begin{enumerate}[label=(\roman*),itemsep=0.25em,topsep=0.25em]
		\item a non-decreasing sequence $\mC_{\alg,0}\subseteq\cdots\subseteq\mC_{\alg,p}\subseteq\R^p$ of linear subspaces having dimensions $r_{\alg,s}=\dim(\mC_{\alg,s}) \leq s$ for all $0\leq s\leq p$;  \label{def:reg.alg.subsp}
		\item the orthogonal projections $\bU_{\alg,s}\in\R^{p\times p}$ of $\R^p$ onto $\mC_{\alg,s}$, the projected matrices $\bA_{\alg,s}=\bU_{\alg,s}\bA\bU_{\alg,s}\in\R_{\succeq0}^{p\times p}$ and the projected vectors $\bb_{\alg,s}=\bU_{\alg,s}\bb\in\mR(\bA_{\alg,s})$; \label{def:reg.alg.proj}
		\item the minimum-$L^2$-norm solutions $\bzeta_{\alg,s} = \bA_{\alg,s}^\dagger\bb_{\alg,s}$ of the reduced least-squares problems $\ls(\bA_{\alg,s},\bb_{\alg,s})$; \label{def:reg.alg.proj.ls}
		\item when $s=p$, one recovers $\bzeta_{\alg,p}=\bzeta_{\ls}$ the solution of the unreduced least-squares problem $\ls(\bA,\bb)$. \label{def:reg.alg.ls}
	\end{enumerate}
	We call {\it reduction algorithm} the collection
	\begin{align}\label{eq:reg.alg}
		\alg(\bA,\bb) = \left\{\btheta_{\alg,s} = \Big(\mC_{\alg,s},\ \bU_{\alg,s},\ r_{\alg,s},\ \bzeta_{\alg,s} \Big) : 0\leq s\leq p \right\},
	\end{align}
	determined by the above choices.
\end{definition}

\begin{lemma}\label{lem:reg.alg.def}
	One can replace Definition~\ref{def:reg.alg}~\ref{def:reg.alg.proj.ls} with the following equivalent definitions:
	\begin{align*}
		\bzeta_{\alg,s} &= \bA_{\alg,s}^\dagger\bb_{\alg,s} \in \ls\big(\bA_{\alg,s},\bb_{\alg,s}\big), \\
		\bzeta_{\alg,s} &= (\bA^{\frac{1}{2}}\bU_{\alg,s})^\dagger\bA^{\frac{\dagger}{2}}\bb \in \ls\big(\bA^{\frac{1}{2}},\bA^{\frac{\dagger}{2}}\bb,\mR(\bU_{\alg,s})\big), \\
		\bzeta_{\alg,s} &= \bU_{\alg,s} \bzeta_{\ls} \in \bU_{\alg,s} \cdot \ls\big(\bA,\bb\big).
	\end{align*}
\end{lemma}
\begin{proof}[Proof of Lemma~\ref{lem:reg.alg.def}]
	A reduction algorithm $\alg(\bA,\bb)$ satisfies Definition~\ref{def:reg.alg}~\ref{def:reg.alg.proj.ls} if and only if, for all $1\leq s\leq p$, one has $\bzeta_{\alg,s} = \bA_{\alg,s}^\dagger\bb_{\alg,s} \in \ls(\bA_{\alg,s},\bb_{\alg,s})$. With the equivalent formulation of generalized inverse,
	\begin{align*}
		\bzeta_{\alg,s} = (\bU_{\alg,s}\bA\bU_{\alg,s})^\dagger\bU_{\alg,s}\bb = (\bU_{\alg,s}\bA^{\frac{1}{2}} \bA^{\frac{1}{2}}\bU_{\alg,s})^\dagger\bU_{\alg,s}\bA^{\frac{1}{2}}\bA^{\frac{\dagger}{2}}\bb = (\bA^{\frac{1}{2}}\bU_{\alg,s})^\dagger\bA^{\frac{\dagger}{2}}\bb
	\end{align*}
	is the unique solution of the equivalent least-squares problem
	\begin{align*}
		\ls(\bA^{\frac{1}{2}}\bU_{\alg,s},\bA^{\frac{\dagger}{2}}\bb,\mR(\bU_{\alg,s})) = \argmin_{\bzeta\in\mR(\bU_{\alg,s})} \|\bA^{\frac{1}{2}}\bU_{\alg,s}\bzeta-\bA^{\frac{\dagger}{2}}\bb\|_{2}^2 = \ls(\bA^{\frac{1}{2}},\bA^{\frac{\dagger}{2}}\bb,\mR(\bU_{\alg,s})).
	\end{align*}
	We now invoke Lemma~\ref{lem:ls.C} to infer that $\bzeta_{\alg,s} = \bU_{\alg,s} (\bA^{\frac{1}{2}})^\dagger\bA^{\frac{\dagger}{2}}\bb = \bU_{\alg,s} \bzeta_{\ls}$ is the unique solution of the equivalent least-squares problem $\bU_{\alg,s} \cdot \ls(\bA,\bb)$.
\end{proof}

\begin{remark}[Degrees-of-Freedom] \label{rem:reg.alg.dof}
	Notice that the parameter $\btheta_{\alg,s}$ defined in Equation~\eqref{eq:reg.alg} is redundant in the sense that it is uniquely determined by $\bA$, $\bb$ and $\mC_{\alg,s}$. This means that two parameters $\btheta_{\alg,s}$ and $\btheta_{\alg,s'}$ are identical if and only if they have the same dimension $r_{\alg,s}=r_{\alg,s'}$. 
	
	For any parameter $\btheta_{\alg} = (\mC_{\alg}, \bU_{\alg}, r_{\alg}, \bzeta_{\alg})$ computed from a reduction algorithm $\alg(\bA,\bb)$, we define its {\it degrees-of-freedom} to be the dimension of its subspace, namely $\dof(\btheta_{\alg}) = r_{\alg}$, whereas the set $\dof(\alg(\bA,\bb)) = \{\dof(\btheta_{\alg}) : \btheta_{\alg}\in\alg(\bA,\bb)\}$ contains all the degrees-of-freedom without repetition. By construction, for each $r \in \dof(\alg(\bA,\bb))$ there exists a unique parameter $\btheta_{\alg}^{(r)}\in\alg(\bA,\bb)$ such that $\dof(\btheta_{\alg}^{(r)})=r$.
	
	We say that a reduction algorithm $\alg(\bA,\bb)$ \textit{preserves the degrees-of-freedom} of the least-squares problem $\ls(\bA,\bb)$ if the parameter $\btheta_{\alg,p}$ consists of a linear subspace $\mC_{\alg,p}\subseteq\mR(\bA)$. This implies that the non-decreasing sequence of linear subspaces in Definition~\ref{def:reg.alg}~\ref{def:reg.alg.subsp} is contained as a whole in the range $\mR(\bA)$ and that no more than $\dof(\btheta_{\alg,p}) = r_{\alg,p} \leq \rank(\bA)$ are used.
\end{remark}

\begin{remark}[Examples of Reduction Algorithms] \label{rem:reg.alg.examples}
	In view of Remark~\ref{rem:reg.alg.dof}, in order for $\alg(\bA,\bb)$ to be a reduction algorithm, it is sufficient to provide a choice function $\alg(\cdot)$ that satisfies Condition~\ref{def:reg.alg.subsp} and Condition~\ref{def:reg.alg.ls} in Definition~\ref{def:reg.alg}.
	
	Principal Components Regression by \cite{hotelling1933anal} can be translated into a choice function $\pcr(\cdot)$ and reduction algorithm $\pcr(\bA,\bb)$ with the following properties. For all $0\leq s\leq p$, the choice of the linear subspaces $\mC_{\pcr,s}$ depends only on the matrix $\bA$ and not on the vector $\bb$. In particular, with $1\leq r_{\bA}=\rank(\bA)\leq p$ and $\lambda_{\bA,1}\geq\cdots\geq\lambda_{\bA,r_{\bA}}>0$ the sorted positive eigenvalues of $\bA$, $\mC_{\pcr,s} = \mV_{s}(\bA)$ is the span of eigenvectors $\{\bv_{1}(\bA),\ldots,\bv_{s}(\bA)\}$ of $\bA$ corresponding to the largest $s$ eigenvalues. When $s=p$, we find that $\mC_{\pcr,p} = \mV_{p}(\bA) = \mR(\bA)$ recovers exactly the range of $\bA$ so that $\bzeta_{\pcr,p}=\bzeta_{\ls}$.
	
	Partial Least Squares by \cite{Wold66non} can be translated into a choice function $\pls(\cdot)$ and reduction algorithm $\pls(\bA,\bb)$ with the following properties. For all $0\leq s\leq p$, the linear subspaces $\mC_{\pls,s} = \mK_{s}(\bA,\bb)$ are spanned by the first $s$ vectors of the Krylov basis $\{\bb,\bA\bb,\ldots,\bA^{s-1}\bb\}$ generated by $\bA$ and $\bb$. The subspace $\mC_{\pls,s}$ is one-dimensional when $\bb$ is an eigenvector of $\bA$. When $s=p$, we find that $\mC_{\pls,p} = \mK_{p}(\bA,\bb) \subseteq \mR(\bA)$ and check in Lemma~\ref{lem:ls.krylov} that $\bzeta_{\pls,p}=\bzeta_{\ls}$. 
	
	Only a variation of Best Subset Selection by \cite{beale1967} falls within our framework. Forward Subset Selection can be translated into a choice function $\fss(\cdot)$ and reduction algorithm $\fss(\bA,\bb)$ with the following properties. For all $0\leq s\leq p$, the linear subspaces $\mC_{\fss,s} = \mE_{J_{\fss,s}}(\bA,\bb)$ are spanned by the vectors of the canonical basis $\{\be_{j}:j\in J_{\fss,s}(\bA,\bb)\}$ corresponding to an active set $J_{\fss,s}(\bA,\bb)\subseteq J_{\ls}(\bA,\bb)$, where $J_{\ls}$ is the active set of $\bzeta_{\ls}$. The active sets are computed iteratively, in the sense that $j_{\fss,s+1}\in J_{\ls}\setminus J_{\fss,s}$.  When $s=p$, we find $\mC_{\fss,p} = \mE_{J_{\ls}}(\bA,\bb)$ and $\bzeta_{\fss,p}=\bzeta_{\ls}$. 
\end{remark}

\subsection{Stability of Regularization Algorithms} \label{sec:rls:alg.pert}
Let $\bA\in\R_{\succeq0}^{p\times p}$ and $\bb\in\mR(\bA)$ be given and consider any reduction algorithm $\alg(\bA,\bb)$ with parameters $\btheta_{\alg,s} = (\mC_{\alg,s},\ \bU_{\alg,s},\ r_{\alg,s},\ \bzeta_{\alg,s})$ for all $0\leq s\leq p$. 

We are interested in the stability of the algorithm with respect to arbitrary perturbations $\wt{\bA}\in\R_{\succeq0}^{p\times p}$ and $\wt{\bb}\in\mR(\wt{\bA})$ for which $\alg(\wt{\bA},\wt{\bb})$ is also a reduction algorithm with parameters $\wt{\btheta}_{\alg,s} = (\wt{\mC}_{\alg,s},\ \wt{\bU}_{\alg,s},\ \wt{r}_{\alg,s},\ \wt{\bzeta}_{\alg,s})$ for all $0\leq s\leq p$. We are only interested in perturbations preserving the degrees-of-freedom, in the sense that $\dof(\alg(\bA,\bb)) \subseteq \dof(\alg(\wt{\bA},\wt{\bb}))$. The latter inclusion means that for any parameter $\btheta_{\alg}\in\alg(\bA,\bb)$ there exists a compatible parameter $\wt{\btheta}_{\alg}\in\alg(\wt{\bA},\wt{\bb})$ such that $\dof(\btheta_{\alg})=\dof(\wt{\btheta}_{\alg})$. In particular, such parameter $\wt{\btheta}_{\alg}$ must be unique due to the fact that $\wt{r}_{\alg,s}=\wt{r}_{\alg,s'} \iff \wt{\btheta}_{\alg,s} = \wt{\btheta}_{\alg,s'}$. Formally, the set of all such perturbations is
\begin{align}\label{eq:A.b.alg.pert}
	\begin{split}
		\Delta_{\alg}(\bA,\bb) = \Big\{(\wt{\bA},\wt{\bb}) \in \R_{\succeq0}^{p\times p} \times \mR(\wt{\bA}) :\ & \dof(\alg(\bA,\bb)) \subseteq \dof(\alg(\wt{\bA},\wt{\bb})) \Big\}
	\end{split}
\end{align}
and the size of a perturbation $(\wt{\bA},\wt{\bb}) \in \Delta_{\alg}(\bA,\bb)$ is
\begin{align}\label{eq:A.b.alg.eps}
	\eps(\wt{\bA},\wt{\bb}) = \frac{\|\wt{\bA}-\bA\|_{op}}{\|\bA\|_{op}} \vee \frac{\|\wt{\bb}-\bb\|_{2}}{\|\bb\|_{2}}.
\end{align}
The perturbation set defined above is non-empty because the trivial perturbation $(\bA,\bb)$ belongs to $\Delta_{\alg}(\bA,\bb)$ and has size $\eps(\bA,\bb) = 0$. For each $(\wt{\bA},\wt{\bb}) \in \Delta_{\alg}(\bA,\bb)$ and each $r \in \dof(\alg(\bA,\bb))$, there exists a unique pair of compatible parameters 
\begin{align}\label{eq:alg.theta.r}
	\btheta_{\alg}^{(r)} = \big(\mC_{\alg}^{(r)},\ \bU_{\alg}^{(r)},\ r,\ \bzeta_{\alg}^{(r)} \big),\quad  
	\wt{\btheta}_{\alg}^{(r)} = \big(\wt{\mC}_{\alg}^{(r)},\ \wt{\bU}_{\alg}^{(r)},\ r,\ \wt{\bzeta}_{\alg}^{(r)} \big),
\end{align}
such that $\btheta_{\alg}^{(r)}\in\alg(\bA,\bb)$ and $\wt{\btheta}_{\alg}^{(r)}\in\alg(\wt{\bA},\wt{\bb})$. 

We quantify the effect of perturbations in terms of both the operator norm $\|\cdot\|_{op}$ and the principal-angle $\phi_{1}(\cdot,\cdot)$. We say that $\alg(\bA,\bb)$ is {\it stable} if, for all $r \in \dof(\alg(\bA,\bb))$,
\begin{align}\label{eq:A.b.alg.C.D}
	\begin{split}
		C_{\alg,r}(\bA,\bb) &= 1 \vee \sup_{\substack{(\wt{\bA},\wt{\bb}) \in \Delta_{\alg}(\bA,\bb)\\ \eps(\wt{\bA},\wt{\bb})>0}}\ \frac{\big\|\wt{\bU}_{\alg}^{(r)} - \bU_{\alg}^{(r)}\big\|_{op}}{\eps(\wt{\bA},\wt{\bb})} < +\infty, \\
		D_{\alg,r}(\bA,\bb) &= 1 \vee \sup_{\substack{(\wt{\bA},\wt{\bb}) \in \Delta_{\alg}(\bA,\bb)\\ \eps(\wt{\bA},\wt{\bb})>0}}\ \frac{ \phi_{1}\big(\wt{\mC}_{\alg}^{(r)}, \mC_{\alg}^{(r)}\big)}{\eps(\wt{\bA},\wt{\bb})} < +\infty.
	\end{split}
\end{align}
Notice that, using $t\leq \arcsin(t)\leq (\pi/2) t$ for all $0\leq t\leq 1$ and the definition of principal-angle $\phi_{1}(\cdot,\cdot)$ in Equation~\eqref{eq:subsp.dist.angle}, it follows immediately that $C_{\alg,r} \leq D_{\alg,r} \leq (\pi/2) C_{\alg,r}$. Therefore, an algorithm is stable as long as either one of the quantities in the above display is finite. For a stable algorithm we denote
\begin{align}\label{eq:A.b.alg.M}
	M_{\alg,r}(\bA,\bb) = 2 \cdot \kappa_{2}(\bA_{\alg}^{(r)}) \cdot \{4\ C_{\alg,r}(\bA,\bb) + 1\} \cdot \left\{\frac{\|\bA\|_{op}}{\|\bA_{\alg}^{(r)}\|_{op}} \vee \frac{\|\bb\|_{2}}{\|\bb_{\alg}^{(r)}\|_{2}}\right\},
\end{align}
for all $r \in \dof(\alg(\bA,\bb))$.

We recall here a classical result established by~\cite{wei1989}.
\begin{theorem}[Theorem~1.1 by~\cite{wei1989}]\label{thm:ls.pert}
	Let $\bzeta_{\ls} = \ls(\bA,\bb)$ be the minimum-$L^2$-norm solution of a least-squares problem with $\bA\in\R^{p\times p}$ some symmetric and positive semi-definite matrix and $\bb\in\mR(\bA)$ some vector. Let $\wt{\bzeta}_{\ls} = \ls(\wt{\bA},\wt{\bb})$ be the minimum-$L^2$-norm solution of a perturbed least-squares problem with $\wt{\bA}=\bA+\wt{\Delta\bA}\in\R^{p\times p}$ some symmetric and positive semi-definite matrix and $\wt{\bb}=\bb+\wt{\Delta\bb}\in\mR(\wt{\bA})$ some vector. Assume that $\rank(\wt{\bA})=\rank(\bA)$ and 
	\begin{align*}
		\frac{\|\wt{\Delta\bb}\|_{2}}{\|\bb\|_{2}} \leq \eps,\quad \frac{\|\wt{\Delta\bA}\|_{op}}{\|\bA\|_{op}} \leq \eps,\quad 0 \leq \eps \leq \frac{1}{2\cdot \kappa_{2}(\bA)}.
	\end{align*}
	Then,
	\begin{align*}
		\frac{\|\wt{\bzeta}_{\ls}-\bzeta_{\ls}\|_{2}}{\|\bzeta_{\ls}\|_{2}} &\leq 5 \cdot \kappa_{2}(\bA) \cdot \eps.
	\end{align*}
\end{theorem}

The following theorem is our main contribution to the theory of reduced least-squares problems. 
\begin{theorem}\label{thm:alg.pert}
	Let $\alg(\bA,\bb)$ be a reduction algorithm in the sense of Definition~\ref{def:reg.alg}. Let $(\wt{\bA},\wt{\bb}) \in \Delta_{\alg}(\bA,\bb)$ be a perturbation of size $\eps = \eps(\wt{\bA},\wt{\bb})$. For all $r \in \dof(\alg(\bA,\bb))$, let $\btheta_{\alg}^{(r)}\in\alg(\bA,\bb)$ and $\wt{\btheta}_{\alg}^{(r)}\in\alg(\wt{\bA},\wt{\bb})$ be the compatible parameters in Equation~\eqref{eq:alg.theta.r}. If the reduction algorithm is stable with constants $C_{\alg,r}$, $D_{\alg,r}$ from Equation~\eqref{eq:A.b.alg.C.D}, then
	\begin{align*}
		\big\|\wt{\bU}_{\alg}^{(r)} - \bU_{\alg}^{(r)} \big\|_{op} \leq C_{\alg,r} \cdot \eps,\quad 
		\phi_{1}\big(\wt{\mC}_{\alg}^{(r)}, \mC_{\alg}^{(r)} \big) \leq D_{\alg,r} \cdot \eps.
	\end{align*}
	Furthermore, if $M_{\alg,r} \cdot \eps < 1$ with the constant from Equation~\eqref{eq:A.b.alg.M}, then
	\begin{align*}
		\frac{\big\|\wt{\bzeta}_{\alg}^{(r)}-\bzeta_{\alg}^{(r)} \big\|_{2}}{\big\|\bzeta_{\alg}^{(r)} \big\|_{2}} &\leq \frac{5}{2} \cdot M_{\alg,r} \cdot \eps.
	\end{align*}
\end{theorem}
\begin{proof}[Proof of Theorem~\ref{thm:alg.pert}]
	Equation~\eqref{eq:A.b.alg.C.D} defines $C_{\alg,r}$ as the smallest constant $C$ such that $\|\wt{\bU}_{\alg}^{(r)} - \bU_{\alg}^{(r)}\|_{op} \leq C \eps(\wt{\bA},\wt{\bb})$ and $D_{\alg,r}$ as the smallest constant $D$ such that $ \phi_{1}(\wt{\mC}_{\alg}^{(r)}, \mC_{\alg}^{(r)}) \leq D \eps(\wt{\bA},\wt{\bb})$ for all perturbations $(\wt{\bA},\wt{\bb})\in\Delta_{\alg}(\bA,\bb)$. It remains to prove the perturbation bound on the least-squares solutions. By definition of orthogonal projection, we have $\bU_{\alg}^{(r)} = \bU_{\alg}^{(r)}{}^\top = \bU_{\alg}^{(r)}\bU_{\alg}^{(r)}{}^\top$ and the same is true for $\wt{\bU}_{\alg}^{(r)}$. By Definition~\ref{def:reg.alg},
	\begin{align*}
		\bzeta_{\alg}^{(r)} = \bA_{\alg}^{(r)}{}^\dagger\bb_{\alg}^{(r)} &\in \ls(\bA_{\alg}^{(r)},\bb_{\alg}^{(r)}),\quad 
		\wt{\bzeta}_{\alg}^{(r)} = \wt{\bA}_{\alg}^{(r)}{}^\dagger\wt{\bb}_{\alg}^{(r)} \in \ls(\wt{\bA}_{\alg}^{(r)},\wt{\bb}_{\alg}^{(r)}).
	\end{align*}
	We check that $(\wt{\bA}_{\alg}^{(r)},\wt{\bb}_{\alg}^{(r)})$ is a sufficiently small perturbation of $(\bA_{\alg}^{(r)},\bb_{\alg}^{(r)})$ and apply Theorem~\ref{thm:ls.pert} to the above display. Notice that $\rank(\bA_{\alg}^{(r)})=r=\rank(\wt{\bA}_{\alg}^{(r)})$ since the parameters $\btheta_{\alg}^{(r)}\in\alg(\bA,\bb)$ and $\wt{\btheta}_{\alg}^{(r)}\in\alg(\wt{\bA},\wt{\bb})$ have the same degrees-of-freedom. First, we bound
	\begin{align*}
		\frac{\|\wt{\bb}_{\alg}^{(r)}-\bb_{\alg}^{(r)}\|_{2}}{\|\bb_{\alg}^{(r)}\|_{2}} &= \frac{\|\wt{\bU}_{\alg}^{(r)}\wt{\bb}-\bU_{\alg}^{(r)}\bb\|_{2}}{\|\bb_{\alg}^{(r)}\|_{2}} \\
		&\leq \frac{\|(\wt{\bU}_{\alg}^{(r)}-\bU_{\alg}^{(r)})\wt{\bb}\|_{2}}{\|\bb_{\alg}^{(r)}\|_{2}} + \frac{\|\bU_{\alg}^{(r)}(\wt{\bb}-\bb)\|_{2}}{\|\bb_{\alg}^{(r)}\|_{2}} \\
		&\leq \|\wt{\bU}_{\alg}^{(r)}-\bU_{\alg}^{(r)}\|_{op} \cdot \frac{\|\bb\|_{2} + \|\bb\|_{2} \cdot \eps}{\|\bb_{\alg}^{(r)}\|_{2}} + \frac{\|\bb\|_{2} \cdot \eps}{\|\bb_{\alg}^{(r)}\|_{2}} \\
		&\leq C_{\alg,r} \cdot \eps \cdot \frac{\|\bb\|_{2} + \|\bb\|_{2} \cdot \eps}{\|\bb_{\alg}^{(r)}\|_{2}} + \frac{\|\bb\|_{2} \cdot \eps}{\|\bb_{\alg}^{(r)}\|_{2}} \\
		&\leq (2\ C_{\alg,r} + 1) \cdot \frac{\|\bb\|_{2}}{\|\bb_{\alg}^{(r)}\|_{2}} \cdot \eps.
	\end{align*}
	Second, we bound
	\begin{align*}
		\frac{\|\wt{\bA}_{\alg}^{(r)}-\bA_{\alg}^{(r)}\|_{op}}{\|\bA_{\alg}^{(r)}\|_{op}} &= \frac{\|\wt{\bU}_{\alg}^{(r)}\wt{\bA}\wt{\bU}_{\alg}^{(r)}-\bU_{\alg}^{(r)}\bA\bU_{\alg}^{(r)}\|_{op}}{\|\bA_{\alg}^{(r)}\|_{op}} \\
		&\leq \frac{\|(\wt{\bU}_{\alg}^{(r)}-\bU_{\alg}^{(r)})\wt{\bA}\wt{\bU}_{\alg}^{(r)}\|_{op}}{\|\bA_{\alg}^{(r)}\|_{op}} + \frac{\|\bU_{\alg}^{(r)}(\wt{\bA}\wt{\bU}_{\alg}^{(r)}-\bA\bU_{\alg}^{(r)})\|_{op}}{\|\bA_{\alg}^{(r)}\|_{op}} \\
		&\leq \frac{\|\wt{\bU}_{\alg}^{(r)}-\bU_{\alg}^{(r)}\|_{op} \cdot \|\wt{\bA}\|_{op}}{\|\bA_{\alg}^{(r)}\|_{op}} + \frac{\|\wt{\bA}(\wt{\bU}_{\alg}^{(r)}-\bU_{\alg}^{(r)})\|_{op}}{\|\bA_{\alg}^{(r)}\|_{op}} + \frac{\|(\wt{\bA}-\bA)\bU_{\alg}^{(r)}\|_{op}}{\|\bA_{\alg}^{(r)}\|_{op}} \\
		&\leq 2 \cdot \|\wt{\bU}_{\alg}^{(r)}-\bU_{\alg}^{(r)}\|_{op} \cdot \frac{\|\wt{\bA}\|_{op}}{\|\bA_{\alg}^{(r)}\|_{op}} + \frac{\|\wt{\bA}-\bA\|_{op}}{\|\bA_{\alg}^{(r)}\|_{op}} \\
		&\leq 2 \cdot C_{\alg,r} \cdot \eps \cdot \frac{\|\bA\|_{op} + \|\bA\|_{op} \cdot \eps}{\|\bA_{\alg}^{(r)}\|_{op}} + \frac{\|\bA\|_{op}}{\|\bA_{\alg}^{(r)}\|_{op}} \cdot \eps \\
		&\leq (4\ C_{\alg,r} + 1) \cdot \frac{\|\bA\|_{op}}{\|\bA_{\alg,s}\|_{op}} \cdot \eps.
	\end{align*}
	Putting the above displays together yields
	\begin{align*}
		\frac{\|\wt{\bA}_{\alg}^{(r)}-\bA_{\alg}^{(r)}\|_{op}}{\|\bA_{\alg}^{(r)}\|_{op}} \vee \frac{\|\wt{\bb}_{\alg}^{(r)}-\bb_{\alg}^{(r)}\|_{2}}{\|\bb_{\alg}^{(r)}\|_{2}} \leq \eps_{r} = (4\ C_{\alg,r} + 1) \left(\frac{\|\bA\|_{op}}{\|\bA_{\alg}^{(r)}\|_{op}} \vee \frac{\|\bb\|_{2}}{\|\bb_{\alg}^{(r)}\|_{2}}\right) \eps.
	\end{align*}
	By assumption, $2 \cdot \kappa_{2}(\bA_{\alg}^{(r)}) \cdot \eps_{r} = M_{\alg,r} \cdot \eps < 1$, thus the assumptions of Theorem~\ref{thm:ls.pert} are satisfied and one can bound
	\begin{align*}
		\frac{\|\wt{\bzeta}_{\alg}^{(r)}-\bzeta_{\alg}^{(r)}\|_{2}}{\|\bzeta_{\alg}^{(r)}\|_{2}} &\leq 5 \cdot \kappa_{2}(\bA_{\alg}^{(r)}) \cdot \eps_{r},
	\end{align*}
	which gives the claim.
\end{proof}

\begin{remark}[Examples of Stable Algorithms] \label{rem:A.b.alg.C.D}
	It is beyond the scope of this paper to provide a comprehensive study on the stability of classical regularization algorithms, but we provide below the main references in the literature.
	
	Following \cite{davis1970rot} and \cite{godunov1993gua} one finds that $\pcr(\bA,\bb)$ is a stable algorithm. In fact, consider an arbitrary matrix $\bA\in\R_{\succeq0}^{p\times p}$ with rank $1\leq r_{\bA}=\rank(\bA)\leq p$ and degree $1\leq d_{\bA}=\deg(\bA)\leq r_{\bA}$, so that the unique positive eigenvalues of $\bA$ are $\lambda_{\bA,1} > \cdots > \lambda_{\bA,d_{\bA}} > 0$. We denote $\bU_{\bA,i}\in\R^{p\times p}$ the orthogonal projection of $\R^p$ onto the $\lambda_{\bA,i}$-eigenspace of $\bA$. The dimension $m_{\bA,i}=\rank(\bU_{\bA,i})$ of each eigenspace is the multiplicity of the eigenvalue $\lambda_{\bA,i}$. We also denote $\delta_{\bA,i} = \min\{\lambda_{\bA,i-1}-\lambda_{\bA,i}, \lambda_{\bA,i}-\lambda_{\bA,i+1}\} > 0$ the minimum gap between $\lambda_{\bA,i}$ and the other eigenvalues. Now consider any perturbation $\wt{\bA}\in\R_{\succeq0}^{p\times p}$ such that $\|\wt{\bA}-\bA\|_{op} \leq \eps \|\bA\|_{op}$. Theorem~5.3 by \cite{godunov1993gua} shows that, for a corresponding orthogonal projections $\wt{\bU}_{\bA,i}\in\R^{p\times p}$ with same dimension $m_{\bA,i}=\rank(\wt{\bU}_{\bA,i})$, it holds $\|\wt{\bU}_{\bA,i}-\bU_{\bA,i}\|_{op} \lesssim \eps / \delta_{\bA,i}$. That is to say, the stability constant $C_{\pcr,m_{\bA,i}}$ in Equation~\eqref{eq:A.b.alg.C.D} is proportional to the inverse-eigengap $\delta_{\bA,i}^{-1} <+\infty$. 
	
	Following \cite{carpraux1994SotK} and \cite{kuznetsov1997Pbot} one finds that $\pls(\bA,\bb)$ is a stable algorithm. In fact, consider an arbitrary matrix $\bA\in\R_{\succeq0}^{p\times p}$ and vector $\bb\in\mR(\bA)$, together with perturbations $\wt{\bA}\in\R_{\succeq0}^{p\times p}$ and $\wt{\bb}\in\mR(\wt{\bA})$ such that $\|\wt{\bA}-\bA\|_{op} \leq \eps \|\bA\|_{op}$ and $\|\wt{\bb}-\bb\|_{2} \leq \eps \|\bb\|_{2}$. For any $1\leq s\leq r_{\bA}=\rank(\bA)$, consider $\bU_{s}\in\R^{p\times p}$ and $\wt{\bU}_{s}\in\R^{p\times p}$ to be respectively the orthogonal projections of $\R^p$ onto the Krylov spaces $\mK_{s}(\bA,\bb)$ and $\mK_{s}(\wt{\bA},\wt{\bb})$. Theorem~3.3 by~\cite{kuznetsov1997Pbot}, see our Theorem~\ref{thm:krylov.pert} and Appendix~\ref{app:aux:pert}, shows that $\|\wt{\bU}_{s}-\bU_{s}\|_{op} \lesssim \eps\ \kappa_{s}(\bA,\bb)$ and implies that the stability constant $C_{\pls,s}(\bA,\bb)$ in Equation~\eqref{eq:A.b.alg.C.D} is proportional to the condition number $\kappa_{s}(\bA,\bb) <+\infty$ of the Krylov space $\mK_{s}(\bA,\bb)$.

	Following \cite{cerone2019lin} and \cite{fosson2020spa} one finds that penalized variations of $\fss(\bA,\bb)$ are stable. In fact, consider an arbitrary matrix $\bA\in\R_{\succeq0}^{p\times p}$ and vector $\bb\in\mR(\bA)$, together with perturbations $\wt{\bA}\in\R_{\succeq0}^{p\times p}$ and $\wt{\bb}\in\mR(\wt{\bA})$ such that $\|\wt{\bA}-\bA\|_{op} \leq \eps \|\bA\|_{op}$ and $\|\wt{\bb}-\bb\|_{2} \leq \eps \|\bb\|_{2}$. For any $1\leq s\leq p$, let $J_{s}=J_{s}(\bA,\bb)$ and $\wt{J}_{s}= J_{s}(\wt{\bA},\wt{\bb})$ be respectively the best $s$-sparse active sets leading to orthogonal projections $\bU_{J_{s}}=\bI_{J_{s}}\in\R^{p\times p}$ and $\wt{\bU}_{s}=\bI_{\wt{J}_{s}}\in\R^{p\times p}$. Theorem~2 by \cite{fosson2020spa} guarantees exact recovery $\wt{J}_{s}=J_{s}$, so that $\|\wt{\bU}_{s}-\bU_{s}\|_{op}=0$ and the stability constant $C_{\fss,s}(\bA,\bb)$ in Equation~\eqref{eq:A.b.alg.C.D} is $1<+\infty$. 
\end{remark}

\section{Auxiliary Results}\label{app:aux}
Here we gather all the relevant auxiliary results and provide proofs when necessary.

\subsection{Random Vectors}\label{app:aux:rand}
\begin{lemma}\label{lem:x.y.cov}
    Let $\bx\in\R^p$ be a possibly degenerate random vector and $y\in\R$ a random variable, both centered and with finite second moments $\bSigma,\bsigma$. Then, $\bx\in\mR(\bSigma)$ almost surely and $\bsigma\in\mR(\bSigma)$.
\end{lemma}
\begin{proof}[Proof of Lemma~\ref{lem:x.y.cov}]
For the first statement, let $\bU_{\bx}\in\R^{p\times p}$ be the orthogonal projection onto $\mR(\bSigma)$, so that $\bU_{\bx}\bSigma = \bSigma = \bSigma\bU_{\bx}$ and $\mR(\bSigma)=\mR(\bU_{\bx})$. Then, with $\bU_{\bx^\bot}=\bI_{p}-\bU_{\bx}$ and the random vector $\be_{\bx} = \bU_{\bx^\bot}\bx \in\mR(\bSigma)^\bot$, we prove that $\P(\be_{\bx}=\bzero_{p})=1$. For this, compute the covariance matrix
\begin{align*}
    \bSigma_{\be_{\bx}} = \E(\be_{\bx}\be_{\bx}^\top) = \E(\bU_{\bx^\bot}\bx\bx^\top\bU_{\bx^\bot}) = \bU_{\bx^\bot} \bSigma \bU_{\bx^\bot} = \bzero_{p\times p},
\end{align*}
thus the vector $\be_{\bx}$ is almost surely equal to its expectation $\E(\be_{\bx}) = \E(\bU_{\bx^\bot}\bx) = \bzero_{p}$.

The second statement is a consequence of the first, since $\bx=\bU_{\bx}\bx$ almost surely implies
\begin{align*}
    \bsigma = \E(\bx y) = \E(\bU_{\bx}\bx y) = \bU_{\bx}\bsigma \in \mR(\bSigma).
\end{align*}
\end{proof}

\begin{lemma}\label{lem:x.y.ls}
    Let $(\bx,y)\in\R^p\times\R$ be a centered random pair for which the squared-loss $\ell_{\bx,y}(\bbeta)=\E(y-\bx^\top\bbeta)^2$ is well-defined for all $\bbeta\in\R^p$. The set of least-squares solutions $\ls(\bx,y,\R^p) = \argmin_{\bbeta\in\R^p} \ell_{\bx,y}(\bbeta)$ is $\{\bbeta\in\R^p:\bSigma\bbeta=\bsigma\}$ and the minimum-$L^2$-norm solution is $\bbeta_{\ls} = \bSigma^\dagger\bsigma$.
\end{lemma}
\begin{proof}[Proof of Lemma~\ref{lem:x.y.ls}]
    The squared-loss function is
    \begin{align*}
        \bbeta\mapsto\ell_{\bx,y}(\bbeta) = \E(y-\bx^\top\bbeta)^2\in\R
    \end{align*}
    over all $\bbeta\in\R^p$. The gradient of the squared-loss function is
    \begin{align*}
        \bbeta\mapsto\nabla_{\bbeta}\ell_{\bx,y}(\bbeta) = 2\E(\bx\bx^\top\bbeta) - 2\E(\bx y) = 2\cdot \left\{\bSigma\bbeta - \bsigma\right\}\in\R^p,
    \end{align*}
    and its Hessian
    \begin{align*}
        \bbeta\mapsto\nabla_{\bbeta}\nabla_{\bbeta}^\top\ell_{\bx,y}(\bbeta) = 2\E(\bx\bx^\top) = 2\cdot \bSigma \in\R_{\succeq0}^{p\times p}
    \end{align*}
    is a positive semi-definite matrix, so the squared-loss function $\bbeta\mapsto\ell_{\bx,y}(\bbeta)$ is convex everywhere, although possibly not strictly convex. As a consequence, the set of least-squares solutions $\ls(\bx,y,\R^p) = \argmin_{\bbeta\in\R^p} \ell_{\bx,y}(\bbeta)$ coincides with the set $\{\bbeta\in\R^p:\nabla_{\bbeta}\ell_{\bx,y}(\bbeta)=\bzero_{p}\}$ of critical points. By the above displays, any critical point $\bbeta\in\R^p$ satisfies the normal equation $\bSigma\bbeta=\bsigma$, which admits at least one solution (the set of critical points is non-empty) since Lemma~\ref{lem:x.y.cov} shows that the covariance vector $\bsigma$ always belongs to the range of the covariance matrix $\bSigma$. In particular, the minimum-$L^2$-norm solution is $\bbeta_{\ls} = \bSigma^\dagger\bsigma$.
\end{proof}

\subsection{Empirical Processes}\label{app:aux:emp}

\begin{lemma}\label{lem:x.complexity}
    For some integer $p\geq1$, let $\bxi\in\R^p$ be a possibly degenerate random vector and $(\bxi_{i})_{i=1,\ldots,n}$ be i.i.d. copies of $\bxi$ with finite
    \begin{align*}
        r_{\bxi} = \rank(\bSigma_{\bxi}),\quad \rho_{\bxi} = \frac{\E(\|\bxi\|_{2}^2)}{\|\bSigma_{\bxi}\|_{op}},\quad \rho_{\bxi,n} = \frac{\E(\max_{1\leq i\leq n} \|\bxi_{i}\|_{2}^q)^{2/q}}{\|\bSigma_{\bxi}\|_{op}},
    \end{align*}
    for some $q>4$. Then, it follows that $1\leq \rho_{\bxi} \leq r_{\bxi} \leq p$ and:
    \begin{enumerate}[label=(\roman*),itemsep=0.25em,topsep=0.25em]
        \item if $\|\bxi\|_{2}^q \leq L_{\bxi}^{q/2} \|\bSigma_{\bxi}\|_{op}^{q/2}\ \rho_{\bxi}^{q/2}$ almost surely, then $\rho_{\bxi} \leq \rho_{\bxi,n} \leq L_{\bxi} \rho_{\bxi}$; \label{lem:x.complexity.bound}
        \item if $\|\bxi\|_{2}^q$ admits MGF $M_{\bxi}(\cdot)$ and $\log\left(M_{\bxi}(t)\right) \leq t\cdot L_{\bxi}^{q/2} \E(\|\bxi\|_{2}^2)^{q/2}$ for all $0<t<t_{\bxi}$ , then $\rho_{\bxi} \leq \rho_{\bxi,n} \leq L_{\bxi} \rho_{\bxi} \cdot (\log n)^{2/q}$; \label{lem:x.complexity.mgf}
        \item if $\E(\|\bxi\|_{2}^q) \leq L_{\bxi}^{q/2} \E(\|\bxi\|_{2}^2)^{q/2}$, then $\rho_{\bxi} \leq \rho_{\bxi,n} \leq L_{\bxi} \rho_{\bxi}\cdot n^{2/q}$. \label{lem:x.complexity.heavy}
    \end{enumerate}
\end{lemma}
\begin{proof}[Proof of Lemma~\ref{lem:x.complexity}]
    The first inequality follows from
    \begin{align*}
        1\leq \rho_{\bxi} = \frac{\Tr(\bSigma_{\bxi})}{\lambda_{max}(\bSigma_{\bxi})} \leq \frac{\lambda_{max}(\bSigma_{\bxi}) \rank(\bSigma_{\bxi})}{\lambda_{max}(\bSigma_{\bxi})} = r_{\bxi} \leq p.
    \end{align*}
    The lower bound $\rho_{\bxi} \leq \rho_{\bxi,n}$ holds in general since
    \begin{align*}
    	\E(\max_{1\leq i\leq n} \|\bxi_{i}\|_{2}^q)^{2/q} \geq \E(\|\bxi\|_{2}^q)^{2/q} \geq \E(\|\bxi\|_{2}^2)^{2/2} = \E(\|\bxi\|_{2}^2),
    \end{align*}
    due to the monotonicity of $L^q$-norms and $q>4>2$. We now deal with the upper bounds. We prove \ref{lem:x.complexity.bound} by direct computation with
    \begin{align*}
        \rho_{\bxi,n} &= \E\left(\max_{1\leq i\leq n} \frac{\|\bxi_{i}\|_{2}^q}{\|\bSigma_{\bxi}\|_{op}^{q/2}} \right)^{\frac{2}{q}} \leq \left(\frac{L_{\bxi}^{q/2} \|\bSigma_{\bxi}\|_{op}^{q/2}\ \rho_{\bxi}^{q/2}}{\|\bSigma_{\bxi}\|_{op}^{q/2}} \right)^{\frac{2}{q}} = L_{\bxi} \rho_{\bxi}.
    \end{align*}
    We prove \ref{lem:x.complexity.mgf} by direct computation via Jensen's inequality and the change of variable $s = t\cdot \|\bSigma_{\bxi}\|_{op}^{-1}$, so that
    \begin{align*}
        \rho_{\bxi,n} &= \E\left(\max_{1\leq i\leq n} \frac{\|\bxi_{i}\|_{2}^q}{\|\bSigma_{\bxi}\|_{op}^{q/2}} \right)^{\frac{2}{q}} \displaybreak[0] \\
        &\leq \left\{ \inf_{0<s<s_{\bxi}} \frac{\log\left(\E\left(\exp\left(s \cdot \max_{1\leq i\leq n} \frac{\|\bxi_{i}\|_{2}^q}{\|\bSigma_{\bxi}\|_{op}^{q/2}} \right) \right) \right)}{s} \right\}^{\frac{2}{q}} \displaybreak[0] \\
        &\leq \left\{ \inf_{0<s<s_{\bxi}} \frac{\log\left(\E\left(\sum_{i=1}^n \exp\left(s \cdot \frac{\|\bxi_{i}\|_{2}^q}{\|\bSigma_{\bxi}\|_{op}^{q/2}} \right) \right) \right)}{s} \right\}^{\frac{2}{q}} \displaybreak[0] \\
        &= \left\{ \inf_{0<s<s_{\bxi}} \frac{\log n + \log\left(\E\left(\exp\left(s \cdot \frac{\|\bxi\|_{2}^q}{\|\bSigma_{\bxi}\|_{op}^{q/2}} \right) \right) \right)}{s} \right\}^{\frac{2}{q}} \displaybreak[0] \\
        &\leq \left\{ \inf_{0<s<s_{\bxi}} \frac{\log n \cdot s \cdot L_{\bxi}^{q/2} \E\left(\frac{\|\bxi\|_{2}^2}{\|\bSigma_{\bxi}\|_{op}}\right)^{q/2}}{s} \right\}^{\frac{2}{q}} \displaybreak[0] \\
        &= L_{\bxi} \E\left(\frac{\|\bxi\|_{2}^2}{\|\bSigma_{\bxi}\|_{op}} \right) \cdot (\log n)^{\frac{2}{q}}.
    \end{align*}
    We prove \ref{lem:x.complexity.heavy} by direct computation via Jensen's inequality
    \begin{align*}
        \rho_{\bxi,n} &=  \E\left(\max_{1\leq i\leq n} \frac{\|\bxi_{i}\|_{2}^q}{\|\bSigma_{\bxi}\|_{op}^{q/2}} \right)^{\frac{2}{q}} \displaybreak[0] \\
        &\leq \E\left(\sum_{i=1}^n \frac{\|\bxi_{i}\|_{2}^q}{\|\bSigma_{\bxi}\|_{op}^{q/2}} \right)^{\frac{2}{q}} \displaybreak[0] \\
        &= n^{\frac{2}{q}} \cdot \E\left(\frac{\|\bxi\|_{2}^{q}}{\|\bSigma_{\bxi}\|_{op}^{q/2}} \right)^{\frac{2}{q}} \displaybreak[0] \\
        &\leq L_{\bxi} \E\left(\frac{\|\bxi\|_{2}^2}{\|\bSigma_{\bxi}\|_{op}} \right) \cdot n^{\frac{2}{q}}.
    \end{align*}
\end{proof}

\begin{lemma}[Theorem~6 by \cite{Jirak2025}]\label{lem:mom.mat.heavy.new}
    For some integer $p\geq1$, let $\bxi\in\R^p$ be a possibly degenerate random vector and $(\bxi_{i})_{i=1,\ldots,n}$ be i.i.d. copies of $\bxi$ with finite
    \begin{align*}
        \rho_{\bxi} = \frac{\E(\|\bxi\|_{2}^2)}{\|\bSigma_{\bxi}\|_{op}}, \quad \rho_{\bxi,n} = \frac{\E(\max_{1\leq i\leq n} \|\bxi_{i}\|_{2}^q)^{2/q}}{\|\bSigma_{\bxi}\|_{op}}, \quad L_{\bxi} = \sup_{\bv\in\partial\B_{2}^p(1)} \frac{\E(|\bxi^\top\bv|^q)^{2/q}}{\E(|\bxi^\top\bv|^2)},
    \end{align*}
    for some $q>4$. Then, for some absolute constants $c>0$ and $C>0$ with $n > \rho_{\bxi} / c$,
    \begin{align*}
        \E\big(\|\wh{\bSigma}_{\bxi} - \bSigma_{\bxi} \|_{op}^2 \big)^{1/2} &\leq C \|\bSigma_{\bxi}\|_{op}\ \delta_{\bxi,n},\quad \delta_{\bxi,n} =\sqrt{\frac{\rho_{\bxi}}{n}} + \frac{\rho_{\bxi,n}}{n}.
    \end{align*}
\end{lemma}

\begin{corollary}\label{cor:mom.mat.heavy.new}
    Let $\{p_{n}\}_{n\geq1}$ be a sequence of integers $p_{n}\geq1$ and $\{\bxi^{(n)}\}_{n\geq1}$ a sequence of possibly degenerate random vectors $\bxi^{(n)}\in\R^{p_{n}}$ under the assumptions of Lemma~\ref{lem:mom.mat.heavy.new} for some $q>4$. It holds
    \begin{align*}
        \sqrt{\frac{\rho_{\bxi^{(n)}}}{n}} + \frac{\rho_{\bxi^{(n)},n}}{n} \leq \sqrt{\frac{\rho_{\bxi^{(n)}}}{n}} \cdot \left\{1 + L_{\bxi^{(n)}} \sqrt{\frac{\rho_{\bxi^{(n)}}}{n^{\frac{q-4}{q}}}} \right\}.
    \end{align*}
\end{corollary}
\begin{proof}[Proof of Corollary~\ref{cor:mom.mat.heavy.new}]
    Under the assumptions of Lemma~\ref{lem:mom.mat.heavy.new} the random vectors $\bxi^{(n)}$ satisfy Lemma~\ref{lem:x.complexity}~\ref{lem:x.complexity.heavy}. Therefore, with $q>4$ one can bound
    \begin{align*}
        \sqrt{\frac{\rho_{\bxi^{(n)}}}{n}} + \frac{\rho_{\bxi^{(n)},n}}{n} \leq \sqrt{\frac{\rho_{\bxi^{(n)}}}{n}} + \frac{L_{\bxi^{(n)}} \rho_{\bxi^{(n)}}\cdot n^{\frac{2}{q}}}{n} = \sqrt{\frac{\rho_{\bxi^{(n)}}}{n}} \cdot \left\{1 + L_{\bxi^{(n)}} \sqrt{\frac{\rho_{\bxi^{(n)}}}{n^{\frac{q-4}{q}}}} \right\},
    \end{align*}
    which is the claim.    
\end{proof}

\begin{lemma}\label{lem:mult.vec.heavy}
    For some integer $p\geq1$, let $\bxi\in\R^p$ be a possibly degenerate random vector, $\zeta\in\R$ a random variable and $(\bxi_{i})_{i=1,\ldots,n}$, $(\zeta_{i})_{i=1,\ldots,n}$ i.i.d. copies of $\bxi,\zeta$ with finite
    \begin{align*}
        r_{\bxi} = \rank(\bSigma_{\bxi}),\quad \rho_{\bxi} = \frac{\E(\|\bxi\|_{2}^2)}{\|\bSigma_{\bxi}\|_{op}},\quad L_{\bxi} = \frac{\E(\|\bxi\|_{2}^4)^{\frac{1}{4}}}{\E(\|\bxi\|_{2}^2)^{\frac{1}{2}}},\quad L_{\zeta} = \frac{\E(\zeta^4)^{\frac{1}{4}}}{\sigma_{\zeta}}.
    \end{align*}
    Then, 
    \begin{align*}
        \E\big(\|\wh{\bsigma}_{\bxi,\zeta} - \bsigma_{\bxi,\zeta} \|_{2} \big) \leq \|\bSigma_{\bxi}\|_{op}^{\frac{1}{2}}\ \sigma_{\zeta}\ \delta_{\bxi,\zeta},\quad \delta_{\bxi,\zeta} = L_{\bxi} L_{\zeta} \sqrt{\frac{\rho_{\bxi}}{n}}.
    \end{align*}
\end{lemma}
\begin{proof}[Proof of Lemma~\ref{lem:mult.vec.heavy}]
    We start with Jensen's inequality
    \begin{align*}
        \E\left(\| \wh{\bsigma}_{\bxi,\zeta} - \bsigma_{\bxi,\zeta} \|_{2} \right) \leq \E\left(\| \wh{\bsigma}_{\bxi,\zeta} - \bsigma_{\bxi,\zeta} \|_{2}^2 \right)^{\frac{1}{2}}
    \end{align*}
    and compute
    \begin{align*}
        \E\left(\left\| \wh{\bsigma}_{\bxi,\zeta} - \bsigma_{\bxi,\zeta} \right\|_{2}^2 \right) &= \E\left(\left\| \frac{1}{n} \sum_{i=1}^n \bxi_{i} \zeta_{i} - \E(\bxi \zeta) \right\|_{2}^2 \right) \displaybreak[0] \\
        &= \E\left( \sum_{j=1}^p \left\{\frac{1}{n} \sum_{i=1}^n \bxi_{i} \zeta_{i} - \E(\bxi \zeta)\right\}_{j}^2 \right) \displaybreak[0] \\
        &= \E\left( \sum_{j=1}^p \left\{\frac{1}{n} \sum_{i=1}^n \xi_{i,j} \zeta_{i} - \E(\xi_{j} \zeta)\right\}^2 \right) \displaybreak[0] \\
        &= \frac{1}{n^2} \sum_{j=1}^p \sum_{i=1}^n \sum_{i'=1}^n \E\left( \{\xi_{i,j} \zeta_{i} - \E(\xi_{j} \zeta)\} \cdot \{\xi_{i',j} \zeta_{i'} - \E(\xi_{j} \zeta)\} \right) \displaybreak[0] \\
        &= \frac{1}{n^2} \sum_{j=1}^p \sum_{i=1}^n \E\left( \{\xi_{i,j} \zeta_{i} - \E(\xi_{j} \zeta)\}^2 \right) \displaybreak[0] \\
        &= \frac{1}{n} \sum_{j=1}^p \E\left( \{\xi_{j} \zeta - \E(\xi_{j} \zeta)\}^2 \right).
    \end{align*}
    We bound the variances with the corresponding second moments and get
    \begin{align*}
        \E\left(\left\| \wh{\bsigma}_{\bxi,\zeta} - \bsigma_{\bxi,\zeta} \right\|_{2}^2 \right) &\leq \frac{1}{n} \E\left( \zeta^2 \|\bxi\|_{2}^2 \right)
        \leq \frac{1}{n} \E\left( \zeta^4 \right)^{\frac{1}{2}} \E\left(\|\bxi\|_{2}^4 \right)^{\frac{1}{2}}
        = \frac{L_{\zeta}^2 \sigma_{\zeta}^2 L_{\bxi}^2 \rho_{\bxi} \|\bSigma_{\bxi}\|_{op}}{n}.
    \end{align*}
\end{proof}

\subsection{Numerical Perturbation Theory}\label{app:aux:pert}
In this section we provide classical and novel results which are relevant to the theory of deterministic perturbations of least-squares problems.

\begin{lemma}\label{lem:ls.krylov}
    Let $\bA\in\R_{\succeq0}^{p\times p}$ and $\bb\in\mR(\bA)$ arbitrary. With $\bzeta_{\ls} = \bA^\dagger\bb \in \ls(\bA,\bb)$, we have $\bzeta_{\ls} \in \mK_{d_{\bA}}(\bA,\bb)$ with $d_{\bA}=\deg(\bA)$. 
\end{lemma}
\begin{proof}[Proof of Lemma~\ref{lem:ls.krylov}]
    The Cayley-Hamilton theorem, see Theorem~8.1 in~\cite{zhang1997quat}, guarantees that $p_{\bA}(\bA)=\bzero_{p\times p}$ for $p_{\bA}$ the minimum polynomial of $\bA$, having degree $d_{\bA}=\deg(\bA)\leq p$. Since $\bA^\top=\bA$, Theorem~3 in~\cite{decell1965appl} guarantees that the generalized inverse can be represented as $\bA^\dagger = (\bA^\top)^{\delta} \sum_{k=1}^{d_{\bA}} c_{k-1} \bA^{k-1} = (\sum_{k=1}^{d_{\bA}} c_{k-1} \bA^{k-1}) (\bA^\top)^{\delta}$, with $\delta=0$ if $\bA^\dagger=\bA^{-1}$ and $\delta=1$ otherwise. Notice that $\mR(\bA^\top) = \mR(\bA\bA^\dagger)$ and $\bb\in\mR(\bA)$ implies $\bA\bA^\dagger\bb=\bb$, we find
    \begin{align*}
        \bzeta_{\ls} = \bA^\dagger\bb \in\spa\{\bA\bA^\dagger\bb,\ldots,\bA^{r-1} \bA\bA^\dagger\bb\} = \spa\{\bb,\ldots,\bA^{d_{\bA}-1} \bb\},
    \end{align*}
    as required.
\end{proof}

We recall and adapt the main results from the numerical perturbation theory of Krylov spaces initiated by \cite{carpraux1994SotK} and later developed by \cite{kuznetsov1997Pbot}. It is worth noticing that the whole theory has been developed in terms of perturbation bounds with respect to the Frobenius norm $\|\cdot\|_{F}$ instead of the operator norm $\|\cdot\|_{op}$. However, the choice of metric on the space of square matrices is arbitrary, as long as it is unitary invariant. To see this, notice that all the proofs by \cite{kuznetsov1997Pbot} exploit the properties of orthogonal matrices presented in their Section~4, which are already expressed in terms of operator norms. We provide below the immediate generalization of the main objects from the classical theory to the case of perturbations in operator norm. Let $\bU\in\R^{p\times p}$ be some matrix such that $\bU^\top\bU=\bI_{p}$, its spectrum $\Sp(\bU)=\{\lambda_{1}(\bU),\ldots,\lambda_{p}(\bU)\}$ is a subset of the unit circle. That is, one can write $\lambda_{j}(\bU) = e^{i\omega_{j}(\bU)}$ for some $\omega_{j}(\bU)\in\R$ and order these values as $-\pi\leq\omega_{1}(\bU)\leq\ldots\leq\omega_{p}(\bU)<\pi$. The following is an adaptation of Definition~2.1 by \cite{kuznetsov1997Pbot} to suit our needs, we denote
\begin{align*}
    \rho(\bU) = \max\left\{|\omega_{j}| : e^{i\omega_{j}}\in\Sp(\bU), 0<\omega_{j}<\pi \right\},
\end{align*}
where the bound on the interval $(0,\pi)$ is justified by the fact that the eigenvalues of $\bU$ consist of complex-conjugate pairs. As for the original definition, we remove the real eigenvalues $\{\pm1\}$. Let $\mB$ and $\wt{\mB}$ be two $m$-dimensional subspaces of $\R^p$, and let $\bB\in\R^{p\times m}$ and $\wt{\bB}\in\R^{p\times m}$ some orthonormal basis of $\mB$ and $\wt{\mB}$ respectively. Then, there exists some matrix $\bU\in\R^{p\times p}$ with $\bU^\top\bU=\bI_{p}$ such that $\wt{\bB} = \bU\bB$. With the above, we define
\begin{align}\label{eq:dist.subsp}
    d(\bB,\wt{\bB}) = \inf_{\bU} \rho(\bU),\quad d(\mB,\wt{\mB}) = \inf_{\bB,\wt{\bB}} d(\bB,\wt{\bB}),
\end{align}
where the first infimum is taken over all possible orthonormal matrices such that $\wt{\bB} = \bU\bB$ and the second infimum is taken over all possible orthonormal bases $\bB,\wt{\bB}$ of $\mB,\wt{\mB}$. This is the same as Definition~2.2 by \cite{kuznetsov1997Pbot}, the only difference being the definition of spectral radius $\rho(\bU)$ in the previous display. It is important noticing that this distance is equivalent to the principal-angle defined in Equation~\eqref{eq:subsp.dist.angle}.

\begin{lemma}\label{lem:dist.perm}
    For some $1\leq m\leq p$, let $\mB\subseteq\R^p$ and $\wt{\mB}\subseteq\R^p$ be two orthogonal $m$-dimensional subspaces. Then, $d(\mB,\wt{\mB}) = \pi/2$.
\end{lemma}
\begin{proof}[Proof of Lemma~\ref{lem:dist.perm}]
    Let $\bB=(\bv_{1}|\cdots|\bv_{m})\in\R^{p\times m}$ and $\wt{\bB}=(\wt{\bv}_{1}|\cdots|\wt{\bv}_{m})\in\R^{p\times m}$ be any two orthonormal basis of $\mB$ and $\wt{\mB}$, respectively. By orthogonality, the vectors $\bv_{i}$ and $\wt{\bv}_{j}$ are linearly independent for all $i,j=1,\ldots,m$. Thus, there exist vectors $\bw_{1},\ldots,\bw_{p-2m}$ such that $(\bv_{1}|\cdots|\bv_{m}|\wt{\bv}_{1}|\cdots|\wt{\bv}_{m}|\bw_{1}|\cdots|\bw_{p-2m})\in\R^{p\times p}$ is an orthonormal basis of $\R^p$. Now, among all possible orthogonal transformations $\bU\in\R^{p\times p}$ such that $\wt{\bB}=\bU\bB$, the ones achieving the smallest $\rho(\bU)$ in Equation~\eqref{eq:dist.subsp} are those such that $\bU\bw_{i}=\pm\bw_{i}$ for all $i=1,\ldots,p-2m$. Any such matrix is then a pairwise permutation matrix that swaps positions of $\bv_{i}$ and $\wt{\bv}_{i}$ in the original basis of $\R^p$. The spectrum $\Sp(\bU)$ of such a matrix consists of the eigenvalue $\pm 1$ corresponding to the fixed points and the complex unit root $i = e^{i \pi/2}$. From Equation~\eqref{eq:dist.subsp}, we get
    \begin{align*}
        d\left(\mB,\wt{\mB} \right)  
        = \min_{\wt{\bB},\bB}\ \min_{\bU:\wt{\bB}=\bU\bB} \rho(\bU)
        = \min_{\wt{\bB},\bB}\ \min_{\bU:\wt{\bB}=\bU\bB} \max\left\{\left|\frac{\pi }{2}\right| : e^{i \pi /2}\in\Sp(\bU)\right\}
        = \frac{\pi}{2}.
    \end{align*}
\end{proof}

\begin{lemma}[Lemma~1 by~\cite{carpraux1994SotK}]
    Let $\bB,\wt{\bB}$ be orthonormal bases of $\mB,\wt{\mB}$ such that $\wt{\bB} = (\bI_{p} + \bDelta + O(\|\bDelta\|_{op}^2))\bB$ for some $\|\bDelta\|_{op}\ll1$. Then, at the first order in $\|\bDelta\|_{op}\ll1$, one has
    \begin{align*}
        d(\bB,\wt{\bB}) = \inf_{\bDelta} \|\bDelta\|_{op},
    \end{align*}
    where the infimum is taken over all matrices $\bDelta\in\R^{p\times p}$ such that $\|\bDelta\|_{op}\ll1$, $\bDelta^\top=-\bDelta$ and $\wt{\bB} = (\bI_{p} + \bDelta + O(\|\bDelta\|_{op}^2))\bB$.
\end{lemma}

We always consider $\bA\in\R^{p\times p}$ some symmetric and positive semi-definite matrix and $\bb\in\mR(\bA)$ some vector, together with a Krylov space $\mK_{m}=\mK_{m}(\bA,\bb)$ of full dimension, that is $m=\dim(\mK_{m}(\bA,\bb))$ for some $1\leq m\leq p$. We are interested in all perturbations $\wt{\bA}=\bA+\wt{\Delta\bA}\in\R^{p\times p}$ some symmetric and positive semi-definite matrix and $\wt{\bb}=\bb+\wt{\Delta\bb}\in\mR(\wt{\bA})$ some vector such that the perturbed Krylov space $\wt{\mK}_{m}=\mK_{m}(\wt{\bA},\wt{\bb})$ has full dimension, that is $m=\dim(\mK_{m}(\wt{\bA},\wt{\bb}))$. We denote $\bK_{m}\in\R^{p\times m}$ and $\wt{\bK}_{m}\in\R^{p\times m}$ any natural orthonormal bases of $\mK_{m}(\bA,\bb)$ and $\mK_{m}(\wt{\bA},\wt{\bb})$, in the sense of Definition~2 by \cite{carpraux1994SotK}. Notice that these bases are unique up to their signs and can be computed, for example, with the Arnoldi iteration devised by \cite{Arnoldi1951}. When $\wt{\bA}=\bA$ and $\wt{\bb}=\bb$, one can always find some orthonormal matrix $\bU\in\R^{p\times p}$ having spectrum $\Sp(\bU)=\{\pm 1\}$ and such that $\wt{\bK}_{m}=\bU\bK_{m}$, so that $d(\bK_{m},\wt{\bK}_{m})=0$ even though $\bK_{m}\neq\wt{\bK}_{m}$ (the identity holds up to the signs of the columns). For arbitrary $\wt{\bA}$ and $\wt{\bb}$ we denote the size of the perturbation by $\Delta(\wt{\bA},\wt{\bb}) = \frac{\|\wt{\bA}-\bA\|_{op}}{\|\bA\|_{op}} \vee \frac{\|\wt{\bb}-\bb\|_{2}}{\|\bb\|_{2}}$ and define
\begin{align}\label{eq:krylov.cond}
    \kappa_{b}(\bK_{m}) &= \inf_{\eps>0}\ \sup_{(\wt{\bA},\wt{\bb})\ :\ \Delta(\wt{\bA},\wt{\bb})\leq\eps} \frac{d(\bK_{m},\wt{\bK}_{m})}{\Delta(\wt{\bA},\wt{\bb})},\quad \kappa_{2}(\mK_{m}(\bA,\bb)) = \min_{\bK_{m}} \kappa_{b}(\bK_{m}).
\end{align}
Since Krylov spaces are invariant under orthonormal transformations, there exist $\bV_{m} = (\bK_{m}|\bK_{m^\bot})\in\R^{p\times p}$ and $\wt{\bV}_{m} = (\wt{\bK}_{m}|\wt{\bK}_{m^\bot})\in\R^{p\times p}$ both orthonormal bases of $\R^p$ such that $\bG_{m}=\bV_{m}^\top\bK_{m}$ and $\wt{\bG}_{m}=\wt{\bV}_{m}^\top\wt{\bK}_{m}$ are the natural orthonormal bases of $\mK_{m}(\bH_{m},\be_{1})$ and $\mK_{m}(\wt{\bH}_{m},\be_{1})$ with $\be_{1}\in\R^p$ the first vector of the canonical basis and both $\bH_{m}=\bV_{m}^\top\bA\bV_{m}\in\R^{p\times p}$, $\wt{\bH}_{m}=\wt{\bV}_{m}^\top\wt{\bA}\wt{\bV}_{m}\in\R^{p\times p}$ tridiagonal symmetric (Hessenberg symmetric) matrices. In particular, one can always reduce the problem to perturbations of Krylov spaces having same vector $\wt{\bb}=\bb$ and tridiagonal symmetric matrices. This gives the equivalent definition
\begin{align*}
    \kappa_{b}(\bG_{m}) &= \inf_{\eps>0}\ \sup_{\wt{\bH} :\ \Delta(\wt{\bH})\leq\eps} \frac{d(\bG_{m},\wt{\bG}_{m})}{\Delta(\wt{\bH})},\quad \kappa_{2}(\mK_{m}(\bH,\be_{1})) = \min_{\bG_{m}} \kappa_{b}(\bG_{m}),
\end{align*}
corresponding to Definition~3 by ~\cite{carpraux1994SotK}, for which the reduction to the Hessenberg case (tridiagonal symmetric for us) is given in their Theorem~1.

\begin{theorem}[Theorem~3.1 by~\cite{kuznetsov1997Pbot}]\label{thm:krylov.cauchy}
    Let $\bH\in\R^{p\times p}$ be a symmetric tridiagonal matrix and $m=\dim(\mK_{m}(\bH,\be_{1}))$ for some $1\leq m\leq p$. Assume $\bH(t)\in\R^{p\times p}$ is a continuously differentiable matrix function such that
    \begin{align*}
        \bH(0) = \bH,\quad \left\|\frac{d\bH(t)}{dt}\right\|_{op} \leq \nu \|\bH(t)\|_{op},
    \end{align*}
    for some $0<\nu<1$. Let $\bV_{m}(t)\in\R^{p\times p}$ be the orthogonal matrix defined as the solution of the Cauchy-problem in Equations~(3.1)~-~(3.3) by \cite{kuznetsov1997Pbot}. For all $\bG_{m}\in\R^{p\times m}$ natural orthonormal bases of $\mK_{m}(\bH,\be_{1})$ and
    \begin{align*}
        0 \leq t \leq \frac{1}{16 \nu \kappa_{b}(\bG_{m}) (\kappa_{b}(\bG_{m})+1)},
    \end{align*}
    one has 
    \begin{align*}
        \|\bV_{m}^\top(t)\bG_{m} - \bG_{m}\|_{op} \leq 2 \kappa_{b}(\bG_{m}) \nu t.
    \end{align*}
\end{theorem}
\begin{proof}[Proof of Theorem~\ref{thm:krylov.cauchy}]
    We slightly adapt the proof of Theorem~3.1 by~\cite{kuznetsov1997Pbot}. Their Equation~(3.7) becomes for us
    \begin{align*}
        \delta_{m}(t) = 2\sin\left(\frac{\int_{0}^\top \|\bX_{m}(\xi)\|_{op}d\xi}{2}\right),\quad
        \rho_{m}(t) = \left\{2\delta_{m}(t) + t\nu [1+\delta_{m}(t)] \right\} \|\bH\|_{op}.
    \end{align*}
    Using that $\|\bH(t)-\bH\|_{op} = \|\int_{0}^\top \frac{d}{d\xi}\bX(\xi)d\xi\|_{op} \leq t\nu\|\bH\|_{op}$, one recovers their bound $\|\bZ_{m}(t) - \bH\|_{op} \leq \rho_{m}(t)$. The remainder of the proof proceeds as in the original reference.
\end{proof}

\begin{theorem}[Theorem~3.3 by~\cite{kuznetsov1997Pbot}]\label{thm:krylov.pert}
    Let $\bA\in\R^{p\times p}$ be some symmetric and positive semi-definite matrix and $\bb\in\mR(\bA)$ some vector. Let $\wt{\bA}=\bA+\wt{\Delta\bA}\in\R^{p\times p}$ be some symmetric and positive semi-definite matrix and $\wt{\bb}=\bb+\wt{\Delta\bb}\in\mR(\wt{\bA})$ some vector. Let $\bK_{m}\in\R^{p\times m}$ be any natural orthonormal basis of $\mK_{m}=\mK_{m}(\bA,\bb)$ for some $1\leq m\leq p$. Assume that $m = \dim(\mK_{m}(\bA,\bb)) = \dim(\mK_{m}(\wt{\bA},\wt{\bb}))$ and 
    \begin{align*}
        \frac{\|\wt{\Delta\bA}\|_{op}}{\|\bA\|_{op}} \leq \eps,\quad \frac{\|\wt{\Delta\bb}\|_{2}}{\|\bb\|_{2}} \leq \eps,\quad 0 \leq \eps \leq \frac{1}{64 \kappa_{b}(\bK_{m}) (\kappa_{b}(\bK_{m}) + 1)}.
    \end{align*}
    Then, there exists a natural orthonormal basis $\wt{\bK}_{m}\in\R^{p\times m}$ of $\wt{\mK}_{m}=\mK_{m}(\wt{\bA},\wt{\bb})$ such that
    \begin{align*}
        \|\wt{\bK}_{m}-\bK_{m}\|_{op} \leq 11  \kappa_{b}(\bK_{m}) \eps.
    \end{align*}
\end{theorem}
\begin{proof}[Proof of Theorem~\ref{thm:krylov.pert}]
    We slightly improve the proof of Theorem~3.3 by \cite{kuznetsov1997Pbot}. Define the continuously differentiable matrix function $\bA(t)=\bA+(\wt{\bA}-\bA)t\in\R^{p\times p}$ so that $\bA(0)=\bA$ and $\bA(1)=\wt{\bA}$ and $\|\frac{d\bA(t)}{dt}\|_{op} = \|\wt{\Delta\bA}\|_{op} \leq \eps \|\bA\|_{op}$. One can find $\bV\in\R^{p\times p}$ and $\wt{\bV}\in\R^{p\times p}$ orthonormal matrices such that $\|\bV-\wt{\bV}\|_{2}\leq \sqrt{2}\eps$ and define Hessenberg matrices $\bH = \bV^\top\bA(0)\bV$ and $\wt{\bH} = \wt{\bV}^\top\bA(1)\wt{\bV}$. One can check that $\|\wt{\bH}-\bH\|_{op} \leq (2\sqrt{2}+1)\|\bH\|_{op}\eps \leq 4\|\bH\|_{op}\eps$. The assumptions of Theorem~\ref{thm:krylov.cauchy} hold for $t=1$ and $\nu=4\eps$, thus, with suitable bases $\bG_{m}$ and $\wt{\bG}_{m}$ of $\mK_{m}(\bH,\be_{1})$ and $\mK_{m}(\wt{\bH},\be_{1})$ one has $\|\wt{\bG}_{m}-\bG_{m}\|_{op}\leq 8 \kappa_{b}(\bG_{m}) \eps$. One concludes the proof for the bases $\bK_{m}$ and $\wt{\bK}_{m}$ by writing
    \begin{align*}
        \|\wt{\bK}_{m}-\bK_{m}\|_{op} = \|\wt{\bV}\wt{\bG}_{m}-\bV\bG_{m}\|_{op} \leq \|\bV-\wt{\bV}\|_{2} + \|\wt{\bG}_{m}-\bG_{m}\|_{op} \leq 11 \kappa_{b}(\bK_{m}) \eps.
    \end{align*}
\end{proof}

The next result is a combination of the proof of Theorem~3.3 by~\cite{kuznetsov1997Pbot}, together with one of its corollaries.
\begin{corollary}[Corollary~2 by~\cite{kuznetsov1997Pbot}]\label{cor:krylov.pert}
    Under assumptions of Theorem~\ref{thm:krylov.pert}, let $\bA_{m}=\bK_{m}^\top\bA\bK_{m}\in\R^{m\times m}$, $\bb_{m}=\bK_{m}^\top\bb\in\R^m$ be the projected matrix and vector relative to the orthonormal basis $\bK_{m}\in\R^{p\times m}$ and $\wt{\bA}_{m}=\wt{\bK}_{m}^\top\wt{\bA}\wt{\bK}_{m}\in\R^{m\times m}$, $\wt{\bb}_{m}=\wt{\bK}_{m}^\top\wt{\bb}\in\R^m$ be the projected matrix and vector relative to the orthonormal basis $\wt{\bK}_{m}\in\R^{p\times m}$. Then,
    \begin{align*}
        \frac{\|\wt{\bb}_{m}-\bb_{m}\|_{2}}{\|\bb_{m}\|_{2}} \leq 2 \eps,\quad
        \frac{\|\wt{\bA}_{m}-\bA_{m}\|_{op}}{\|\bA_{m}\|_{op}} \leq 24 \cdot \kappa_{b}(\bK_{m}) \|\bA\|_{op} \|\bA_{m}\|_{op}^{-1} \eps.
    \end{align*}
\end{corollary}

\begin{lemma}[Lemma~3.1 by~\cite{kuznetsov1997Pbot}]\label{lem:krylov.cond}
    Under assumptions of Theorem~\ref{thm:krylov.pert}, $m=\dim(\mK(\bA,\bb))$ implies $\wt{m}=\dim(\mK(\wt{\bA},\wt{\bb}))\geq m$ and, for all $m < s \leq p$,
    \begin{align*}
        \kappa_{2}(\mK_{s}(\bA,\bb)) = +\infty,\quad \kappa_{2}(\mK_{s}(\wt{\bA},\wt{\bb})) \geq 
        \frac{1}{\eps \left\{14 + 56 \kappa_{2}(\mK_{m}(\bA,\bb)) \right\}}.
    \end{align*}
\end{lemma}

\section{Proofs}\label{app:proof}
Here we provide all the proofs for the results in the main sections.

\subsection{Proofs for Section~\ref{sec:lm}}\label{app:proof:lm}

\begin{proof}[Proof of Lemma~\ref{lem:ls.rel.pop}]
    For the first statement, we notice that the conditions determining the relevant subspace $\mB_{y}$ in Definition~\ref{def:rel.sub} are equivalent to
    \begin{align*}
        \bSigma = \bU_{y}\bSigma\bU_{y} + \bU_{y^\bot}\bSigma\bU_{y^\bot},\quad \bsigma = \bU_{y}\bsigma.
    \end{align*}
    This implies that $\mB_{y}$ is the unique $\bSigma$-envelope of $\spa\{\bsigma\}\subseteq\mR(\bSigma)$ in the sense of Definition~2.1 by \cite{Cook2010env}. That is to say, $\mB_{y}$ is the intersection of all reducing subspaces for $\bSigma$ that contain $\spa\{\bsigma\}$.
    
    For the second statement, we recall that $\bbeta_{\ls}(\bSigma,\bsigma) = \bSigma^\dagger\bsigma$ and $\bbeta_{\ls}(\bSigma_{y},\bsigma_{y}) = \bSigma_{y}^\dagger\bsigma_{y}$. From the definition of relevant subspace in Definition~\ref{def:rel.sub} it follows
    \begin{align*}
        \bsigma &= \E(\bx y) = \E(\bx_{y} y) \oplus \E(\bx_{y^\bot} y) = \E(\bx_{y} y) = \bsigma_{y}, \\
        \bSigma &= \E(\bx\bx^\top) = \E(\bx_{y}\oplus\bx_{y^\bot})(\bx_{y}\oplus\bx_{y^\bot})^\top = \E(\bx_{y}\bx_{y}^\top) \oplus \E(\bx_{y^\bot}\bx_{y^\bot}^\top) = \bSigma_{y} \oplus \bSigma_{\bx_{y^\bot}}.
    \end{align*}
    The range of the matrix $\bSigma_{y}$ is the relevant subspace $\mR(\bSigma_{y})=\mB_{y}$, whereas range of the matrix $\bSigma_{\bx_{y^\bot}}$ is the irrelevant subspace $\mR(\bSigma_{\bx_{y^\bot}})=\mB_{y}^\bot$. Thus, the same holds for the generalized inverse $\bSigma^\dagger = (\bSigma_{y} \oplus \bSigma_{\bx_{y^\bot}})^\dagger = \bSigma_{y}^\dagger \oplus \bSigma_{\bx_{y^\bot}}^\dagger$. One last computation yields
    \begin{align*}
        \bbeta_{\ls}(\bSigma,\bsigma) = \bSigma^\dagger\bsigma = \bSigma_{y}^\dagger\bsigma_{y} \oplus \bSigma_{\bx_{y^\bot}}^\dagger\bsigma_{y} = \bSigma_{y}^\dagger\bsigma_{y} = \bbeta_{\ls}(\bSigma_{y},\bsigma_{y}),
    \end{align*}
    which is the claim.
\end{proof}

\begin{proof}[Proof of Lemma~\ref{lem:phi.B}]
Notice that $\bbeta_{\ls}(\bSigma_{y},\bsigma_{y}) = \bbeta_{\mB_{d_{y}}}$ and also $\bbeta_{\mB_{d_{y}}} - \bbeta_{s} = \bbeta_{s^\bot}$. This means
\begin{align*}
\eps_{s} = \frac{\E(\bx_{y}^\top\bbeta_{\mB_{d_{y}}} - \bx_{y}^\top\bbeta_{s})^2}{\bbeta_{\mB_{d_{y}}}^\top\bbeta_{\mB_{d_{y}}} \cdot \lambda_{1}(\bSigma)} =  \frac{\bbeta_{s^\bot}^\top\bSigma_{\bx_{\mB_{s}^\bot}}\bbeta_{s^\bot}}{\bbeta_{\mB_{d_{y}}}^\top\bbeta_{\mB_{d_{y}}} \cdot \lambda_{1}(\bSigma)} \leq \frac{\lambda_{s+1}(\bSigma_{y})}{\lambda_{1}(\bSigma_{y})} \cdot \frac{\|\bbeta_{s^\bot}\|_{2}^2}{\|\bbeta_{\mB_{d_{y}}}\|_{2}^2},
\end{align*}
which gives the claim.
\end{proof}

\begin{lemma}\label{lem:x.rel.y.eps}
	Let $(\bx,y)\in\R^p\times\R$ satisfy Assumption~\ref{ass:x.y.lm.2nd} and $\mB_{s}\subseteq\mB_{y}$ be any span of eigenspaces for $1\leq s\leq d_{y}$ as in Lemma~\ref{lem:phi.B}. Then, it holds
	\begin{align*}
		\frac{\|\bSigma_{y}-\bSigma_{s}\|_{op}}{\|\bSigma_{s}\|_{op}} \vee \frac{\|\bsigma_{y}-\bsigma_{s}\|_{2}}{\|\bsigma_{s}\|_{2}} \leq \frac{1}{\kappa_{2}(\bSigma_{s+1})} \vee  \frac{\kappa_{2}(\bSigma_{s})}{\kappa_{2}(\bSigma_{s+1})} \frac{\|\bbeta_{s^\bot}\|_{2}}{\|\bbeta_{s}\|_{2}},
	\end{align*}
	with the convention that $\kappa_{2}(\bSigma_{d_{y}+1})=+\infty$.
\end{lemma}

\begin{proof}[Proof of Lemma~\ref{lem:x.rel.y.eps}]
Recall that $\bSigma_{y}=\sum_{k=1}^{d_{y}} \lambda_{\bx_{y},k} \bU_{\bx_{y},k}$ and $\bsigma_{y}=\bSigma_{y}\bbeta_{\ls}$. We compute
\begin{align*}
\frac{\|\bSigma_{y}-\bSigma_{s}\|_{op}}{\|\bSigma_{s}\|_{op}} &= \frac{\|\sum_{k=1}^{d_{y}} \lambda_{\bx_{y},k} \bU_{\bx_{y},k} - \sum_{k=1}^{s} \lambda_{\bx_{y},k} \bU_{\bx_{y},k}\|_{op}}{\|\sum_{k=1}^{s} \lambda_{\bx_{y},k} \bU_{\bx_{y},k}\|_{op}} = \frac{\lambda_{s+1}(\bSigma_{y})}{\lambda_{1}(\bSigma_{y})}, \\ 
\frac{\|\bsigma_{y}-\bsigma_{s}\|_{2}}{\|\bsigma_{s}\|_{2}} &= \frac{\|\sum_{k=1}^{d_{y}} \lambda_{\bx_{y},k} \bU_{\bx_{y},k}\bbeta_{\ls}-\sum_{k=1}^{s} \lambda_{\bx_{y},k} \bU_{\bx_{y},k}\bbeta_{\ls}\|_{2}}{\|\sum_{k=1}^{s} \lambda_{\bx_{y},k} \bU_{\bx_{y},k}\bbeta_{\ls}\|_{2}} \leq \frac{\lambda_{s+1}(\bSigma_{y})}{\lambda_{s}(\bSigma_{y})} \cdot \frac{\|\bU_{s^\bot}\bbeta_{\ls}\|_{2}}{\|\bU_{s}\bbeta_{\ls}\|_{2}},
\end{align*}
which is the claim.
\end{proof}

\begin{proof} [Proof of Theorem~\ref{thm:x.y.reg.alg.pop}]
	Following Definition~\ref{def:reg.alg}, the population algorithms compute parameters
	\begin{align*}
		\btheta_{\alg}^* = \big(\mB_{\alg}^*,\ \bU_{\alg}^*,\ r_{\alg}^*,\ \bbeta_{\alg}^*\big) \in \alg(\bSigma_{s},\bsigma_{s}),\quad 
		\btheta_{\alg} = \big(\mB_{\alg},\ \bU_{\alg},\ r_{\alg},\ \bbeta_{\alg}\big) \in \alg(\bSigma,\bsigma).
	\end{align*}
	By assumption, the following are true:
	\begin{enumerate}[label=(\roman*),itemsep=0.25em,topsep=0.25em]
		\item the population algorithm $\alg(\bSigma_{s},\bsigma_{s})$ is parsimonious; \label{ass:x.y.reg.alg.pop.parsim}
		\item the population algorithm $\alg(\bSigma_{s},\bsigma_{s})$ is stable with constant $C_{\alg}^*\geq1$ as in Equation~\eqref{eq:A.b.alg.C.D} in Section~\ref{sec:rls:alg.pert} and
		\begin{align*}
			M_{\alg}^* = 2 \cdot \kappa_{2}(\bU_{\alg}^*\bSigma_{s}\bU_{\alg}^*) \cdot \{4\ C_{\alg}^* + 1\} \cdot \left\{\frac{\|\bSigma_{s}\|_{op}}{\|\bU_{\alg}^*\bSigma_{s}\bU_{\alg}^*\|_{op}} \vee \frac{\|\bsigma_{s}\|_{2}}{\|\bU_{\alg}^*\bsigma_{s}\|_{2}}\right\},
		\end{align*}
		corresponding to Equation~\eqref{eq:A.b.alg.M} in Section~\ref{sec:rls:alg.pert};
		\label{ass:x.y.reg.alg.pop.stab}
		\item the population algorithms are compatible with $r_{\alg}^*=r_{\alg}$; \label{ass:x.y.reg.alg.pop.dof}
		\item the population algorithm $\alg(\bSigma,\bsigma)$ is adaptive; \label{ass:x.y.reg.alg.pop.adap}
		\item the size of the perturbation satisfies $\eps^* < 1/M_{\alg}^*$.
		\label{ass:x.y.reg.alg.pop.pert}
	\end{enumerate}
	Condition~\ref{ass:x.y.reg.alg.pop.adap} implies $\alg(\bSigma,\bsigma)=\alg(\bSigma_{y},\bsigma_{y})$.
	We want to invoke Theorem~\ref{thm:alg.pert} with $\bA=\bSigma_{s}$, $\bb=\bsigma_{s}$, $\wt{\bA}=\bSigma_{y}$, $\wt{\bb}=\bsigma_{y}$. We now check the assumptions. First, the compatibility of the perturbation required by Equation~\eqref{eq:A.b.alg.pert} holds by Condition~\ref{ass:x.y.reg.alg.pop.dof}. This implies $(\wt{\bA},\wt{\bb})\in\Delta_{\alg}(\bA,\bb)$. Second, the population algorithm with oracle knowledge is stable with constants $C_{\alg}^*$ and $D_{\alg}^*$ by Condition~\ref{ass:x.y.reg.alg.pop.stab}. Third, the size of the perturbation is sufficiently small by Condition~\ref{ass:x.y.reg.alg.pop.pert}. We can thus apply Theorem~\ref{thm:alg.pert} and obtain the required bounds.
\end{proof}

\begin{corollary}\label{cor:x.y.reg.alg.pop}
	Under the assumptions of Theorem~\ref{thm:x.y.reg.alg.pop}, assume that for some $r \leq r_{\alg^*} = \dim(\mB_{\alg}^*)$ there exist compatible parameters $\bbeta_{\alg}^{(r)}(\bx_{s},y)$ computed by $\alg(\bSigma_{s},\bsigma_{s})$ and  $\bbeta_{\alg}^{(r)}(\bx,y)$ computed by $\alg(\bSigma,\bsigma)$. The early-stopping population error on the least-squares solution is
	\begin{align*}
		\frac{\|\bbeta_{\alg}^{(r)}(\bx,y)-\bbeta_{s}\|_{2}}{\|\bbeta_{s}\|_{2}} &\leq \sqrt{r_{\alg^*} - r} + \frac{5}{2}\ M_{\alg^*}\ \eps^*.
	\end{align*}
\end{corollary}

\begin{proof}[Proof of Corollary~\ref{cor:x.y.reg.alg.pop}]
    We have
    \begin{align*}
        \frac{\|\bbeta_{\alg}^{(r)}(\bx,y)-\bbeta_{\alg}(\bx_{s},y)\|_{2}}{\|\bbeta_{\alg}(\bx_{s},y)\|_{2}} &\leq \frac{\|\bbeta_{\alg}^{(r)}(\bx,y)-\bbeta_{\alg}^{(r)}(\bx_{s},y)\|_{2}}{\|\bbeta_{\alg}(\bx_{s},y)\|_{2}} + \frac{\|\bbeta_{\alg}^{(r)}(\bx_{s},y) - \bbeta_{\alg}(\bx_{s},y)\|_{2}}{\|\bbeta_{\alg}(\bx_{s},y)\|_{2}}.
    \end{align*}  
    We now bound the two terms in the above display. By Definition~\ref{def:reg.alg} and the equivalent representations in Lemma~\ref{lem:reg.alg.def}, the population solution satisfies $\bbeta_{\alg}^{(r)} = \bU_{\alg}^{(r)} \bbeta_{\alg}$. That is to say, for some orthonormal basis $\{\bu_{\alg,1},\ldots,\bu_{\alg,r_{\alg}}\}$ one has $\bbeta_{\alg} = \sum_{\ell=1}^{r_{\alg}} c_{\ell} \bu_{\alg,\ell}$ and $\bbeta_{\alg}^{(r)} = \sum_{\ell=1}^r c_{\ell} \bu_{\alg,\ell}$ with the same coefficients $c_{\ell}$, $\ell=1,\ldots,s$. Since $\|\bbeta_{\alg}\|_{2}^2 = \sum_{\ell=1}^{r_{\alg}} c_{\ell}^2$, we can bound
    \begin{align*}
        \|\bbeta_{\alg}^{(r)} - \bbeta_{\alg}\|_{2} = \left(\sum_{\ell=s+1}^{r_{\alg}} c_{\ell}^2\right)^{\frac{1}{2}} \leq \left(\max_{\ell=s+1,\ldots,r_{\alg}} c_{\ell}^2\right)^{\frac{1}{2}} \sqrt{r_{\alg}-r} \leq \|\bbeta_{\alg}\|_{2}\ \sqrt{r_{\alg}-r}.
    \end{align*}
    An inspection of the proof of Theorem~\ref{thm:x.y.reg.alg.pop} shows that its induced bounds hold for all population solutions $\bbeta_{\alg}^{(r)}(\bx,y)$ and $\bbeta_{\alg}^{(r)}(\bx_{s},y)$ that are compatible with $r\leq r_{\alg}^*$, since the quantities appearing in the assumptions are largest when $r=r_{\alg}$. That is to say, we can bound
    \begin{align*}
        \|\bbeta_{\alg}^{(r)}(\bx,y)-\bbeta_{\alg}^{(r)}(\bx_{s},y)\|_{2} \leq \frac{5}{2}\ \|\bbeta_{\alg}^{(r)}(\bx_{s},y)\|_{2}\ M_{\alg}^*\ \eps^* \leq \frac{5}{2}\ \|\bbeta_{\alg}(\bx_{s},y)\|_{2}\ M_{\alg}^*\ \eps^*.
    \end{align*}
    The claim follows by combining the above displays.
\end{proof}

\begin{lemma}\label{lem:x.y.reg.alg.sam.event}
    Under the assumptions of Theorem~\ref{thm:x.y.reg.alg.sam}, the event
    \begin{align*}
        \Omega_{\bx,y}(\nu_{n}) = \left\{\frac{\|\wh{\bSigma}-\bSigma\|_{op}}{\|\bSigma\|_{op}} \vee \frac{\|\wh{\bsigma}-\bsigma\|_{2}}{\|\bsigma\|_{2}} \leq K\ \nu_{n}^{-1}\ \delta_{n} \right\},
    \end{align*}
    has probability at least $1-2\nu_{n}$.
\end{lemma}
\begin{proof}[Proof of Lemma~\ref{lem:x.y.reg.alg.sam.event}]
    We denote $\bA=\bSigma$, $\bb=\bsigma$, $\wh{\bA}=\wh{\bSigma}$, $\wh{\bb}=\wh{\bsigma}$. The size of the perturbation on the covariance vector is
    \begin{align*}
        \left\|\wh{\Delta\bb}\right\|_{2} &= \left\|\wh{\bsigma}-\bsigma\right\|_{2} = \left\|\frac{1}{n}\sum_{i=1}^n \bx_{i} y_{i} - \E(\bx y)\right\|_{2}.
    \end{align*}
    We can apply Lemma~\ref{lem:mult.vec.heavy} to the above, since the required moments are bounded by Assumption~\ref{ass:x.y.lm.4th}. With our definitions and assumptions, we find     
    \begin{align*}
        \E\left(\frac{\|\wh{\Delta\bb}\|_{2}}{\|\bb\|_{2}} \right) 
        &= \frac{\E(\|\wh{\bsigma}-\bsigma\|_{2})}{\|\bsigma\|_{2}}
        \leq \frac{L_{y} L_{\bx}\ \sigma_{y} \|\bSigma\|_{op}^\frac{1}{2}\ \delta_{n}}{\|\bsigma\|_{2}} 
        \leq K\ \delta_{n},
    \end{align*}
    and an application of Markov's inequality yields
    \begin{align*}
        \P\left(\frac{\|\wh{\Delta\bb}\|_{2}}{\|\bb\|_{2}} > K\ \nu_{n}^{-1}\ \delta_{n} \right) < \nu_{n}.
    \end{align*}
    The size of the perturbation on the covariance matrix can be written as
    \begin{align*}
        \left\|\wh{\Delta\bA}\right\|_{op} &= \left\|\wh{\bSigma}-\bSigma\right\|_{op} = \left\|\frac{1}{n}\sum_{i=1}^n \bx_{i}\bx_{i}^\top - \E(\bx\bx^\top)\right\|_{op}.
    \end{align*}
    We can apply Lemma~\ref{lem:mom.mat.heavy.new} to the above since the required moments are bounded by Assumption~\ref{ass:x.y.lm.4th} and the sample size is sufficiently large by assumption. We find
    \begin{align*}
        \E\left(\frac{\|\wh{\Delta\bA}\|_{op}}{\|\bA\|_{op}} \right) 
        &= \frac{\E(\|\wh{\bSigma}-\bSigma\|_{op})}{\|\bSigma\|_{op}}
        \leq \frac{C \|\bSigma\|_{op}\ \delta_{n}}{\|\bSigma\|_{op}}
        \leq K\ \delta_{n},
    \end{align*}
    and an application of Markov's inequality yields
    \begin{align*}
        \P\left(\frac{\|\wh{\Delta\bA}\|_{op}}{\|\bA\|_{op}} > K\ \nu_{n}^{-1}\ \delta_{n} \right) < \nu_{n}.
    \end{align*}
    The intersection of the complements of the above events has probability at least $1-2\nu_{n}$ and, conditionally on this event, 
    \begin{align*}
        \frac{\|\wh{\Delta\bA}\|_{op}}{\|\bA\|_{op}} \vee \frac{\|\wh{\Delta\bb}\|_{2}}{\|\bb\|_{2}} \leq K\ \nu_{n}^{-1}\ \delta_{n},
    \end{align*}
    which gives the claim.
\end{proof}

\begin{proof}[Proof of Theorem~\ref{thm:x.y.reg.alg.sam}]
	Following Definition~\ref{def:reg.alg}, the population and sample algorithms compute parameters
	\begin{align*}
		\btheta_{\alg} = \big(\mB_{\alg},\ \bU_{\alg},\ r_{\alg},\ \bbeta_{\alg}\big)\in\alg(\bSigma,\bsigma), \quad
		\wh{\btheta}_{\alg} = \big(\wh{\mB}_{\alg},\ \wh{\bU}_{\alg},\ \wh{r}_{\alg},\ \wh{\bbeta}_{\alg}\big)\in\alg(\wh{\bSigma},\wh{\bsigma}).
	\end{align*}
	By assumption, it holds:
	\begin{enumerate}[label=(\roman*),itemsep=0.25em,topsep=0.25em]
		\item the population algorithm $\alg(\bSigma,\bsigma)$ is stable with constant $C_{\alg}\geq1$, as in Equation~\eqref{eq:A.b.alg.C.D} in Section~\ref{sec:rls:alg.pert} and
		\begin{align*}
			M_{\alg} = 2 \cdot \kappa_{2}(\bU_{\alg}\bSigma\bU_{\alg}) \cdot \{4\ C_{\alg} + 1\} \cdot \left\{\frac{\|\bSigma\|_{op}}{\|\bU_{\alg}\bSigma\bU_{\alg}\|_{op}} \vee \frac{\|\bsigma\|_{2}}{\|\bU_{\alg}\bsigma\|_{2}}\right\},
		\end{align*}
		corresponding to Equation~\eqref{eq:A.b.alg.M} in Section~\ref{sec:rls:alg.pert}; \label{ass:x.y.reg.alg.sam.stab}
		\item the sample and population algorithms are compatible with $r_{\alg} = \wh{r}_{\alg}$; \label{ass:x.y.reg.alg.sam.dof}
		\item with $\delta_{n}$ the complexity in Equation~\eqref{eq:x.y.delta.star.n}, some absolute constant $C\geq1$,
		\begin{align*}
			K &= C \vee \frac{L_{y} L_{\bx}\ \sigma_{y} \|\bSigma\|_{op}^\frac{1}{2}}{\|\bsigma\|_{2}},
		\end{align*}
		it holds $\delta_{n} \xrightarrow{n\to\infty} 0$ and $K M_{\alg} \delta_{n} < 1/2$. \label{ass:x.y.reg.alg.sam.n.large}
	\end{enumerate}
	The assumptions of Lemma~\ref{lem:x.y.reg.alg.sam.event} hold and, with $\wh{\eps}$ the size of the sample perturbation and $\eps_{n} = K\ \nu_{n}^{-1}\ \delta_{n}$, the event $\Omega_{\bx,y}(\nu_{n}) = \{\wh{\eps} \leq \eps_{n}\}$ has probability at least $1-2\nu_{n}$. On this event, we check that we can invoke Theorem~\ref{thm:alg.pert} with $\bA=\bSigma$, $\bb=\bsigma$ and $\wt{\bA}=\wh{\bSigma}$, $\wt{\bb}=\wh{\bsigma}$. First, the compatibility of the perturbation required by Equation~\eqref{eq:A.b.alg.pert} holds by Condition~\ref{ass:x.y.reg.alg.sam.dof}. This implies $(\wt{\bA},\wt{\bb})\in\Delta_{\alg}(\bA,\bb)$. Second, the population algorithm is stable with constants $C_{\alg}$, $D_{\alg}$ and $M_{\alg}$ by Condition~\ref{ass:x.y.reg.alg.sam.stab}. Third, the size of the perturbation is sufficiently small with
    \begin{align*}
        M_{\alg} \cdot \wh{\eps} \leq M_{\alg} \cdot \eps_{n} < \frac{M_{\alg} K}{M} = 1.
    \end{align*}
    Thus, on the event $\Omega_{\bx,y}(\nu_{n})$, it holds the claim.
\end{proof}

\begin{corollary}\label{cor:x.y.reg.alg.sam}
	Under the assumptions of Theorem~\ref{thm:x.y.reg.alg.sam}, assume that for some $r \leq r_{\alg}=\dim(\mB_{\alg})$ there exist compatible parameters $\bbeta_{\alg}^{(r)}$ computed by $\alg(\bSigma,\bsigma)$ and $\wh{\bbeta}_{\alg}^{(r)}$ computed by $\alg(\wh{\bSigma},\wh{\bsigma})$. On the same event of probability at least $1-2\nu_{n}$, the early-stopping sample error on the least-squares solution is
	\begin{align*}
		\frac{\|\wh{\bbeta}_{\alg}^{(r)} - \bbeta_{\alg}\|_{2}}{\|\bbeta_{\alg}\|_{2}} &\leq \sqrt{r_{\alg}-r} + \frac{5}{2}\ K M_{\alg}\ \bigg\{\sqrt{\frac{\rho_{\bx}}{n\nu_{n}^2}} + \frac{\rho_{\bx,n}}{n\nu_{n}}\bigg\},
	\end{align*} 
\end{corollary}

\begin{proof}[Proof of Corollary~\ref{cor:x.y.reg.alg.sam}]
    We find 
    \begin{align*}
        \frac{\|\wh{\bbeta}_{\alg}^{(r)} - \bbeta_{\alg}\|_{2}}{\|\bbeta_{\alg}\|_{2}} &\leq \frac{\|\wh{\bbeta}_{\alg}^{(r)} - \bbeta_{\alg}^{(r)}\|_{2}}{\|\bbeta_{\alg}\|_{2}} + \frac{\|\bbeta_{\alg}^{(r)} - \bbeta_{\alg}\|_{2}}{\|\bbeta_{\alg}\|_{2}}.
    \end{align*}  
    With Definition~\ref{def:reg.alg} and the equivalent representations in Lemma~\ref{lem:reg.alg.def}, the same argument in the proof of Corollary~\ref{cor:x.y.reg.alg.pop} yields
    \begin{align*}
        \|\bbeta_{\alg}^{(r)} - \bbeta_{\alg}\|_{2} &\leq \|\bbeta_{\alg}\|_{2}\ \sqrt{r_{\alg}-r}.
    \end{align*}
    An inspection of the proof of Theorem~\ref{thm:x.y.reg.alg.sam} shows that its induced bounds hold for all sample parameters $\wh{\bbeta}_{\alg}^{(r)}$ and population parameters $\bbeta_{\alg}^{(r)}$ as long as $r\leq r_{\alg}$, since the quantities appearing in the assumptions are largest when $r=r_{\alg}$. That is to say, we can bound (on the same event)
    \begin{align*}
        \|\wh{\bbeta}_{\alg}^{(r)} - \bbeta_{\alg}^{(r)}\|_{2} &\leq \frac{5}{2}\ K M_{\alg}\ \bigg\{\sqrt{\frac{\rho_{\bx}}{n\nu_{n}^2}} + \frac{\rho_{\bx,n}}{n\nu_{n}}\bigg\}.
    \end{align*}
    Putting all the above displays together we get the claim.
\end{proof}

Theorem~\ref{thm:x.y.reg.alg.pop.sam} is a consequence of the following.
\begin{theorem} \label{thm:x.y.reg.alg.pop.sam.stop}
	Under the assumptions of Theorem~\ref{thm:x.y.reg.alg.pop} and Theorem~\ref{thm:x.y.reg.alg.sam}, assume that for some $r \leq r_{\alg}=\dim(\mB_{\alg})$ there exist compatible parameters there exist compatible parameters $\bbeta_{\alg}^{(r)}(\bx_{s},y)$ computed by $\alg(\bSigma_{s},\bsigma_{s})$ and $\wh{\bbeta}_{\alg}^{(r)}$ computed by $\alg(\wh{\bSigma},\wh{\bsigma})$. On the same event of probability at least $1-2\nu_{n}$, the early-stopping estimation error is
	\begin{align*}
		\frac{\|\wh{\bbeta}_{\alg}^{(r)} - \bbeta_{s}\|_{2}}{\|\bbeta_{s}\|_{2}} &\leq \sqrt{r_{\alg}-s} + \frac{5}{2}\ M_{\alg}^*\ \eps^* + \frac{35}{4}\ K M_{\alg}\ \bigg\{\sqrt{\frac{\rho_{\bx}}{n\nu_{n}^2}} + \frac{\rho_{\bx,n}}{n\nu_{n}}\bigg\}.
	\end{align*} 
\end{theorem}

\begin{proof}[Proof of Theorem~\ref{thm:x.y.reg.alg.pop.sam}]
    Recall that $\bbeta_{\alg}^*=\bbeta_{s}$. We find
    \begin{align*}
        \frac{\|\wh{\bbeta}_{\alg}^{(r)} - \bbeta_{\alg}^*\|_{2}}{\|\bbeta_{\alg}^*\|_{2}} &\leq \frac{\|\wh{\bbeta}_{\alg}^{(r)} - \bbeta_{\alg}^{(r)}(\bx,y)\|_{2}}{\|\bbeta_{\alg}^*\|_{2}} + \frac{\|\bbeta_{\alg}^{(r)}(\bx,y) - \bbeta_{\alg}^{(r)}(\bx_{s},y)\|_{2}}{\|\bbeta_{\alg}^*\|_{2}} + \frac{\|\bbeta_{\alg}^{(r)}(\bx_{s},y) - \bbeta_{\alg}^*\|_{2}}{\|\bbeta_{\alg}^*\|_{2}}.
    \end{align*}  
    With Definition~\ref{def:reg.alg} and the equivalent representations in Lemma~\ref{lem:reg.alg.def}, the same argument in the proof of Corollary~\ref{cor:x.y.reg.alg.pop} yields both
    \begin{align*}
        \frac{\|\bbeta_{\alg}^{(r)}(\bx_{s},y) - \bbeta_{\alg}^*\|_{2}}{\|\bbeta_{\alg}^*\|_{2}} &\leq \sqrt{r_{\alg}-s}, \\
        \frac{\|\bbeta_{\alg}^{(r)}(\bx,y) - \bbeta_{\alg}^{(r)}(\bx_{s},y)\|_{2}}{\|\bbeta_{\alg}^*\|_{2}} &\leq \frac{5}{2}\  M_{\alg}^*\ \eps^*.
    \end{align*}
    An inspection of the proof of Theorem~\ref{thm:x.y.reg.alg.sam} shows that its induced bounds hold for all sample parameters $\wh{\bbeta}_{\alg}^{(r)}$ and population parameters $\bbeta_{\alg}^{(r)}$ as long as $r\leq r_{\alg}$, since the quantities appearing in the assumptions are largest when $r=r_{\alg}$. That is to say, we can bound (on the same event)
    \begin{align*}
        \|\wh{\bbeta}_{\alg}^{(r)} - \bbeta_{\alg}^{(r)}(\bx,y)\|_{2} &\leq \|\bbeta_{\alg}^{(r)}(\bx,y)\|_{2}\ \frac{5}{2}\ K M_{\alg}\ \bigg\{\sqrt{\frac{\rho_{\bx}}{n\nu_{n}^2}} + \frac{\rho_{\bx,n}}{n\nu_{n}}\bigg\} \\
        &\leq \left\{\|\bbeta_{\alg}^{(r)}(\bx_{s},y)\|_{2} + \|\bbeta_{\alg}^{(r)}(\bx,y) - \bbeta_{\alg}^{(r)}(\bx_{s},y)\|_{2}\right\}\ \frac{5}{2}\ K M_{\alg}\ \bigg\{\sqrt{\frac{\rho_{\bx}}{n\nu_{n}^2}} + \frac{\rho_{\bx,n}}{n\nu_{n}}\bigg\} \\
        &\leq \|\bbeta_{\alg}^*\|_{2} \left\{1 + \frac{5}{2}\ M_{\alg}^*\ \eps^*\right\}\ \frac{5}{2}\ K M_{\alg}\ \bigg\{\sqrt{\frac{\rho_{\bx}}{n\nu_{n}^2}} + \frac{\rho_{\bx,n}}{n\nu_{n}}\bigg\}.
    \end{align*}
    Putting together all the above displays we get the claim since $1+(5/2)M_{\alg}^*\ \eps^*\leq 7/2$.
\end{proof}

\clearpage
\bibliography{full.bib}       

\end{document}